\documentclass[matprg]{svjour3}
\usepackage{graphics,graphicx,epsf,subfigure,epstopdf}
\usepackage{amsmath,amsxtra,amsfonts,amscd,amssymb,bm}
\usepackage{algorithm}
\usepackage{algpseudocode}
\usepackage[top=2.5cm,bottom=2.5cm,right=2.5cm,left=2.5cm]{geometry}
\usepackage[mathscr]{eucal}
\usepackage{color}
\numberwithin{equation}{section}

\usepackage{hyperref}
\DeclareMathOperator*{\argmax}{arg\,max}
\DeclareMathOperator*{\argmin}{arg\,min}

\newcommand{\N}{{\mathbb{N}}}
\newcommand{\tsum}{\textstyle\sum}

\newcommand{\beq}{\begin{equation}}
\newcommand{\eeq}{\end{equation}}
\newcommand{\beqa}{\begin{eqnarray}}
\newcommand{\eeqa}{\end{eqnarray}}
\newcommand{\beqas}{\begin{eqnarray*}}
\newcommand{\eeqas}{\end{eqnarray*}}

\newtheorem{assumption}{Assumption}

\newcommand{\nn}{\nonumber}

\def\cP{{\cal P}}

\def\vgap{\vspace*{.1in}}

\global\long\def\R{\mathbb{R}}%
\global\long\def\D{\mathbb{D}}%
\global\long\def\E{\mathbb{E}}%
\global\long\def\P{\mathbb{P}}%
\global\long\def\Q{\mathbb{Q}}%

\newcommand{\cL}{\mathcal{L}}

\newcommand{\cF}{\mathcal{F}}
\newcommand{\cO}{\mathcal{O}}

\title{Nearly Optimal $L_p$ Risk Minimization
\thanks{This research was partially supported by Office of Naval Research grant N000142412654 and N000142412621, and Defense Advanced
Research Projects Agency grant HR00112490446.
The paper was first released at https://arxiv.org/abs/2407.15368 on 07/22/2024.}}

\author{Zhichao Jia \quad Guanghui Lan \quad Zhe Zhang}

\institute{Zhichao Jia \\ H. Milton Stewart School of Industrial and Systems Engineering, Georgia Institute of Technology, Atlanta, GA, 30332. \\ Email: {\tt zjia75@gatech.edu} \\ \\ Guanghui Lan \\ H. Milton Stewart School of Industrial and Systems Engineering, Georgia Institute of Technology, Atlanta, GA, 30332. \\ Email: {\tt george.lan@isye.gatech.edu} \\ \\ Zhe Zhang \\ School of Industrial Engineering, Purdue University, West Lafayette, IN, 47907. \\ Email: {\tt zhan5111@purdue.edu}}

\date{}

\begin{document}

\maketitle

\begin{abstract}
    Convex risk measures play a foundational role in the area of stochastic optimization. However, in contrast to risk neutral models, their applications are still limited due to the lack of efficient solution methods. In particular, the mean $L_p$ semi-deviation is a classic risk minimization model, but its solution is highly challenging due to the composition of concave-convex functions and the lack of uniform Lipschitz continuity. In this paper, we discuss some progresses on the design of efficient algorithms for $L_p$ risk minimization, including a novel lifting reformulation to handle the concave-convex composition, and a new stochastic approximation method to handle the non-Lipschitz continuity. We establish an upper bound on the sample complexity associated with this approach and show that this bound is not improvable for $L_p$ risk minimization in general through the construction of a nearly matching lower complexity bound.
\end{abstract}
\vgap

{\noindent\keywordname $L_p$ risk measure $\cdot$ risk averse optimization $\cdot$ stochastic approximation $\cdot$ lower complexity bounds}
\vgap

{\noindent\subclassname 90C15 $\cdot$ 90C25 $\cdot$ 90C30 $\cdot$ 90C47 $\cdot$ 91B05 $\cdot$ 91G70 $\cdot$ 62L20 $\cdot$ 68Q25}

\section{Introduction}

Risk averse optimization (RAO)~\cite[Section 6]{shapiro2021lectures} plays an important role in the field of stochastic optimization. Unlike risk neutral models, in RAO the decision maker needs to take into account the risk preferences while minimizing the expected cost. Specifically, we consider the RAO problem in the following mean-risk form:
\begin{align}
    \min_{x\in X} \ \left\{\rho[Z_x]:=\E[Z_x]+c\D[Z_x] \right\},
    \label{RAproblem}
\end{align}
where $x$ is the decision variable, $X\subseteq\R^n$ is the constraint set, $Z_x$ is a random variable that depends on $x$ and describes the uncertain outcome, $\E$ is the expectation, $\D$ measures the uncertainty of $Z_x$, and the coefficient $c\geq 0$ balances the expectation and risk preference. We call $\rho[\cdot]$ a risk measure.

Convex risk measure~\cite[Definition 6.4]{shapiro2021lectures} is an important subclass of $\rho[\cdot]$ that has attracted broad interests for its potential applications in risk management~\cite{fertis2012robust}, portfolio optimization~\cite{luthi2005convex}, and more recently, machine learning~\cite{gotoh2014interaction} and reinforcement learning~\cite{coache2024reinforcement}. Well-known convex mean-risk models include the average value-at-risk (AVaR)~\cite{rockafellar2000optimization,shapiro2021lectures}, which is also called conditional value-at-risk (CVaR), at level $\alpha>0$ with $c=1$ and $\D[Z_x]=\inf_{t\in\R}\E[-Z_x+t+\alpha^{-1}(Z_x-t)_+]$, as well as the mean-upper-semideviation risk measure of order $p\geq 1$ with $\D[Z_x]=\E^{1/p}[(Z_x-\E[Z_x])_+^p]$ and $c\in(0,1]$. However, the applications of many convex risk measures are still limited due to the lack of efficient solution methods. In this paper, we focus on solving the following mean-upper-semideviation RAO problem:
\begin{align}
    \min_{x\in X} \ h(x):=\E[F(x,\xi)]+c\E^{\tfrac{1}{p}}\left[(F(x,\xi)-\E[F(x,\xi)])_+^p\right],
    \label{lpproblem}
\end{align}
where $p>1,c\in(0,1]$, and $X\subseteq\R^n$ is a convex and compact set. For any $x\in X$, $F(x,\xi)\in \mathrm{L}_p(\Omega,\cF,\P):=\{Z:\Omega\mapsto\R \ \text{measurable} \mid \int_{\Omega}|Z(\omega)|^pd\P(\omega)<\infty\}$. Moreover, $F(\cdot,\xi)$ is continuous and convex on $X$ for every $\xi\in\Xi$, and $F$ is $\mathrm{L}_p$-continuous with respect to $x$ on $X$, i.e., $\lim_{x'\to x}\E[|F(x',\xi)-F(x,\xi)|^p]=0$ for all $x\in X$. As a result, in view of~\cite[Proposition 6.9]{shapiro2021lectures}, $h$ is finite-valued and continuous on $X$. Note that $c\in (0,1]$ is required because, as discussed in~\cite[Section 6]{shapiro2021lectures}, the mean-upper-semideviation risk measure with $c>1$ violates the monotonicity condition (see~\cite[Axiom R2]{shapiro2021lectures}) and therefore fails to be a convex risk measure. Consequently, $h(x)$ may become nonconvex when $c>1$. In contrast, when $c\in (0,1]$, $h(x)$ is convex by the composition rule for risk measures~\cite[Proposition 6.11]{shapiro2021lectures}, since the mean-upper-semideviation risk measure is convex and monotone (see~\cite[Example 6.23]{shapiro2021lectures}), and $F$ is convex. 
It is well-known that RAO is closely related to distributionally robust optimization (DRO) (see \cite{shapiro2021lectures}).
In particular,
it is shown in~\cite[Section 6]{shapiro2021lectures} that for any random outcome $Z\in\mathrm{L}_p(\Omega,\cF,\P)$, the mean-upper-semideviation risk measure $\rho[Z]$ is equivalent to the distributionally robust formulation $\sup_{\Q\in\cP}\E_\Q[Z]$ with the ambiguity set $\cP=\{\Q:\exists \zeta\in\mathrm{L}_q(\Omega,\cF,\P) \ \mathrm{s.t.} \ \mathrm{d}\Q/\mathrm{d}\P=1+\zeta-\E_\P[\zeta],\|\zeta\|_q\leq c,\zeta\succeq 0\}$, where $1/p+1/q=1,\|\zeta\|_q=(\int_{\omega\in\Omega}|\zeta(\omega)|^qd\P(\omega))^{1/q}$, and $\zeta\succeq 0$ means $\zeta(\omega)\geq 0$ for $\P$-almost every $\omega\in\Omega$. For any probability distribution $\Q\in\cP$, the Radon--Nikodym derivative $\mathrm{d}\Q/\mathrm{d}\P$ has expectation one and deviates from one by a zero-mean random variable with a bounded $q$-th moment under the reference probability measure $\P$.

In this paper, we focus on the design of efficient algorithms for solving the RAO problem~\eqref{lpproblem} with $p>1$ (and its associated DRO formulation). Analogous to the standard setting in stochastic programming, we assume that access to a stochastic first-order oracle: each call at $x$ returns an independent random observation $(F(x,\xi),F'(x,\xi))$ satisfying 
\[
\E[F(x,\xi)]=f(x)\ \ \mbox{and} \ \ \E[F'(x,\xi)]\in \partial f(x),
\] 
where $\partial f(x)$ denotes the subdifferential of $f$ at $x$. 
Despite the importance of $L_p$ risk minimization,
there has been relatively little work on developing stochastic first-order methods with provable sample complexity guarantees.

To motivate our approach, we briefly review several natural solution strategies and explain the challenges they encounter when applied to problem~\eqref{lpproblem}.
Define $f_0(x):=x^{1/p},F_1(x_1,x_2,\xi):=(F(x_1,\xi)-x_2)_+^p$ and $F_2:=(x,F(x,\xi))$. Then the objective function in problem~\eqref{lpproblem} admits the following nested representation:
\begin{align*}
    h(x)=\E[F(x,\xi)]+c\E^{\tfrac{1}{p}}\left[(F(x,\xi)-\E[F(x,\xi)])_+^p\right]=\E[F(x,\xi)]+cf_0(\E[F_1(\E[F_2(x,\xi)],\xi)]).
\end{align*}
A natural approach is to apply stochastic mirror descent (SMD)~\cite{nemirovski2009robust,nemirovskij1983problem}, approximating $\E[F(x,\xi)]$ by an empirical average at each iteration. However, establishing the stochastic first-order oracle complexity is challenging because the outer layer function $f_0$ is concave and is neither uniformly Lipschitz continuous nor uniformly Lipschitz smooth on $[0,\infty)$. Moreover,
owing to the concavity of $f_0$, the stochastic sequential dual (SSD) method proposed in~\cite{zhang2024optimal} for 
convex nested stochastic optimization problems is also inapplicable, since it requires every layer function to be convex. Another possibility is to disregard the convexity structure of  problem~\eqref{lpproblem}, and instead apply general methods for nonconvex stochastic optimization~\cite{ghadimi2020single}, which converge to stationary points. 
However, this approach requires all layer functions to be smooth, whereas both $F_1$ and $F_2$ may be nonsmooth whenever $F$ is nonsmooth. Even when $F$ is smooth with Lipschitz continuous gradients, the method in~\cite{ghadimi2020single} is still not directly applicable, since $f_0$ is not uniformly Lipschitz smooth on $[0,\infty)$. 
More generally, when all layer functions are nonsmooth but locally Lipschitz continuous,~\cite{ruszczynski2021stochastic} establishes an asymptotically convergent algorithm. Unfortunately, this framework still does not apply here because $f_0$ fails to be locally Lipschitz continuous on $[0,\infty)$.
To the best of our knowledge, no efficient stochastic first-order method with provable sample complexity guarantees currently exists for solving the $L_p$ risk minimization problem when $p>1$.


\subsection{Our Contributions}

In this paper, we first solve the following mean-upper-semideviation risk averse optimization problem~\eqref{lpproblem} of order $p=2$:
\begin{align}
    \min_{x\in X} \ \left\{ h(x):=\E[F(x,\xi)]+c\E^{\tfrac{1}{2}}\left[(F(x,\xi)-\E[F(x,\xi)])_+^2\right]\right\},
    \label{l2problem}
\end{align}
then extend our results to the more general $L_p$ risk minimization problem for any $p>1$. For notational simplicity, we define
\begin{align}
    f(x):=\E[F(x,\xi)] \quad \text{and} \quad g(x):=\E\left[(F(x,\xi)-\E[F(x,\xi)])_+^2\right].
    \label{def:fg}
\end{align}
Our contributions can be briefly summarized as follows.

First, we propose a lifting formulation for the upper semideviation risk measure. After lifting the upper semideviation part by introducing a one-dimensional auxiliary variable $z>0$, the outer concave layer $f_0$ disappears, and problem~\eqref{l2problem} is reformulated as a two-layer nested problem, with the new objective function being jointly convex with respect to $(F(x,\xi),z)$ (but not jointly convex with respect to $(x,z)$ in general). In order to further obtain a convex formulation and remove the remaining inner nested structure, we introduce another one-dimensional auxiliary variable $y$ as a substitution for $\E[F(x,\xi)]$. The new objective function appears to be jointly convex with respect to $(x,y,z)$, and due to the monotonic non-decreasing of the function $y+c\E^{1/2}[(F(x,\xi)-y)_+^2]$ with respect to $y$, we relax $\E[F(x,\xi)]-y=0$ to a convex constraint $\E[F(x,\xi)]-y\leq 0$. Therefore, problem~\eqref{l2problem} is reformulated to a convex functional constrained stochastic optimization problem, and equivalently, a convex-concave stochastic saddle point problem as follows:
\begin{align*}
    \inf_{x\in X,y\in\R,z>0} \max_{\lambda\geq 0} \ \left\{\tfrac{c}{z}\E\left[(F(x,\xi)-y)_+^2\right]+y+\tfrac{c}{4}z+\lambda(\E[F(x,\xi)]-y)\right\}.
\end{align*}
Note that the feasible region for $z$ can be extended to the closed set $[0,+\infty)$; when $z=0$, we define the objective function to take the value $y+\lambda(\E[F(x,\xi)]-y)$ if $\E[(F(x,\xi)-y)_+^2]=0$, and $+\infty$ otherwise. This extension ensures that the reformulation remains well-defined. Based on this reformulation, we show the boundedness of the dual variable $\lambda$ and modify the feasible regions for the primal variables $y$ and $z$, under which the Lipschitz constants for the objective functions are bounded in the order of $\cO(\epsilon^{-2})$. Therefore, problem~\eqref{l2problem} is solvable by directly applying various existing methods with the oracle complexity $\cO(\epsilon^{-6})$.

Second, in order to further improve the oracle complexity, we propose a two-layer probabilistic bisection method to solve the reformulated convex-concave stochastic saddle point problem. When $z$ is small, the size of the subgradient blows up and is dominated by the subgradient with respect to $z$. To address this issue, our method separates $z$ from $x$ and $y$, solving a convex-concave stochastic minimax problem under a fixed $z$ in the inner layer, and searching for the optimal value $z=z^*$ through bisections in the outer layer. The oracle complexity of our method is bounded by $\Tilde{\cO}\left(\max\{z^*,\epsilon\}^{-2}\epsilon^{-2}\right)$ for finding an $\epsilon$-optimal solution for problem~\eqref{l2problem}, where $z^*$ relies on the specific instance of problem~\eqref{l2problem}. Moreover, we introduce a reliable stopping criterion that prevents our method from incurring the worst-case complexity $\Tilde{\cO}(\epsilon^{-4})$ when $z^*>\cO(\epsilon)$, by avoiding excessive computations after an $\epsilon$-optimal solution is found. The probabilistic bisection method, which is conducted based on the linear combination of the stochastic subgradients output by the inner layer stochastic mirror descent method, appears to be novel and can be of independent interest.

Third, we derive a lower complexity bound for solving problem~\eqref{l2problem} in general, by constructing and analyzing a specific subclass of problems. The subclass of problems satisfies the assumptions we make for problem~\eqref{l2problem}, and we can specify these problems to have the instance-dependent value $z^*$ as any nonnegative value. We illustrate the relation between the optimality gaps $g(x)-g(x^*)$ and $h(x)-h(x^*)$ in these problems, where $x^*$ is the optimal solution for $\min_{x\in X}h(x)$, and analyze the lower complexity bound $\Omega(\epsilon^{-2})$ for the problem $\min_{x\in X}g(x)$. Based on the results above, we generate a lower complexity bound for problem~\eqref{l2problem} as $\Omega\left(\max\{z^*,\epsilon\}^{-2}\epsilon^{-2}\right)$, which matches the oracle complexity achieved by our probabilistic bisection method. It shows that our method is nearly optimal for solving problem~\eqref{l2problem}.

Fourth, we extend our lifting formulation and probabilistic bisection method to the $L_p$ risk minimization problem~\eqref{lpproblem} of any order $p>1$, and construct a lower complexity bound for this more general class of problems. The oracle complexity for our method is $\Tilde{\cO}\left(\max\{z^*,\epsilon\}^{-2p+2}\epsilon^{-2}\right)$, which matches the lower complexity bound as $\Omega\left(\max\{z^*,\epsilon\}^{-2p+2}\epsilon^{-2}\right)$. Therefore, our method is nearly optimal for an $L_p$ risk minimization problem in general.

This paper is structured as follows. Section 2 introduces the lifting formulation for the upper semideviation risk measure. Section 3 presents the two-layer probabilistic bisection method and Section 4 is dedicated to the lower complexity bound for the mean-upper-semideviation RAO problem. Section 5 extends our formulation, method and lower complexity bound to the general $L_p$ risk minimization problem, followed by the conclusion in Section 6.
\subsection{Notations and Assumptions}

Throughout the paper, we use the following notations and assumptions. Let $\|\cdot\|_X$ denote a norm on $\R^n$, and $\|\cdot\|_{*,X}$ denote the dual norm of $\|\cdot\|_X$ that is defined as $\|y\|_{*,X}=\sup_{\|x\|_X\leq 1}\{y^Tx\}$. Since $X$ is convex and compact, we assume $\|x_1-x_2\|_X\leq D_X$ for any $x_1,x_2\in X$. Moreover, we assume the solution set $X^*:=\argmin_{x\in X}h(x)$ to be nonempty and denote $x^*\in X^*$ as an arbitrary optimal solution. The standard inner product is denoted as $\langle\cdot,\cdot\rangle$, $(x)_+$ is defined as $\max\{x,0\}$ for any $x\in\R$, the convex hull of a set $S$ is denoted as $\mathrm{co}\{S\}$, and the ball with radius $r$ centered at $x$ is defined as $B(x,r):=\{y:\|y-x\|\leq r\}$. For any function $r:S\to\R$, we say that the infimum (respectively, supremum) of $r$ over $S$ is approached as $s\to s'\in\mathrm{cl}(S)$ if $\inf_{s\in S}r(s)=\lim_{s\to s'}r(s)$ (respectively, $\sup_{s\in S}r(s)=\lim_{s\to s'}r(s)$), where $\mathrm{cl}(S)$ denotes the closure of $S$. For any continuous function $r: X\to\R$, we interpret differentiability at any $x\in X$ as differentiability of some extension of $r$ to an open neighborhood of $x$. In this sense, the set of Clarke subgradients of $r$ at $x$ is defined as:
\begin{align*}
    \partial r(x)=\mathrm{co}\left\{\lim_{i\to\infty}\nabla r(x_i): r \ \text{is differentiable at every} \ x_i\in X, \text{and} \ x_i\to x\right\}.
\end{align*}
We denote $r'(x)\in\partial r(x)$. A function $r:\R^n\to\R$ is $L$-Lipschitz continuous on $X\subseteq\R^n$ if for any $x_1,x_2\in X$,
\begin{align*}
    |r(x_1)-r(x_2)|\leq L\|x_1-x_2\|_X,
\end{align*}
and by~\cite[Proposition 2.1.2]{clarke1990optimization}, it is equivalent that for any $x\in X$
\begin{align*}
    \|r'(x)\|_{*,X}\leq L.
\end{align*}
A continuous function $r:\R^n\to\R$ is $\mu$-strongly convex on $X\subseteq\R^n$ if for any $x_1,x_2\in X$ and $r'(x_1)\in\partial r(x_1)$
\begin{align*}
    r(x_2)\geq r(x_1)+\langle r'(x_1),x_2-x_1\rangle+\frac{\mu}{2}\|x_1-x_2\|_X^2.
\end{align*}

Let $v:X\to\R$ be any distance generating function that is continuously differentiable and $1$-strongly convex (modulus $1$ with respect to $\|\cdot\|_X$). The prox-function (Bregman divergence) $V(\cdot,\cdot)$ associated with $v(\cdot)$ is defined as:
\begin{align}
    V(x_1,x_2)=v(x_2)-v(x_1)-\langle\nabla v(x_1),x_2-x_1\rangle, \ \forall x_1,x_2\in X,
    \label{def:prox-function}
\end{align}
and we have
\begin{align}
    V(x_1,x_2)\geq\tfrac{1}{2}\|x_1-x_2\|_X^2, \ \forall x_1,x_2\in X
    \label{assum:Vstrongconvex}
\end{align}
by the strong convexity of $v$.

As a standard measure of optimality, we define $\Bar{x}$ as an $\epsilon$-optimal solution for a general convex stochastic optimization problem $\min_{x\in X}\{r(x):=\E_{\xi}[R(x,\xi)]\}$, if $\Bar{x}\in X$ and
\begin{align}
    \E[r(\Bar{x})]-\min_{x\in X}r(x)\leq\epsilon,
    \label{eq:def_eps_optimal_solution}
\end{align}
where the expectation in~\eqref{eq:def_eps_optimal_solution} is taken over the randomness of the samples used by the algorithm.

In problem~\eqref{l2problem}, we assume $f(x)=\E[F(x,\xi)]$ is $L_f$-Lipschitz continuous on $X$, i.e.,
\begin{align}
    f(x_1)-f(x_2)\leq L_f\|x_1-x_2\|_X,\quad\forall x_1,x_2\in X \qquad \text{and} \qquad \|f'(x)\|_{*,X}\leq L_f,\quad\forall x\in X.
    \label{assum:flipschitz}
\end{align}
Moreover, we assume that there exists a first-order stochastic oracle that returns a pair of unbiased estimators $(F(x,\xi),F'(x,\xi))$ for $(f(x),f'(x))$ independently in each query to the oracle. Particularly, given any $x\in X$, it outputs $F'(x,\xi)$ such that
\begin{align}
    \E\left[\|F'(x,\xi)-f'(x)\|_{*,X}^2\right]\leq\sigma_f^2,
    \label{assum:fgradbddvar}
\end{align}
and outputs $F(x,\xi)$ such that
\begin{align}
    \E\left[(F(x,\xi)-f(x))^2\right]\leq\beta^2.
    \label{assum:fbddvar}
\end{align}
\section{Lifting Formulation for Upper Semideviation Risk Measure}

In this section, we introduce a novel lifting strategy to take care of the upper semideviation part, and reformulate problem~\eqref{l2problem} as a convex functional constrained stochastic optimization problem, and equivalently, a convex-concave stochastic saddle point problem.

\subsection{Convex Functional Constrained Stochastic Optimization Problem}

In this subsection, we introduce two auxiliary variables to help us reformulate problem~\eqref{l2problem} as a convex functional constrained stochastic optimization problem, which is in a more solvable form and could potentially be handled by various constrained stochastic methods.

\subsubsection{First Auxiliary Variable --- Removing Outer Nested Structure}

We define the upper semideviation risk measure $R(U):=\E^{1/2}[(U-\E[U])_+^2]$ where $U$ is a random variable. Our first goal is to get rid of the outer concave layer (the square root) in the upper semideviation. To motivate our approach, we first consider a proper, upper semicontinuous and concave function $f_0:S\to\mathbb{R}\cup\{\pm\infty\}$, with $S$ being closed and the variable $s \in S$ possibly being in expectation form (e.g. $s=\E[U]$ where $U$ is a random variable). It follows that $-f_0(s)$ is a proper, lower semicontinuous and convex function, and its Fenchel conjugate, given by $(-f_0)^*(\pi)=\sup_{s\in S}\pi s+f_0(s)$, is convex. Let $\Pi$ denote the domain of $(-f_0)^*$, i.e., $\Pi=\left\{\pi|(-f_0)^*(\pi)<\infty\right\}$. For any $\pi\in\Pi$, we suppose that the previous supremum is attained at $s=s^*(\pi)$, thus $(-f_0)^*(\pi)=\pi s^*(\pi)+f_0(s^*(\pi))$. By the Fenchel-Moreau Theorem, the biconjugate function satisfies $(-f_0)^{**}=-f_0$ such that
\begin{align*}
    -f_0(s)=(-f_0)^{**}(s)=\sup_{\pi\in\Pi}s\pi-(-f_0)^*(\pi)=\sup_{\pi\in\Pi}s\pi-\pi s^*(\pi)-f_0(s^*(\pi)),
\end{align*}
and as a consequence, we have $f_0(s)=\inf_{\pi\in\Pi}-\pi s+\pi s^*(\pi)+f_0(s^*(\pi))$. Note that when the variable $s$ is in expectation form, the formula above provides us with a way to decouple the concave layer $f_0$ from the expectation term $s$. We apply this approach to the special case with $f_0(s):=s^{1/2}$ on $S=[0,\infty)$ and $s=\E[(U-\E[U])_+^2]$ (i.e., $f_0(s)=R(U)$). Notice that $-f_0(s)$ is proper convex and continuous on $S$ with the conjugate function given by $(-f_0)^*(\pi)=\sup_{s\in S}\pi s+s^{1/2}$. The domain of $(-f_0)^*$ is $\Pi:=(-\infty,0)$, since $\lim_{s \to \infty} \pi s+s^{1/2}\to\infty$ when $\pi \ge 0$; when $\pi<0$, $\sup_{s\in S} \pi s+s^{1/2}$ is attained at $s^*=1/(4\pi^2)$ and $(-f_0)^*(\pi)=-1/(4\pi)$ for any $ \pi \in \Pi$. By the Fenchel-Moreau Theorem, we have $-f_0(s)=(-f_0)^{**}(s)=\sup_{\pi\in\Pi}s\pi-(-f_0)^*(\pi)=\sup_{\pi<0}s\pi+1/(4\pi)$ where $(-f_0)^{**}$ is the biconjugate function. Finally, if we set $z=-1/\pi\in(0,\infty)$, we can formulate $f_0(s)$ as the infimum of a perspective function (i.e. $f_0(s)=\inf_{z>0}s/z+z/4$). Inspired by this discussion, we provide the lemma below.
\begin{lemma}
    For any random variable $U$, $R(U)$ can be reformulated as
    \begin{align}
        R(U)=\inf_{z>0}\left\{r(U,z):=\tfrac{\E\left[(U-\E[U])_+^2\right]}{z}+\tfrac{z}{4}\right\},
        \label{def:RU}
    \end{align}
    and $\inf_{z>0}r(U,z)$ is approached as $z\to 2\E^{1/2}\left[(U-\E[U])_+^2\right]$.
    \label{lem1form}
\end{lemma}
\begin{proof}
    For any $U$, $r(U,z)$ is convex with respect to $z$ when $z>0$. Let $s$ denote $\E[(U-\E[U])_+^2]$. When $s=0$, $r(U,z)=z/4$ and its infimum is approached as $z\to 0^+$; when $s>0$, according to the first-order optimality condition $-s/z^2+1/4=0$, $\inf_{z>0}r(U,z)$ is attained at $z=2s^{1/2}$. Combining the two cases completes the proof.
\end{proof}

\vgap

In the above formulation~\eqref{def:RU} of $r(U,z)$, we adopt the auxiliary variable $z$ to lift $\E^{1/2}[(U-\E[U])_+^2]$ as $\E[(U-\E[U])_+^2]/z+z/4$, and in this way, the square root over the expectation vanishes. When we take $z=2\E^{1/2}[(U-\E[U])_+^2]$ which is the optimal value of the auxiliary variable, the lifted form reduces back to the square-root form. Meanwhile, since $\E[(U-\E[U])_+^2]$ is convex with respect to $U$, $r(U,z)=z\E[(U/z-\E[U/z])_+^2]+z/4$ is jointly convex with respect to $(U,z)$ when $z>0$, thus $\inf_{U,z}r(U,z)$ provides a convex problem formulation for minimizing $R(U)$. It follows that problem~\eqref{l2problem} can be reformulated as
\begin{align}
    \inf_{x\in X,z>0}\left\{\E[F(x,\xi)]+cr(F(x,\xi),z)=\E\left[F(x,\xi)+\tfrac{c(F(x,\xi)-\E[F(x,\xi)])_+^2}{z}+\tfrac{cz}{4}\right]\right\}.
    \label{eq1form}
\end{align}

\subsubsection{Second Auxiliary Variable --- Removing Inner Nested Structure}

Although the outer concave layer has disappeared in the formulation~\eqref{eq1form} above, our original concerns still remain. To begin with, problem~\eqref{eq1form} still contains a nested structure, and the existing optimal methods in~\cite{zhang2024optimal} for solving this problem require the joint convexity of the objective function with respect to $(x,z)$. However, the joint convexity is not guaranteed~\footnote{We illustrate this through a trivial example. Consider the one-dimensional case that $F(x,\xi)=16$ with probability $1/2$, $F(x,\xi)=4x^2$ with probability $\tfrac{1}{2}$, $X=[-1,1]$ and $c=1$. In this case, $E[F(x,\xi)]=8+2x^2$ and the objective function in problem~\eqref{eq1form} under $z=1$ is $8+2x^2+[16-(8+2x^2)]^2/2+1/4=2x^4-14x^2+161/4$, with the second derivative with respect to $x$ being $24x^2-28<0$. Thus, the objective function for problem~\eqref{eq1form} is not convex with respect to $x$ under fixed $z=1$, hence not jointly convex with respect to $(x,z)$.} (the convexity with respect to $x$ is not even guaranteed under fixed $z>0$). Therefore, our next goal is to get rid of the inner nested structure. It can be reached similarly through introducing another one-dimensional auxiliary variable $y$, and hiding the value of $\E[F(x,\xi)]$ inside the optimal new auxiliary variable $y^*(x)$. More specifically, we define the upper semideviation risk measure from the target $y$ as $R_y(U):=\E^{1/2}[(U-y)_+^2]$. Similar to Lemma~\ref{lem1form}, $R_y(U)$ can also be reformulated as
\begin{align*}
    R_y(U)=\inf_{z>0}\left\{r_y(U,z):=\tfrac{\E\left[(U-y)_+^2\right]}{z}+\tfrac{z}{4}\right\},
\end{align*}
where the infimum is approached as $z\to 2\E^{1/2}[(U-y)_+^2]$. Note that $R(U)$ and $r(U)$ use $\E[F(x,\xi)]$ as the target $y$ in $R_y(U)$ and $r_y(U)$, respectively. Then following from~\eqref{eq1form}, problem~\eqref{l2problem} can be reformulated as
\begin{align}
    \begin{cases}
        \inf_{x\in X,y\in\R,z>0} \ &\phi(x,y,z):=y+cr_y(F(x,\xi),z)=\E\left[y+\tfrac{c(F(x,\xi)-y)_+^2}{z}+\tfrac{cz}{4}\right] \\
        \mathrm{s.t.} \ &y=\E[F(x,\xi)].
    \end{cases}
    \label{eq2form}
\end{align}
By introducing the auxiliary variable $y$, not only the inner nested structure is eliminated from the objective function in problem~\eqref{eq1form}, but the new objective function $\phi(x,y,z)$ becomes jointly convex with respect to $(x,y,z)$. To show this in details, we start from the joint convexity of the function $\E[(U-y)_+^2]$ with respect to $(U,y)$, which is guaranteed by the convexity and monotonic increasing of the outer squared function over $[0,\infty)$, as well as the inner piecewise linear function. Then the function $\E[y+c(U-y)_+^2/z+cz/4]=y+cz\E[(U/z-y/z)_+^2]+cz/4$ is jointly convex with respect to $(U,y,z)$ when $z>0$. Since this function is also monotonically non-decreasing with respect to $U$, and $U=F(x,\xi)$ is convex, the joint convexity of $\phi(x,y,z)$ with respect to $(x,y,z)$ is justified.

The convex objective function in problem~\eqref{eq2form} provides us with a hope to construct a convex constrained problem that is equivalent to problem~\eqref{l2problem}. However, the auxiliary variable $y$ is restricted by a nonconvex equality constraint in problem~\eqref{eq2form}. Fortunately, we observe that the function
\begin{align}
    \Phi(x,y):=y+cR_y(F(x,\xi),z)=y+c\E^{\tfrac{1}{2}}\left[(F(x,\xi)-y)_+^2\right]=\inf_{z>0}\phi(x,y,z)
    \label{def:Phi_xy}
\end{align}
is monotonically non-decreasing with respect to $y$ for every fixed $x\in X$, by the following lemma.

\begin{lemma}
    For any fixed $x\in X$, $\Phi(x,y)$ defined in~\eqref{def:Phi_xy} is monotonically non-decreasing with respect to $y$ on $\R$.
    \label{lem:Phi_monotone_y}
\end{lemma}
\begin{proof}
    For any fixed $x\in X$, we show that the subgradient of $\Phi(x,y)$ with respect to $y$, i.e., $\Phi'_y(x,y)$ is nonnegative at every $y\in\R$. To this end, we consider two cases. The first case $\E[(F(x,\xi)-y)_+^2]=0$ is trivial, because in view of~\eqref{def:Phi_xy}, we have $\Phi(x,y)=y$ and $\Phi'_y(x,y)=1>0$. Below we discuss the second case $\E[(F(x,\xi)-y)_+^2]>0$. Using the chain rule (see~\cite[Theorem 2.3.9]{clarke1990optimization} and~\cite[Theorem 4.3.1]{hiriart2004fundamentals}), we obtain
    \begin{align}
        \Phi'_y(x,y)=1+c\tfrac{\partial \E^{1/2}\left[(F(x,\xi)-y)_+^2\right]}{\partial \E\left[(F(x,\xi)-y)_+^2\right]}\times\tfrac{\partial \E\left[(F(x,\xi)-y)_+^2\right]}{\partial y}=1-\tfrac{c\E\left[(F(x,\xi)-y)_+\right]}{\E^{1/2}\left[(F(x,\xi)-y)_+^2\right]},
        \label{chainrule1}
    \end{align}
    in which the second equality applies
    \begin{align}
        \tfrac{\partial \E[(F(x,\xi)-y)_+^2]}{\partial y}=\E\left[\tfrac{\partial(F(x,\xi)-y)_+^2}{\partial y}\right]=\E\left[-2(F(x,\xi)-y)_+\right],
        \label{chainrule2}
    \end{align}
    where the first equality follows from the dominated convergence theorem, i.e., $|[(F(x,\xi)-y-\Delta y)_+^2-(F(x,\xi)-y)_+^2]/\Delta y| \leq 2(|F(x,\xi)|+|y|)+|\Delta y|=:G(x,y,\Delta y,\xi)$ for all $\Delta y>0$, and $G(x,y,\Delta y,\xi)$ is integrable with respect to $\xi$ since $F(x,\xi)\in\mathrm{L}_2(\Omega,\cF,\P)$, and the second equality in~\eqref{chainrule2} holds by combining the following two cases on $\xi$: when $F(x,\xi)\geq y$, we have $\partial (F(x,\xi)-y)_+^2/\partial y=\partial (F(x,\xi)-y)^2/\partial y=-2(F(x,\xi)-y)$; otherwise, $F(x,\xi)<y$ yields $(F(x,\xi)-y)_+=0$ and thus $\partial (F(x,\xi)-y)_+^2/\partial y=0$. It then follows from~\eqref{chainrule1} that
    \begin{align}
        \Phi'_y(x,y)=1-\tfrac{c\E\left[(F(x,\xi)-y)_+\right]}{\E^{1/2}\left[(F(x,\xi)-y)_+^2\right]}\geq\tfrac{\E^{1/2}\left[(F(x,\xi)-y)_+^2\right]-\E\left[(F(x,\xi)-y)_+\right]}{\E^{1/2}\left[(F(x,\xi)-y)_+^2\right]}\geq 0,
        \label{eq:ymonotone}
    \end{align}
    where the first inequality is satisfied since $c\leq 1$, and the second inequality holds due to the fact that $\E[(F(x,\xi)-y)_+^2]=\mathrm{Var}\left((F(x,\xi)-y)_+\right)+\E^2[(F(x,\xi)-y)_+]\geq\E^2[(F(x,\xi)-y)_+]$. Combining the two cases above completes the proof.
\end{proof}
\vgap

Lemma~\ref{lem:Phi_monotone_y} allows us to relax the equality restriction on $y$ (see~\eqref{eq2form}) to an inequality constraint. In other words, problem~\eqref{eq2form} can be equivalently written as
\begin{align*}
    \begin{cases}
        \min_{x\in X,y\in\R} \ &\Phi(x,y)=y+c\E^{\tfrac{1}{2}}\left[(F(x,\xi)-y)_+^2\right] \\
        \mathrm{s.t.} \ &y\geq\E[F(x,\xi)].
    \end{cases}
\end{align*}
Notice that the inequality constraint above is convex. Since $\Phi(x,y)=\inf_{z>0}\phi(x,y,z)$, we finally present the formulation of an equivalent convex functional constrained stochastic optimization problem for problem~\eqref{l2problem} as
\begin{align}
    \begin{cases}
        \inf_{x\in X,y\in\R,z>0} \ &\phi(x,y,z)=\tfrac{c}{z}\E\left[(F(x,\xi)-y)_+^2\right]+y+\tfrac{c}{4}z \\
        \mathrm{s.t.} \ &\E[F(x,\xi)]-y\leq 0.
    \end{cases}
    \label{eq3form}
\end{align}
The result below summarizes important properties of $\phi(x,y,z)$ and $\Phi(x,y)$.
\begin{lemma}
    For any $x\in X$ and $y\in\mathbb{R}$, $\inf_{z>0}\phi(x,y,z)=\Phi(x,y)$ is approached as $z\to 2\E^{1/2}[(F(x,\xi)-y)_+^2]$. Moreover, for any $x\in X$, $\inf_{y\geq\E[F(x,\xi)],z>0}\phi(x,y,z)=\inf_{y\geq\E[F(x,\xi)]}\Phi(x,y)=h(x)$ is attained at $y=\E[F(x,\xi)]$ and approached as $z\to 2\E^{1/2}[(F(x,\xi)-\E[F(x,\xi)])_+^2]$.
    \label{lem:phi}
\end{lemma}
\begin{proof}
    For any $x\in X$ and $y\in\mathbb{R}$, we consider two cases. When $\E[(F(x,\xi)-y)_+^2]=0$, $\phi(x,y,z)=y+cz/4$ and $\inf_{z>0}\phi(x,y,z)$ is approached as $z\to 0^+$. Otherwise, since $\phi(x,y,z)$ is convex with respect to $z$, the first-order optimality condition is
    \begin{align*}
        -\tfrac{c}{z^2}\E\left[(F(x,\xi)-y)_+^2\right]+\tfrac{c}{4}=0,
    \end{align*}
    thus $\inf_{z>0}\phi(x,y,z)$ is approached as $z\to 2\E^{1/2}[(F(x,\xi)-y)_+^2]$. Moreover, for any $x\in X$, due to the monotonic non-decreasing of $\Phi(x,y)$ with respect to $y$ illustrated in~\eqref{eq:ymonotone}, $\inf_{y\geq\E[F(x,\xi)]}\Phi(x,y)$ is attained at $y=\E[F(x,\xi)]$.
\end{proof}
\vgap

Let $x^*\in X$ denote the optimal solution for problem~\eqref{l2problem}, and $(x^*,y^*,z^*)$ denote the point where the infimum in problem~\eqref{eq3form} is approached, then based on Lemma~\ref{lem:phi},
\begin{align}
    y^*=\E[F(x^*,\xi)] \quad \text{and} \quad z^*=2\E^{\tfrac{1}{2}}\left[(F(x^*,\xi)-\E[F(x^*,\xi)])_+^2\right].
    \label{def:yzstar}
\end{align}
The optimality gap of problem~\eqref{eq3form} provides an upper bound for the optimality gap of problem~\eqref{l2problem}, that is, for any $\Bar{x}\in X$, $\Bar{y}\geq\E[F(\Bar{x},\xi)]$ and $\Bar{z}>0$,
\begin{align*}
    h(\Bar{x})-h(x^*)=\inf_{y\geq\E[F(\Bar{x},\xi)],z>0}\phi(\Bar{x},y,z)-\inf_{y\geq\E[F(x^*,\xi)],z>0}\phi(x^*,y,z)\leq \phi(\Bar{x},\Bar{y},\Bar{z})-\lim_{z\to z^*}\phi(x^*,y^*,z),
\end{align*}
where the equality follows from Lemma~\ref{lem:phi}.

\subsection{Convex-Concave Stochastic Saddle Point Problem}

In this subsection, we consider one natural way to solve the convex functional constrained problem~\eqref{eq3form}. First, notice that the feasible region $\{z\mid z>0\}$ for the variable $z$ is open. Here, we extend the objective function $\phi(x,y,z)$ in problem~\eqref{eq3form} by defining
\begin{align*}
    \Tilde{\phi}(x,y,z):=\begin{cases}
        \phi(x,y,z) \ &\text{if } z>0, \\
        y+\tfrac{c}{4}z \ &\text{if } \E[(F(x,\xi)-y)_+^2]=0, \\
        +\infty \ &\text{otherwise}
    \end{cases}
\end{align*}
as an extended real-valued function over a closed set $X\times\R\times[0,+\infty)$. Clearly, in view of Lemma~\ref{lem:phi}, problem~\eqref{eq3form} is equivalent to the following problem
\begin{align}
    \begin{cases}
        \min_{x\in X,y\in\R,z\geq 0} \ &\Tilde{\phi}(x,y,z) \\
        \mathrm{s.t.} \ &\E[F(x,\xi)]-y\leq 0.
    \end{cases}
    \label{extendedproblem}
\end{align}
Since the variable $y$ is unbounded from above, there exists $(x,y)$ such that $\E[F(x,\xi)]-y<0$, then the Slater's condition is satisfied for problem~\eqref{extendedproblem}. Therefore, the strong duality holds, which allows us to transform problem~\eqref{extendedproblem} to a convex-concave minimax problem,  i.e.,
\begin{align}
    \min_{x\in X,y\in\R,z\geq 0}\max_{\lambda\geq 0} \ \left\{\Tilde{L}(x,y,z,\lambda):=\Tilde{\phi}(x,y,z)+\lambda\left(\E[F(x,\xi)]-y\right)\right\}.
    \label{extendedminmaxproblem}
\end{align}
Furthermore, problem~\eqref{extendedminmaxproblem} is equivalent to the following reformulated problem
\begin{align}
    \inf_{x\in X,y\in\R,z>0}\max_{\lambda\geq 0} \ \left\{L(x,y,z,\lambda):=\tfrac{c}{z}\E\left[(F(x,\xi)-y)_+^2\right]+y+\tfrac{c}{4}z+\lambda(\E[F(x,\xi)]-y)\right\},
    \label{eq4form}
\end{align}
for problem~\eqref{eq3form}, where $L(x,y,z,\lambda)$ is defined as the Lagrange function and $\lambda$ is the Lagrange multiplier. Here, $L(x,y,z,\lambda)$ is convex with respect to the primal variable $(x,y,z)$ and concave (linear) with respect to the dual variable $\lambda$. The saddle point $(x^*,y^*,z^*,\lambda^*)$ of this minimax problem~\eqref{eq4form} identifies the point $(x^*,y^*,z^*)$ where the infimum in problem~\eqref{eq3form} is approached, and the optimal solution $x^*$ for problem~\eqref{l2problem}. We also define
\begin{align}
    \cL(x,y,z,\lambda,\xi):=\tfrac{c}{z}(F(x,\xi)-y)_+^2+y+\tfrac{c}{4}z+\lambda(F(x,\xi)-y),
    \label{def:caliL}
\end{align}
and have $L(x,y,z,\lambda)=\E[\cL(x,y,z,\lambda,\xi)]$. Below we analyze some important properties about problem~\eqref{eq4form}, including the boundedness for the optimal dual variable $\lambda^*$, an upper bound for the optimality gap of problem~\eqref{l2problem} provided by problem~\eqref{eq4form}, and the Lipschitz continuity of the Lagrange function $L(x,y,z,\lambda)$.

\subsubsection{Boundedness for Optimal Dual Variable $\lambda^*$ and Optimality Gap}
Even though the domain of $\lambda$ in problem~\eqref{extendedproblem} appears to be unbounded, 
we can specify a convex and compact set $\Lambda$ 
that contains the optimal dual variable $\lambda^*$. Under the existence of a Slater point for problem~\eqref{extendedproblem}, the KKT optimality conditions are satisfied at its optimal solution $(x^*,y^*,z^*)$ with the Lagrange multiplier $\lambda=\lambda^*$. Let $\Tilde{L}'_y(x,y,z,\lambda)$ denote the subgradient of $\Tilde{L}(x,y,z,
\lambda)$ (see the definition of $\Tilde{L}$ in~\eqref{extendedminmaxproblem}) with respect to $y$ at $(x,y,z,\lambda)$. According to the KKT conditions, when $\E[(F(x^*,\xi)-y^*)_+^2]=0$, $\Tilde{L}'_y(x^*,y^*,z^*,\lambda^*)=1-\lambda^*=0$ implies $\lambda^*=1$; otherwise, $\E[(F(x^*,\xi)-y^*)_+^2]>0$ first yields $z^*>0$ (see~\eqref{def:yzstar}), and hence by using~\eqref{chainrule2}, we obtain
\begin{align*}
    \Tilde{L}'_y(x^*,y^*,z^*,\lambda^*)=-\tfrac{2c\E\left[(F(x^*,\xi)-y^*)_+\right]}{z^*}+1-\lambda^*=0,
\end{align*}
which, in view of the facts that $c\leq 1$ and $z^*=2\E^{1/2}[(F(x^*,\xi)-y^*)_+^2]\geq 2\E[(F(x^*,\xi)-y^*)_+]$ by~\eqref{def:yzstar}, implies
\begin{align*}
    \lambda^*=\tfrac{z^*-2c\E\left[(F(x^*,\xi)-y^*)_+\right]}{z^*}\in [0,1].
\end{align*}
Thus we define the feasible region for $\lambda$ as $\Lambda:=[0,1]$. Moreover, there exists a certain relationship between the optimality gaps of problem~\eqref{l2problem} and problem~\eqref{eq4form}. For any $\Bar{x}\in X$, $\Bar{y}\in\R$, $\Bar{z}>0$ and $\Bar{\lambda}\in\Lambda$, first we have
\begin{align}
    h(x^*)=\lim_{z\to z^*}L(x^*,y^*,z,0)=\lim_{z\to z^*}L(x^*,y^*,z,\Bar{\lambda})\geq\inf_{x\in X,y,z>0}L(x,y,z,\Bar{\lambda}), \label{eq:lower_h}
\end{align}
where the first equality follows from $h(x^*)=\lim_{z\to z^*}\phi(x^*,y^*,z)$ by Lemma~\ref{lem:phi} and $\lim_{z\to z^*}\phi(x^*,y^*,z)=\lim_{z\to z^*}L(x^*,y^*,z,0)$, and the second equality follows from $y^*=\E[F(x^*,\xi)]$ by~\eqref{def:yzstar}. Moreover, we can establish an upper bound on $h(\Bar{x})$ by considering two cases. When $\Bar{y}\geq \E[F(\Bar{x},\xi)]$, $\max_{\lambda\in\Lambda}L(\Bar{x},\Bar{y},\Bar{z},\lambda)$ is attained at $\lambda=0$, then
\begin{align*}
    h(\Bar{x})=\inf_{y\geq\E[F(\Bar{x},\xi)],z>0}L(\Bar{x},y,z,0)\leq L(\Bar{x},\Bar{y},\Bar{z},0)=\max_{\lambda\in\Lambda}L(\Bar{x},\Bar{y},\Bar{z},\lambda),
\end{align*}
where the first equality follows from $h(\bar{x})=\inf_{y\geq\E[F(\Bar{x},\xi)],z>0}\phi(\bar{x},y,z)$ by Lemma~\ref{lem:phi} and $\phi(\Bar{x},y,z)=L(\Bar{x},y,z,0)$. When $\Bar{y}<\E[F(\Bar{x},\xi)]$, $\max_{\lambda\in\Lambda}L(\Bar{x},\Bar{y},z,\lambda)$ is attained at $\lambda=1$, then
\begin{align*}
    h(\Bar{x})&=\E[F(\Bar{x},\xi)]+c\E^{\tfrac{1}{2}}\left[(F(\Bar{x},\xi)-\E[F(\Bar{x},\xi)])_+^2\right]\leq\E[F(\Bar{x},\xi)]+c\E^{\tfrac{1}{2}}\left[(F(\Bar{x},\xi)-\Bar{y})_+^2\right] \\
    &=\Bar{y}+c\E^{\tfrac{1}{2}}\left[(F(\Bar{x},\xi)-\Bar{y})_+^2\right]+(\E[F(\Bar{x},\xi)]-\Bar{y})=\inf_{z>0}L(\Bar{x},\Bar{y},z,1)\leq L(\Bar{x},\Bar{y},\Bar{z},1)=\max_{\lambda\in\Lambda}L(\Bar{x},\Bar{y},\Bar{z},\lambda),
\end{align*}
where the first inequality holds because $\E^{1/2}[(F(\Bar{x},\xi)-\Bar{y})_+^2]$ is monotonically non-increasing with respect to $\Bar{y}$, and the third equality follows from $\Phi(\Bar{x},\Bar{y})=\inf_{z>0}\phi(\Bar{x},\Bar{y},z)$ by Lemma~\ref{lem:phi} and $\phi(\Bar{x},\Bar{y},z)+(\E[F(\Bar{x},\xi)]-\Bar{y})=L(\Bar{x},\Bar{y},z,1)$. Therefore, in both cases we have
\begin{align}
    h(\Bar{x})\leq\max_{\lambda\in\Lambda}L(\Bar{x},\Bar{y},\Bar{z},\lambda).\label{eq:upper_h}
\end{align}
Combining \eqref{eq:lower_h} and \eqref{eq:upper_h}, we have
\begin{align*}
    h(\Bar{x})-h(x^*)\leq \max_{\lambda\in\Lambda}L(\Bar{x},\Bar{y},\Bar{z},\lambda)-\inf_{x\in X,y\in\R,z>0}L(x,y,z,\Bar{\lambda}).
\end{align*}

\subsubsection{Guarantee on Lipschitz Continuity for Lagrange Function}

The Lipschitz continuity for the objective function $L(x,y,z,\lambda)$ plays an important role in the analysis of various stochastic first-order methods for solving the convex-concave stochastic saddle point problem~\eqref{eq4form}. Recall that $L(x,y,z,\lambda)=\E[\cL(x,y,z,\lambda,\xi)]$, and let $\cL'_x(x,y,z,\lambda,\xi),\cL'_y(x,y,z,\lambda,\xi),\cL'_z(x,y,z,\lambda,\xi)$ and $\cL'_\lambda(x,y,z,\lambda,\xi)$ be the subgradients of $\cL(x,y,z,\lambda)$ with respect to $x,y,z$ and $\lambda$, respectively. Then we define the vector
\begin{align}
    \cL'(x,y,z,\lambda,\xi):=\left[
    \begin{array}{c}
    \cL'_x(x,y,z,\lambda,\xi) \\
    \cL'_y(x,y,z,\lambda,\xi) \\
    \cL'_z(x,y,z,\lambda,\xi) \\
    -\cL'_\lambda(x,y,z,\lambda,\xi)
    \end{array}
    \right]=\left[
    \begin{array}{c}
    \tfrac{2c}{z}(F(x,\xi)-y)_+F'(x,\xi)+\lambda F'(x,\xi) \\
    -\tfrac{2c}{z}(F(x,\xi)-y)_++1-\lambda \\
    -\tfrac{c}{z^2}(F(x,\xi)-y)_+^2+\tfrac{c}{4} \\
    -F(x,\xi)+y
    \end{array}
    \right],
    \label{eq:subgrad1}
\end{align}
where the calculus of the subgradients applies the chain rule (see~\cite[Theorem 2.3.9]{clarke1990optimization} and~\cite[Theorem 4.3.1]{hiriart2004fundamentals}). Note that $[\cL'_x(x,y,z,\lambda,\xi);\cL'_y(x,y,z,\lambda,\xi);\cL'_z(x,y,z,\lambda,\xi)]$ represents a subgradient vector of $\cL$ with respect to $(x,y,z)$. To verify this, for any $\lambda\in\Lambda$, we first define $\Tilde{\cL}^\lambda(F,y,z):=(c/z)(F-y)_+^2+y+(c/4)z+\lambda(F-y)$ on $\R\times\R\times(0,\infty)$, so that $\Tilde{\cL}^\lambda(F(x,\xi),y,z)=\cL(x,y,z,\lambda,\xi)$. By the differentiability of $\Tilde{\cL}^\lambda$, its gradient at any $(F,y,z)$ is $[(2c/z)(F-y)_++\lambda;-(2c/z)(F-y)_++1-\lambda;-(c/z^2)(F-y)_+^2+c/4]$. Hence, for any $x,x'\in X,y,y'\in\R,z,z'>0$ and $\xi\in\Xi$, we have
\begin{align}
    &\cL(x',y',z',\lambda,\xi)-\cL(x,y,z,\lambda,\xi) = \Tilde{\cL}^\lambda(F(x',\xi),y',z')-\Tilde{\cL}^\lambda(F(x,\xi),y,z) \nn\\
    &\geq \left[\tfrac{2c}{z}(F(x,\xi)\!-\!y)_+\!+\!\lambda\right]\!(F(x',\xi)\!-\!F(x,\xi))\!+\!\left[-\tfrac{2c}{z}(F(x,\xi)\!-\!y)_+\!+\!1\!-\!\lambda\right]\!(y'\!-\!y)\!+\!\left[-\tfrac{c}{z^2}(F(x,\xi)\!-\!y)_+^2\!+\!\tfrac{c}{4}\right]\!(z'\!-\!z) \nn\\
    &\geq \left[\tfrac{2c}{z}(F(x,\xi)\!-\!y)_+\!+\!\lambda\right]\!F'(x,\xi)^{\top}(x'\!-\!x)\!+\!\left[-\tfrac{2c}{z}(F(x,\xi)\!-\!y)_+\!+\!1\!-\!\lambda\right]\!(y'\!-\!y)\!+\!\left[-\tfrac{c}{z^2}(F(x,\xi)\!-\!y)_+^2\!+\!\tfrac{c}{4}\right]\!(z'\!-\!z),
    \label{cL_joint_subgrad}
\end{align}
where the first inequality uses the convexity of $\Tilde{\cL}^\lambda$ (see the discussion below~\eqref{eq2form}), and the second inequality utilizes $(2c/z)(F(x,\xi)-y)_++\lambda\geq 0$ by $z>0$ and $\lambda\in\Lambda$, as well as the convexity of $F(x,\xi)$ with respect to $x$. As a consequence of~\eqref{cL_joint_subgrad}, $[\cL'_x(x,y,z,\lambda,\xi);\cL'_y(x,y,z,\lambda,\xi);\cL'_z(x,y,z,\lambda,\xi)]$ is a subgradient vector of $\cL$ with respect to $(x,y,z)$. 
In view of~\eqref{eq:subgrad1}, the size of the vector $L'(x,y,z,\lambda):=\E[\cL'(x,y,z,\lambda,\xi)]$ blows up as $z \to 0$ or $y \to \pm \infty$, under which case there does not exist a bounded Lipschitz constant for the Lagrange function $L(x,y,z,\lambda)$. As such, we need to properly modify the feasible regions for $y$ and $z$ in order to show the Lipschitz continuity of $L$.

\paragraph{Boundedness for Domain of Variable $y$}

In view of~\eqref{eq:subgrad1}, the size of $L'(x,y,z,\lambda)$ blows up as $y\to\pm\infty$. Therefore, we hope to find a bounded feasible region for $y$ that contains $y^*=\E[F(x^*,\xi)]$. Suppose that a feasible point $x_0\in X$ is known. Using the boundedness of $X$ and the Lipschitz continuity of the function $f(x)=\E[F(x,\xi)]$, for any $x\in X$, we have
\begin{align}
    \left|f(x)-f(x_0)\right|\leq L_f\|x-x_0\|_X\leq L_fD_X.
    \label{eq:yscope1}
\end{align}
Therefore, for any $x\in X$, $f(x)$ lies within the set $Y:=\{y\mid|y-f(x_0)|\leq L_fD_X\}$. However, we are only able to obtain an unbiased estimation of $f(x_0)$. Since the variance of $F(x,\xi)$ is bounded for any $x\in X$, the estimation error is under control. Specifically, for any $\delta>0$, if we estimate $f(x_0)$ using
\begin{align*}
    \hat{f}(x_0)=\tfrac{1}{m}\tsum_{i=1}^mF(x_0,\xi_i)
\end{align*}
where $m\in\N_+$ and $\xi_i,i=1,...,m$ are independent samples, then according to~\eqref{assum:fbddvar} and Chebyshev's inequality,
\begin{align}
    \P\left(\left|\hat{f}(x_0)-f(x_0)\right|>\delta\right)\leq\tfrac{\beta^2}{m\delta^2},
    \label{eq:yscope2}
\end{align}
and it follows from~\eqref{eq:yscope1} and~\eqref{eq:yscope2} that
\begin{align}
    \P\left(\left|f(x)-\hat{f}(x_0)\right|\leq L_fD_X+\delta, \ \forall x\in X\right)\geq 1-\tfrac{\beta^2}{m\delta^2}.
    \label{eq:yscope3}
\end{align}
Without loss of generality, we pick $\delta=1$ and define the set
\begin{align}
    Y:=\left\{y:\left|y-\hat{f}(x_0)\right|\leq L_fD_X+1\right\},
    \label{def:ydelta}
\end{align}
then for any $\alpha_0\in (0,1)$, taking $m\geq\beta^2/\alpha_0$,~\eqref{eq:yscope3} implies that with probability at least $1-\alpha_0$, we have
\begin{align}
    f(x)=\E[F(x,\xi)]\in Y \quad \forall x\in X.
    \label{eq:Ycontainy}
\end{align}
Moreover, in view of~\eqref{def:ydelta}, the boundedness of $Y$ is given by
\begin{align}
    |y_1-y_2|\leq 2(L_fD_X+1) \quad \forall y_1,y_2\in Y.
    \label{eq:ybound}
\end{align}

Combining the boundedness of the variables $\lambda$ and $y$, we conclude that with probability at least $1-\alpha_0$, problem~\eqref{eq4form} is equivalent to the following problem:
\begin{align}
    \inf_{x\in X,y\in Y,z>0}\max_{\lambda\in\Lambda}L(x,y,z,\lambda):=\tfrac{c}{z}\E\left[(F(x,\xi)-y)_+^2\right]+y+\tfrac{c}{4}z+\lambda (\E[F(x,\xi)]-y).
    \label{eq:l2equiform}
\end{align}

Since the estimation of $Y$ is conducted independently prior to solving the saddle point problem~\eqref{eq4form}, we henceforth condition on the event~\eqref{eq:Ycontainy}.

\paragraph{Closedness for Domain of Variable $z$}

In problem~\eqref{eq:l2equiform}, the domain $(0,+\infty)$ for the variable $z$ is still open, and the size of $L'(x,y,z,\lambda)$ blows up as $z\to 0$. Since $z^*$ can be any nonnegative value depending on the problem instance, it is impossible in general to find a positive lower bound for $z^*$. However, by adding a small enough positive lower bound for $z$, we can still maintain approximate optimality for problem~\eqref{eq:l2equiform} and thus for problem~\eqref{l2problem}. Specifically, we consider the following problem
\begin{align}
    \begin{cases}
        \min_{x\in X,y\in Y,z\geq\epsilon} \ &\phi(x,y,z)=\tfrac{c}{z}\E\left[(F(x,\xi)-y)_+^2\right]+y+\tfrac{c}{4}z \\
        \mathrm{s.t.} \ &\E[F(x,\xi)]-y\leq 0.
    \end{cases}
    \label{eq5form}
\end{align}
Since the Slater's condition is satisfied for problem~\eqref{eq5form}, it is equivalent to the following minimax problem
\begin{align}
    \min_{x\in X,y\in Y,z\geq\epsilon}\max_{\lambda\in\Lambda}L(x,y,z,\lambda)=\tfrac{c}{z}\E\left[(F(x,\xi)-y)_+^2\right]+y+\tfrac{c}{4}z+\lambda(\E[F(x,\xi)]-y).
    \label{eq6form}
\end{align}
Let $(x_\epsilon^*,y_\epsilon^*,z_\epsilon^*,\lambda_\epsilon^*)$ denote the optimal saddle point for problem~\eqref{eq6form}.
The result below shows the approximate optimality of $x_\epsilon^*$ for problem~\eqref{l2problem} in terms of $h(x^*_\epsilon) - h(x^*)$, and provides two different ways to upper bound the optimality gap $h(\bar x) - h(x^*)$.
\begin{lemma}
    Let $x^*$ be optimal for problem~\eqref{l2problem}, and $(x_\epsilon^*,y_\epsilon^*,z_\epsilon^*,\lambda_\epsilon^*)$ be an optimal saddle point for problem~\eqref{eq6form} for some $\epsilon >0$. Then, we have
    $$h(x_\epsilon^*)-h(x^*)\leq\tfrac{c}{2}\epsilon.$$
    Moreover, for any feasible solution
    $(\bar x, \bar y, \bar z)$ of problem~\eqref{eq5form}, we have
    $$h(\Bar{x})-h(x^*)<\left[\phi(\Bar{x},\Bar{y},\Bar{z})-\phi(x_\epsilon^*,y_\epsilon^*,z_\epsilon^*)\right]+\tfrac{c}{2}\epsilon.$$
    In addition, for any $\Bar{x}\in X$, $\Bar{y} \in \mathbb{R}$, $\Bar{z}\geq\epsilon$ and $\Bar{\lambda}\in\Lambda$, we have
    $$h(\Bar{x})-h(x^*)<\left[\max_{\lambda\in\Lambda}L(\Bar{x},\Bar{y},\Bar{z},\lambda)-\min_{x\in X,y\in Y,z\geq\epsilon}L(x,y,z,\Bar{\lambda})\right]+\tfrac{c}{2}\epsilon.$$
    \label{lem2form}
\end{lemma}
\begin{proof}
    We first consider the situation 
    when $z^*\geq\epsilon$. In this case, the optimal solution $(x^*,y^*,z^*)$ for problem~\eqref{eq3form} is feasible for problem~\eqref{eq5form}, thus  $(x_\epsilon^*,y_\epsilon^*,z_\epsilon^*)=(x^*,y^*,z^*)$ and $h(x_\epsilon^*)-h(x^*)=0$. Moreover, for any $\Bar{x}\in X$, $\bar{y}\in Y,\Bar{y}\geq\E[F(\Bar{x},\xi)]$ and $\Bar{z}\geq\epsilon$,
    \begin{align*}
        h(\Bar{x})-h(x^*)&=\left[h(\Bar{x})-\phi(x_\epsilon^*,y_\epsilon^*,z_\epsilon^*)\right]+\left[\phi(x_\epsilon^*,y_\epsilon^*,z_\epsilon^*)-h(x^*)\right] \\
        &=\left[\inf_{y\geq\E[F(\Bar{x},\xi)],z>0}\phi(\Bar{x},y,z)-\phi(x_\epsilon^*,y_\epsilon^*,z_\epsilon^*)\right]+[\phi(x^*,y^*,z^*)-h(x^*)] \\
        &\leq\phi(\Bar{x},\Bar{y},\Bar{z})-\phi(x_\epsilon^*,y_\epsilon^*,z_\epsilon^*),
    \end{align*}
    where the second equality and the inequality follow from Lemma~\ref{lem:phi}. In addition, for any $\Bar{x}\in X$, $\Bar{y}\in Y$, $\Bar{z}\geq\epsilon$ and $\Bar{\lambda}\in\Lambda$,
    \begin{align*}
        h(\Bar{x})-h(x^*)&=\left[h(\Bar{x})-L(x^*,y^*,z^*,\Bar{\lambda})\right]+\left[\phi(x^*,y^*,z^*)-h(x^*)\right] \\
        &\leq\max_{\lambda\in\Lambda}L(\Bar{x},\Bar{y},\Bar{z},\lambda)-\min_{x\in X,y\in Y,z\geq\epsilon}L(x,y,z,\Bar{\lambda}),
    \end{align*}
    where the equality follows from $L(x^*,y^*,z^*,\Bar{\lambda})=\phi(x^*,y^*,z^*)$ since $y^*=\E[F(x^*,\xi)]$ by~\eqref{def:yzstar}, and the inequality follows from~\eqref{eq:upper_h} and Lemma~\ref{lem:phi}.
    
    Now we consider the situation when $z^*<\epsilon$. In this case we have
    \begin{align}
        h(x_\epsilon^*)=\inf_{y\geq\E[F(x_\epsilon^*,\xi)],z>0}\phi(x_\epsilon^*,y,z)\leq \phi(x_\epsilon^*,y_\epsilon^*,z_\epsilon^*)\leq \phi(x^*,y^*,\epsilon),
        \label{eq7form}
    \end{align}
    where the equality follows from Lemma~\ref{lem:phi}, the first inequality holds due to the feasibility of $(x_\epsilon^*,y_\epsilon^*,z_\epsilon^*)$ for problem~\eqref{eq5form}, and the second inequality uses the optimality of $(x_\epsilon^*,y_\epsilon^*,z_\epsilon^*)$ and the feasibility of $(x^*,y^*,\epsilon)$ for problem~\eqref{eq5form}. Therefore
    \begin{align}
        h(x_\epsilon^*)-h(x^*)&\leq \phi(x^*,y^*,\epsilon)-h(x^*)\leq\phi(x^*,y^*,\epsilon)-\lim_{z\to z^*}\phi(x^*,y^*,z) \nonumber\\
        &=\tfrac{cz^{*2}}{4\epsilon}+y^*+\tfrac{c}{4}\epsilon-\tfrac{cz^*}{4}-y^*-\tfrac{c}{4}z^*\leq\tfrac{cz^{*2}}{4\epsilon}+\tfrac{c}{4}\epsilon<\tfrac{c}{2}\epsilon,
        \label{eq8form}
    \end{align}
    where the first inequality applies~\eqref{eq7form}, the second inequality follows from Lemma~\ref{lem:phi}, and the first equality utilizes $z^*=2\E^{1/2}[(F(x^*,\xi)-y^*)_+^2]$ by~\eqref{def:yzstar}. Moreover, for any $\Bar{x}\in X$, $\Bar{y}\in Y, \bar{y}\geq\E[F(\Bar{x},\xi)]$ and $\Bar{z}\geq\epsilon$, 
    \begin{align*}
        h(\Bar{x})-h(x^*)&=\left[h(\Bar{x})-\phi(x_\epsilon^*,y_\epsilon^*,z_\epsilon^*)\right]+\left[\phi(x_\epsilon^*,y_\epsilon^*,z_\epsilon^*)-h(x^*)\right]<\left[\inf_{y\geq\E[F(\Bar{x},\xi)],z>0}\phi(\Bar{x},y,z)-\phi(x_\epsilon^*,y_\epsilon^*,z_\epsilon^*)\right]+\tfrac{c}{2}\epsilon \\
        &\leq\left[\phi(\Bar{x},\Bar{y},\Bar{z})-\phi(x_\epsilon^*,y_\epsilon^*,z_\epsilon^*)\right]+\tfrac{c}{2}\epsilon,
    \end{align*}
    where the first inequality follows from Lemma~\ref{lem:phi} and $\phi(x_\epsilon^*,y_\epsilon^*,z_\epsilon^*)-h(x^*)\leq\phi(x^*,y^*,\epsilon)-h(x^*)<c\epsilon/2$ by using~\eqref{eq7form} and~\eqref{eq8form}. In addition, for any $\Bar{x}\in X$, $\Bar{y}\in Y$, $\Bar{z}\geq\epsilon$ and $\Bar{\lambda}\in\Lambda$,
    \begin{align*}
        h(\Bar{x})-h(x^*)\!=\!\left[h(\Bar{x})\!-\!L(x^*,y^*,\epsilon,\Bar{\lambda})\right]\!+\!\left[\phi(x^*,y^*,\epsilon)\!-\!h(x^*)\right]\!<\!\left[\max_{\lambda\in\Lambda}L(\Bar{x},\Bar{y},\Bar{z},\lambda)\!-\!\min_{x\in X,y\in Y,z\geq\epsilon}\!L(x,y,z,\Bar{\lambda})\right]\!+\!\tfrac{c}{2}\epsilon,
    \end{align*}
    where the equality follows from $L(x^*,y^*,\epsilon,\Bar{\lambda})=\phi(x^*,y^*,\epsilon)$ since $y^*=\E[F(x^*,\xi)]$ by~\eqref{def:yzstar}, and the inequality follows from~\eqref{eq:upper_h} and~\eqref{eq8form}.
\end{proof}
\vgap

Lemma~\ref{lem2form} indicates that we can obtain an $\epsilon$-optimal solution for problem~\eqref{l2problem} by approximately solving problem~\eqref{eq5form} or~\eqref{eq6form}. 
It is worth noting that since $z \ge \epsilon > 0$,
in view of~\eqref{eq:subgrad1},
we can bound the Lipschitz constant of
for the Lagrange function $L(x,y,z,\lambda)$ for problem~\eqref{eq6form}
in the order of $O(\epsilon^{-2})$.

\subsubsection{Computation of Second Moment Bound for Stochastic Subgradients}
\label{subsec:lips}

Based on our previous analysis, we focus on solving problem~\eqref{eq6form}, i.e.,
\begin{align}
    \min_{x\in X,y\in Y,z\geq\epsilon}\max_{\lambda\in\Lambda}L(x,y,z,\lambda)=\tfrac{c}{z}\E\left[(F(x,\xi)-y)_+^2\right]+y+\tfrac{c}{4}z+\lambda(\E[F(x,\xi)]-y).
    \label{eq9form}
\end{align}

Here, we equip $\R^n$ with the norm $\|\cdot\|_X$ and $\R$ with the Euclidean norm $\|\cdot\|_2$, then $\R^n\times\R\times\R\times\R$ is equipped with the norm $\|\cdot\|$ that satisfies
\begin{align}
    \|(x,y,z,\lambda)\|=\sqrt{\|x\|_X^2+y^2+z^2+\lambda^2},
    \label{def:norm1}
\end{align}
for any $x\in\R^n,y,z,\lambda\in\R$, and the dual norm $\|\cdot\|_*$ satisfies
\begin{align*}
    \|(\zeta_x,\zeta_y,\zeta_z,\zeta_\lambda)\|_*=\sqrt{\|\zeta_x\|_{*,X}^2+\zeta_y^2+\zeta_z^2+\zeta_\lambda^2}
\end{align*}
for any $\zeta_x\in\R^n,\zeta_y,\zeta_z,\zeta_\lambda\in\R$.

In order to compute an upper bound for $\E[\|\cL'(x,y,z,\lambda,\xi)\|_*^2]$ (see the expression of $\cL'(x,y,z,\lambda,\xi)$ in~\eqref{eq:subgrad1}), we make the following two assumptions.
\begin{assumption}
    $\E[(F(x,\xi)-\E[F(x,\xi)])_+^4]\leq M_f^4$ for any $x\in X$.
    \label{assumption1}
\end{assumption}
Assumption~\ref{assumption1} is satisfied when the fourth central moment $\E[(F(x,\xi)-\E[F(x,\xi)])^4]$ is bounded, which can also be implied by the sub-Gaussian (light-tail) assumption $\E[\exp\{(F(x,\xi)-\E[F(x,\xi)])^4/\sigma^4\}]\leq\exp\{1\}$ for some $\sigma^4>0$ and any $x\in X$, since
\begin{align*}
    \exp\{1\}&\geq\E\left[\exp\left\{(F(x,\xi)-\E[F(x,\xi)])^4/\sigma^4\right\}\right]\geq\exp\left\{\E\left[(F(x,\xi)-\E[F(x,\xi)])^4\right]\Big/\sigma^4\right\},
\end{align*}
where the second inequality follows from Jensen's inequality, and taking logarithms on both sides yields $\E[(F(x,\xi)-\E[F(x,\xi)])^4]\leq\sigma^4$. Note that many random variables satisfy the sub-Gaussian assumption, including Gaussian, mean-zero or symmetric Bernoulli, uniform and any random variable with bounded support. Assumption~\ref{assumption1} will be used to bound the term $\E[\|\cL'_z(x,y,z,\lambda,\xi)\|_*^2]$, as shown in~\eqref{eq:Lzgradbound} in the proof of Lemma~\ref{lem:Lipsboundnum} below.

\begin{assumption}
    Define $G(x,\xi):=(F(x,\xi)-\E[F(x,\xi)])_+^2$, then $\E[\|G'(x,\xi)\|_{*,X}^2]\leq L_G^2$ for any $x\in X$.
    \label{assumption2}
\end{assumption}
Assumption~\ref{assumption2} requires Lipschitz continuity for the function $\E[G(x,\xi)]$, which represents the upper semivariance of $F(x,\xi)$. Using Young's inequality, we have
\begin{align*}
    \E\left[\left\|G'(x,\xi)\right\|_{*,X}^2\right]&=\E\left[\left\|2\left(F(x,\xi)-\E[F(x,\xi)]\right)_+\left(F'(x,\xi)-\E[F'(x,\xi)\right)\right\|_{*,X}^2\right] \\
    &\leq 2\left\{\E\left[(F(x,\xi)-\E[F(x,\xi)])_+^4\right]+\E\left[\left\|F'(x,\xi)-\E[F'(x,\xi)]\right\|_{*,X}^4\right]\right\},
\end{align*}
thus Assumption~\ref{assumption2} is satisfied when Assumption~\ref{assumption1} holds and $\E[\|F'(x,\xi)-\E[F'(x,\xi)]\|_{*,X}^4]$ is bounded. Note that the boundedness of $\E[\|F'(x,\xi)-\E[F'(x,\xi)]\|_{*,X}^4]$ can also be implied by the sub-Gaussian assumption $\E[\exp\{\|F'(x,\xi)-\E[F'(x,\xi)]\|_{*,X}^4/\sigma^4\}]\leq\exp\{1\}$ for some $\sigma^4>0$ and any $x\in X$, since
\begin{align*}
    \exp\{1\}&\geq\E\left[\exp\left\{\left\|F'(x,\xi)-\E[F'(x,\xi)]\right\|_{*,X}^4\big/\sigma^4\right\}\right]\geq\exp\left\{\E\left[\left\|F'(x,\xi)-\E[F'(x,\xi)]\right\|_{*,X}^4\right]\Big/\sigma^4\right\},
\end{align*}
where the second inequality follows from Jensen's inequality, and taking logarithms on both sides yields $\E[\|F'(x,\xi)-\E[F'(x,\xi)]\|_{*,X}^4]\leq\sigma^4$. Assumption~\ref{assumption2} will be utilized to bound the term $\E[\|\cL'_x(x,y,z,\lambda,\xi)\|_{*,X}^2]$, which can be seen in~\eqref{eq:Lxgradbound} in the proof of Lemma~\ref{lem:Lipsboundnum} below.

Based on Assumptions~\ref{assumption1} and~\ref{assumption2}, Lemma~\ref{lem:Lipsboundnum} below provides upper bounds for the components of $\E[\|\cL'(x,y,z,\lambda,\xi)\|_*^2]$.
\begin{lemma}
    Conditioning on~\eqref{eq:Ycontainy}, for any $x\in X,y\in Y,z>0$ and $\lambda\in\Lambda$, we have:
    \begin{align*}
        &\textit{(a)} \ \E\left[\left(\cL'_y(x,y,z,\lambda,\xi)\right)^2\right]\leq\tfrac{12c^2}{z^2}\left[\beta^2+4(L_fD_X+1)^2\right]+3; \\
        &\textit{(b)} \ \E\left[\left(\cL'_\lambda(x,y,z,\lambda,\xi)\right)^2\right]\leq 2\beta^2+8(L_fD_X+1)^2;
    \end{align*}
    under Assumption~\ref{assumption1} we have:
    \begin{align*}
        &\textit{(c)} \ \E\left[\left(\cL'_z(x,y,z,\lambda,\xi)\right)^2\right]\leq\tfrac{12c^2}{z^4}\left[M_f^4+16(L_fD_X+1)^4\right]+\tfrac{3c^2}{16};
    \end{align*}
    and under Assumption~\ref{assumption2} we have:
    \begin{align*}
        &\textit{(d)} \ \E\left[\left\|\cL'_x(x,y,z,\lambda,\xi)\right\|_{*,X}^2\right]\leq\tfrac{3c^2}{z^2}\left[L_G^2+4L_f^2\beta^2+64(L_f^2+\sigma_f^2)(L_fD_X+1)^2\right]+12(L_f^2+\sigma_f^2).
    \end{align*}
    \label{lem:Lipsboundnum}
\end{lemma}
\begin{proof}
First we denote $f(x)=\E[F(x,\xi)]$, then
\begin{align}
    \E\left[\left\|F'(x,\xi)\right\|_{*,X}^2\right]\leq 2\left\|f'(x)\right\|_{*,X}^2+2\E\left[\left\|F'(x,\xi)-f'(x)\right\|_{*,X}^2\right]\leq 2(L_f^2+\sigma_f^2),
    \label{eq:Flips}
\end{align}
where the first inequality applies Cauchy--Schwarz inequality and the second inequality follows from~\eqref{assum:flipschitz} and~\eqref{assum:fgradbddvar}. For the term $\E\left[(\cL'_y(x,y,z,\lambda,\xi))^2\right]$, for any $x\in X,y\in Y,z>0$ and $\lambda\in\Lambda$, we have
\begin{align}
    &\E\left[\left(\cL'_y(x,y,\lambda,\xi)\right)^2\right]=\E\left[\left(-\tfrac{2c}{z}(F(x,\xi)-y)_++1-\lambda\right)^2\right]\leq\E\left[\left(\tfrac{2c}{z}(F(x,\xi)-f(x))_++\tfrac{2c}{z}|y-f(x)|+1\right)^2\right] \nn\\
    &\leq\tfrac{12c^2}{z^2}\E\left[(F(x,\xi)-f(x))_+^2\right]+\tfrac{12c^2}{z^2}(y-f(x))^2+3\leq\tfrac{12c^2}{z^2}\beta^2+\tfrac{12c^2}{z^2}[2(L_fD_X+1)]^2+3 \nn\\
    &\leq\tfrac{12c^2}{z^2}\left[\beta^2+4(L_fD_X+1)^2\right]+3,
    \label{eq:Lygradbound}
\end{align}
where the first inequality uses $(F(x,\xi)-y)_+\leq (F(x,\xi)-f(x))_++|f(x)-y|$ and $\lambda\in[0,1]$, the second inequality applies Young's inequality, and the third inequality follows from~\eqref{assum:fbddvar},~\eqref{eq:Ycontainy} and~\eqref{eq:ybound}. For the term $\E[(\cL'_\lambda(x,y,z,\lambda,\xi))^2]$, for any $x\in X,y\in Y,z>0$ and $\lambda\in\Lambda$, we have
\begin{align}
    \E\left[\left(\cL'_\lambda(x,y,z,\lambda,\xi)\right)^2\right]=\E\left[(F(x,\xi)-y)^2\right]\leq2\E\left[(F(x,\xi)-f(x))^2\right]+2(f(x)-y)^2\leq 2\beta^2+8(L_fD_X+1)^2,
    \label{eq:Llamgradbound}
\end{align}
where the first inequality follows from Young's inequality, and the second inequality applies~\eqref{assum:fbddvar},~\eqref{eq:Ycontainy} and~\eqref{eq:ybound}. For the term $\E[(\cL'_z(x,y,z,\lambda,\xi))^2]$, for any $x\in X,y\in Y,z>0$ and $\lambda\in\Lambda$, we have
\begin{align}
    &\E\left[\left(\cL'_z(x,y,z,\lambda,\xi)\right)^2\right]=\E\left[\left(-\tfrac{c}{z^2}(F(x,\xi)-y)_+^2+\tfrac{c}{4}\right)^2\right]\leq\E\left[\left(\tfrac{2c}{z^2}(F(x,\xi)-f(x))_+^2+\tfrac{2c}{z^2}(f(x)-y)^2+\tfrac{c}{4}\right)^2\right] \nonumber\\
    &\leq\tfrac{12c^2}{z^4}\E\left[(F(x,\xi)-f(x))_+^4\right]+\tfrac{12c^2}{z^4}(f(x)-y)^4+\tfrac{3c^2}{16}\leq\tfrac{12c^2}{z^4}\left[M_f^4+16(L_fD_X+1)^4\right]+\tfrac{3c^2}{16},
    \label{eq:Lzgradbound}
\end{align}
where the first inequality uses $(F(x,\xi)-y)_+^2\leq 2(F(x,\xi)-f(x))_+^2+2(f(x)-y)^2$ since $(F(x,\xi)-y)_+\leq (F(x,\xi)-f(x))_++|f(x)-y|$, the second inequality applies Young's inequality, and the third inequality follows from Assumption~\ref{assumption1},~\eqref{eq:Ycontainy} and~\eqref{eq:ybound}. For the term $\E[\|\cL'_x(x,y,z,\lambda,\xi)\|_*^2]$, for any $x\in X,y\in Y,z>0$ and $\lambda\in\Lambda$, we have
\begin{align}
    &\E\left[\|\cL'_x(x,y,z,\lambda,\xi)\|_{*,X}^2\right]=\E\left[\left\|\tfrac{2c}{z}(F(x,\xi)-y)_+F'(x,\xi)+\lambda F'(x,\xi)\right\|_{*,X}^2\right] \nn\\
    &\leq\E\left[\left\|\tfrac{2c}{z}(F(x,\xi)-f(x))_+F'(x,\xi)+\tfrac{2c}{z}|y-f(x)|F'(x,\xi)+\lambda F'(x,\xi)\right\|_{*,X}^2\right] \nonumber\\
    &=\E\Bigg[\bigg\|\tfrac{2c}{z}(F(x,\xi)-f(x))_+(F'(x,\xi)-f'(x))+\tfrac{2c}{z}(F(x,\xi)-f(x))_+f'(x)+\left(\tfrac{2c}{z}|y-f(x)|+\lambda\right)F'(x,\xi)\bigg\|_{*,X}^2\Bigg] \nonumber\\
    &\leq\tfrac{3c^2}{z^2}\E\left[\left\|2(F(x,\xi)-f(x))_+(F'(x,\xi)-f'(x))\right\|_{*,X}^2\right]+\tfrac{12c^2}{z^2}\left\|f'(x)\right\|_{*,X}^2\E\left[(F(x,\xi)-f(x))_+^2\right] \nonumber\\
    & \quad +3\left(\tfrac{2c}{z}|y-f(x)|+1\right)^2\E\left[\left\|F'(x,\xi)\right\|_{*,X}^2\right] \nn\\
    &\leq\tfrac{3c^2}{z^2}L_G^2+\tfrac{12c^2}{z^2}L_f^2\beta^2+3\left[\tfrac{4c}{z}(L_fD_X+1)+1\right]^2\times 2(L_f^2+\sigma_f^2) \nn\\
    &\leq\tfrac{3c^2}{z^2}\left[L_G^2+4L_f^2\beta^2+64(L_f^2+\sigma_f^2)(L_fD_X+1)^2\right]+12(L_f^2+\sigma_f^2),
    \label{eq:Lxgradbound}
\end{align}
where the first inequality uses $(F(x,\xi)-y)_+\leq (F(x,\xi)-f(x))_++|f(x)-y|$, the second inequality utilizes Cauchy--Schwarz inequality, the third inequality follows from Assumption~\ref{assumption2},~\eqref{assum:flipschitz},~\eqref{assum:fbddvar},~\eqref{eq:Ycontainy},~\eqref{eq:ybound},~\eqref{eq:Flips} and $\lambda\in[0,1]$, and the fourth inequality applies Young's inequality.
\end{proof}

\vgap

Problem~\eqref{eq9form} is solvable by various optimization methods. For example, methods for solving a convex-concave stochastic saddle point problem include the stochastic mirror descent method~\cite{nemirovski2009robust},~\cite[Section 4.3]{lan2020first} and the stochastic accelerated mirror-prox method~\cite[Section 4.5]{lan2020first}. Moreover, problem~\eqref{eq9form} is equivalent to a convex functional constrained stochastic optimization problem, which can be solved by primal methods such as cooperative subgradient methods~\cite{lan2020algorithms} and level-set methods~\cite{lin2018level,lin2018level2}, as well as primal-dual methods such as~\cite{boob2023stochastic,yan2022adaptive}. Suppose that we apply the stochastic mirror descent method~\cite{nemirovski2009robust},~\cite[Section 4.3]{lan2020first} to solve problem~\eqref{eq9form}, in view of
\begin{align}
    \E\left[\left\|\cL'(x,y,z,\lambda,\xi)\right\|_*^2\right]=&\E\left[\left\|\cL'_x(x,y,z,\lambda,\xi)\right\|_{*,X}^2\right]+\E\left[\left(\cL'_y(x,y,z,\lambda,\xi)\right)^2\right]+\E\left[\left(\cL'_z(x,y,z,\lambda,\xi)\right)^2\right] \nn\\
    &+\E\left[\left(\cL'_\lambda(x,y,z,\lambda,\xi)\right)^2\right],
    \label{eq:Lipsbound}
\end{align}
then according to Lemma~\ref{lem:Lipsboundnum}, for any $x\in X,y\in Y,z\geq\epsilon$ and $\lambda\in\Lambda$,
\begin{align}
    \E\left[\left\|\cL'(x,y,z,\lambda,\xi)\right\|_*^2\right]\leq&\tfrac{12c^2}{\epsilon^4}\left[M_f^4+16(L_fD_X+1)^4\right]+2\beta^2+12(L_f^2+\sigma_f^2)+8(L_fD_X+1)^2+\tfrac{3c^2}{16}+3 \nn\\
    &+\tfrac{3c^2}{\epsilon^2}\left\{L_G^2+4(L_f^2+1)\beta^2+16\left[4(L_f^2+\sigma_f^2)+1\right](L_fD_X+1)^2\right\}.
    \label{eq:Lipsboundeps}
\end{align}

From above we see that the upper bound for $\E[\|\cL'(x,y,z,\lambda,\xi)\|_*^2]$ is in the order of $O(\epsilon^{-4})$, which will result in an overall $O(\epsilon^{-6})$ first-order oracle complexity for solving problem~\eqref{eq9form} and thus the mean-upper-semideviation problem~\eqref{l2problem}. Note that this complexity is much worse than the optimal complexity $O(\epsilon^{-2})$ for solving a convex stochastic optimization problem, as well as solving a convex stochastic nested optimization problem, such as the mean-upper-semideviation problem~\eqref{lpproblem} with $p=1$.
\section{Algorithm Design and Analysis}

In this section, we focus on efficiently solving the convex-concave stochastic saddle point problem~\eqref{eq9form}. 
As discussed earlier in Section~\ref{subsec:lips}, the key challenge of achieving a better complexity result stems from the fact that
the bound on $\E[\|\cL'(x,y,z,\lambda,\xi)\|_*^2]$ in~\eqref{eq:Lipsboundeps}
is quite large, given in the order of $\cO(\epsilon^{-4})$.
However, it might be possible to reduce this bound and thus the complexity of our algorithm through the following means. 

First, by Lemma~\ref{lem:Lipsboundnum}, when $z$ tends to zero,
the four components,
$\E[\|\cL'_x(x,y,z,\lambda,\xi)\|_{*,X}^2]$, $\E[(\cL'_y(x,y,z,\lambda,\xi))^2]$, $\E[(\cL'_z(x,y,z,\lambda,\xi))^2]$, and $\E[(\cL'_\lambda(x,y,z,\lambda,\xi))^2]$
appearing in \eqref{eq:Lipsbound} of $\E[\|\cL'(x,y,z,\lambda,\xi)\|_*^2]$ 
are bounded by $\cO(z^{-2})$,
$\cO(z^{-2})$, $\cO(z^{-4})$, and $\cO(1)$, respectively. This motivates us to develop algorithms that only weakly depends on the dominating third component $\E[(\cL'_z(x,y,z,\lambda,\xi))^2]$.

Second, let $z^*$ be defined in~\eqref{def:yzstar}. We may search over $z \ge \max\{z^*,\epsilon\}$, which according to Lemma~\ref{lem:Lipsboundnum}, guarantees a possibly smaller upper bound for $\E[\|\cL'(x,y,z,\lambda,\xi)\|_*^2]$ compared to searching over $z\geq\epsilon$. Notice that the value of $z^*$ is unknown and that it depends on the specific instance of problem~\eqref{l2problem}.

In view of the above discussion, by separating the one-dimensional variable $z$ from $x,y$ and $\lambda$, we consider the following equivalent min-min-max problem formulation to problem~\eqref{eq:l2equiform}:
\begin{align}
   \inf_{z>0}\left\{\psi(z):=\min_{x\in X,y\in Y}\max_{\lambda\in\Lambda}L(x,y,z,\lambda)\right\}.
    \label{l2minminmaxproblem}
\end{align}
Note that since $L(x,y,z,\lambda)$ is jointly convex with respect to $(x,y,z)$ for any $\lambda\in\Lambda$, $\max_{\lambda\in\Lambda}L(x,y,z,\lambda)$ is jointly convex with respect to $(x,y,z)$. Taking partial minimization of $\max_{\lambda\in\Lambda}L(x,y,z,\lambda)$ with respect to $(x,y)$, we can see that the objective function $\psi(z)$ in \eqref{l2minminmaxproblem} is convex. 
To find an approximate solution for problem~\eqref{l2minminmaxproblem}, we design a double-layer method in align with its min-min-max structure. 
Observe that we restrict $x,y$ and $\lambda$ in the same feasible region as in problem~\eqref{eq9form} for the inner min-max problem in \eqref{l2minminmaxproblem}, but do not force the hard constraint $z\geq\epsilon$ for the outer problem as in problem~\eqref{eq9form}. Instead, we will search for $z^*$ over $z \ge \cO(\max\{z^*,\epsilon\})$ in our algorithm.

\subsection{Stochastic Mirror Descent Method}

In this subsection, we focus on solving the inner saddle point problem of~\eqref{l2minminmaxproblem} given by
\begin{align}
    \min_{x\in X,y\in Y}\max_{\lambda\in\Lambda}L(x,y,z,\lambda)
    \label{problem:inner}
\end{align}
for a fixed value of $z>0$. While there exists various algorithms
for solving~\eqref{problem:inner}, we consider here the simplest stochastic mirror descent (SMD) method~\cite{nemirovski2009robust},~\cite[Section 4.3]{lan2020first}. For any $z>0$, we define the vectors
\begin{align}
    l'_z(x,y,\lambda,\xi):=\left[
    \begin{array}{c}
    \cL'_x(x,y,z,\lambda,\xi) \\
    \cL'_y(x,y,z,\lambda,\xi) \\
    -\cL'_\lambda(x,y,z,\lambda,\xi)
    \end{array}
    \right]=\left[
    \begin{array}{c}
    \tfrac{2c}{z}(F(x,\xi)-y)_+F'(x,\xi)+\lambda F'(x,\xi) \\
    -\tfrac{2c}{z}(F(x,\xi)-y)_++1-\lambda \\
    -F(x,\xi)+y
    \end{array}
    \right],
    \label{def:lzsubgrad}
\end{align}
and
\begin{align}
    \Bar{l}'_z(x,y,\lambda):=\E\left[l'_z(x,y,\lambda,\xi)\right].
    \label{def:barlz}
\end{align}
Notice that $\cL'_z(x,y,\lambda,z,\xi)$ is not utilized in~\eqref{def:lzsubgrad} because for each inner problem~\eqref{problem:inner}, $z$ is fixed as a constant. We equip $X$ with a distance generating function $v_X:X\to\R$ modulus $1$ with respect to the norm $\|\cdot\|_X$ and the associated prox-function $V_X(\cdot,\cdot)$, and
define the ``radius'' of $X$ centered at $x\in X$ as
\begin{align}
    D_{v_X,x}^2:=\max_{x'\in X}V_X(x,x').
    \label{def:DXv}
\end{align}
Note that $D_{v_X,x}^2$ is bounded by $LD_X^2/2$ if $v_X$ is $L$-Lipschitz smooth on $X$. For example, for any $X\subseteq\R^n$, the Euclidean distance with $v_X(\cdot)=\|\cdot\|_2^2/2$ has $L=1$. Moreover, even if $v_X$ is not Lipschitz smooth, $D_{v_X,x}^2$ can still be bounded given some specific $x\in X$. For instance, KL-divergence $V_X(\cdot,\cdot)$ on the simplex $X=\{x\in\R^n:x_i\geq 0 \ \forall i=1,...,n,\tsum_{i=1}^n x_i=1\}$ with $x=(1,1,...,1)^\top/n\in X$ has $\max_{x'\in X}V_X(x,x')=\max_{x'\in X}\{\tsum_{i=1}^n x'_i\log(x'_i/x_i)\}=\max_{x'\in X}\{\tsum_{i=1}^n x'_i\log(x'_i)\}+\log(n)\leq \log(n)$. We also equip the one-dimensional spaces $Y$ and $\Lambda$ with the distance generating function $v_2(\cdot)=\|\cdot\|_2^2/2$, and the associated prox-function $V_2(y_1,y_2)=(y_1-y_2)^2/2$. 
Furthermore, we define $U:=X\times Y\times\Lambda$, and equip it with the distance generating function $v(x,y,\lambda)=
v_X(x) + y^2/2 + \lambda^2/2$.
It is easy to see that $v$ has modulus $1$ with respect to the redefined norm $\|\cdot\|$ (see the original definition of $\|\cdot\|$ in~\eqref{def:norm1}) below that satisfies
\begin{align}
    \|(x,y,\lambda)\|=\sqrt{\|x\|_X^2+y^2+\lambda^2}
    \label{def:generalnorm}
\end{align}
for any $x\in\R^n,y,\lambda\in\R$, while its dual norm $\|\cdot\|_*$ satisfies $\|(\zeta_x,\zeta_y,\zeta_\lambda)\|_*=\sqrt{\|\zeta_x\|_{*,X}^2+\zeta_y^2+\zeta_\lambda^2}$ for any $\zeta_x\in\R^n,\zeta_y,\zeta_\lambda\in\R$. It follows that for any $u=(x,y,\lambda),u'=(x',y',\lambda')\in U$, the associated prox-function $V(u,u')$ satisfies
\begin{align}
V(u,u')&=V_X(x,x')+V_2(y,y')+V_2(\lambda,\lambda')\leq D_{v_X,x}^2+\tfrac{1}{2}(y-y')^2+\tfrac{1}{2}(\lambda-\lambda')^2 \nn\\
    &\leq D_{v_X,x}^2+2(L_fD_X+1)^2+\tfrac{1}{2}=:D_{V,x}^2
    \label{def:D}
\end{align}
 where the first inequality follows from~\eqref{def:DXv}, and the second inequality follows from~\eqref{eq:ybound} and $\Lambda=[0,1]$. 

\begin{algorithm}
\caption{The Stochastic Mirror Descent (SMD) Method for Solving Problem~\eqref{problem:inner}}
\label{algo:SMD}
\begin{algorithmic}
\State{\textbf{Input:} $(x_0,y_0,\lambda_0)=u_0\in U,z>0,T>0,\{\gamma_t\}_{t=0}^{T-1}$.}
\For{$t=0,1,...,T-1$}
\State{$u_{t+1} \equiv (x_{t+1},y_{t+1},\lambda_{t+1})=\argmin_{u\in U}\gamma_t\langle l'_z(u_t,\xi_t),u\rangle+V(u_t,u)$.}
\EndFor
\State{\textbf{Output:} $\Bar{u}_T=\left(\tsum_{t=0}^{T-1}\gamma_t\right)^{-1}\tsum_{t=0}^{T-1}\gamma_tu_t$ and $\zeta_T=\left(\tsum_{t=0}^{T-1}\gamma_t\right)^{-1}\tsum_{t=0}^{T-1}\gamma_t\cL'_z(x_t,y_t,z,\lambda_t,\xi'_t)$.}
\end{algorithmic}
\end{algorithm}

The SMD method for solving problem~\eqref{problem:inner} is formally stated in Algorithm~\ref{algo:SMD}.
This algorithm computes a pair of outputs $\bar u_T$ and $\zeta_T$, where $\bar u_T$
is an approximate solution for the inner saddle point problem~\eqref{problem:inner} and $\zeta_T$ is
an approximate subgradient of $\psi$ at $z$.
Note that for $t=0,...,T-1$, the samples $\xi_t$ used in $l'_z(u_t,\xi_t)$ to update $u_{t+1}$ are independent of the samples $\xi'_t$ in $\cL'_z(x_t,y_t,z,\lambda_t,\xi'_t)$ used to compute $\zeta_T$.

We are now ready to describe the convergence properties of Algorithm~\ref{algo:SMD}. Lemma~\ref{lem:SMDconverge} below
shows the convergence of the first output $\bar u_T$, while Lemma~\ref{lem:subdiff}
discusses the convergence of the second output $\zeta_T$ of Algorithm~\ref{algo:SMD}.

\begin{lemma}
    Under Assumption~\ref{assumption2}, and conditioning on~\eqref{eq:Ycontainy}, for any $\Tilde{c}>0$, if we set
    \begin{align*}
        T\geq\tfrac{81\Tilde{c}^2D_{V,x_0}^2M_z^2}{4\epsilon^2} \quad \text{and} \quad \gamma_t=\tfrac{D_{V,x_0}}{M_z\sqrt{T}}
    \end{align*}
     in Algorithm~\ref{algo:SMD}, where $D_{V,x_0}$ is defined in~\eqref{def:D} and
    \begin{align}
        M_z^2:=&\tfrac{3c^2}{z^2}\left\{L_G^2+4(L_f^2+1)\beta^2+16\left[4(L_f^2+\sigma_f^2)+1\right](L_fD_X+1)^2\right\}+12(L_f^2+\sigma_f^2)+8(L_fD_X+1)^2+2\beta^2+3,
        \label{def:M^2}
    \end{align}
    then the output solution $\Bar{u}_T =(\Bar{x}_T,\Bar{y}_T,\Bar{\lambda}_T)$ satisfies
    \begin{align*}
        \E\left[\max_{\lambda\in \Lambda}L(\Bar{x}_T,\Bar{y}_T,z,\lambda)-\min_{x\in X,y\in Y}L(x,y,z,\Bar{\lambda}_T)\right]\leq\left(\tsum_{t=0}^{T-1}\gamma_t\right)^{-1}\E\left[\max_{u\in U}\tsum_{t=0}^{T-1}\gamma_t\Bar{l}'_z(u_t)^{\top}(u_t-u)\right]\leq\tfrac{\epsilon}{\Tilde{c}}.
    \end{align*}
    \label{lem:SMDconverge}
\end{lemma}
\begin{proof}
    The convergence analysis of the SMD method for solving general convex-concave stochastic saddle point problems can be found in~\cite[Lemma 3.1]{nemirovski2009robust} and~\cite[Lemma 4.6]{lan2020first}. Based on these results, we have
    \begin{align}
        &\E\left[\max_{\lambda\in \Lambda}L(\Bar{x}_T,\Bar{y}_T,z,\lambda)-\min_{x\in X,y\in Y}L(x,y,z,\Bar{\lambda}_T)\right]\leq\left(\tsum_{t=0}^{T-1}\gamma_t\right)^{-1}\E\left[\max_{u\in U}\tsum_{t=0}^{T-1}\gamma_t\Bar{l}'_z(u_t)^{\top}(u_t-u)\right] \nn\\
        &\leq\left(\tsum_{t=0}^{T-1}\gamma_t\right)^{-1}\left(\max_{u\in U}2V(u_0,u)+\tfrac{5}{2}\tsum_{t=0}^{T-1}\gamma_t^2\E\left[\|l_z'(u_t,\xi)\|_*^2\right]\right).
        \label{eq1lem1algo}
    \end{align}
    For any $u\in U$, in view of the definition of $l_z'(u,\xi)$ in~\eqref{def:lzsubgrad}, we have
    \begin{align}
        &\E\left[\left\|l_z'(u,\xi)\right\|_*^2\right]=\E\left[\left\|\cL'_x(x,y,z,\lambda,\xi)\right\|_{*,X}^2\right]+\E\left[\left(\cL'_y(x,y,z,\lambda,\xi)\right)^2\right]+\E\left[\left(\cL'_\lambda(x,y,z,\lambda,\xi)\right)^2\right]\leq M_z^2,
        \label{eq:Mz}
    \end{align}
    where the inequality follows from Lemma~\ref{lem:Lipsboundnum} and the definition of $M_z^2$ in~\eqref{def:M^2}. Finally, utilizing $\gamma_t=D_{V,x_0}/(M_z\sqrt{T})$, and the bounds on $V(u_0,u)$ and $\E[\|l_z'(u_t,\xi)\|_*^2]$ provided in~\eqref{def:D} and~\eqref{eq:Mz}, respectively in~\eqref{eq1lem1algo}, we conclude
    \begin{align*}
        \E\left[\max_{\lambda\in \Lambda}L(\Bar{x}_T,\Bar{y}_T,\lambda)-\min_{x\in X,y\in Y}L(x,y,\Bar{\lambda}_T)\right]\leq\left(\tsum_{t=0}^{T-1}\gamma_t\right)^{-1}\E\left[\max_{u\in U}\tsum_{t=0}^{T-1}\gamma_t\Bar{l}'_z(u_t)^{\top}(u_t-u)\right]\leq\tfrac{9D_{V,x_0}M_z}{2\sqrt{T}},
    \end{align*}
    which, in view of $T\geq81\Tilde{c}^2D_{V,x_0}^2M_z^2/(4\epsilon^2)$, implies our result.
\end{proof}

\vgap

Lemma~\ref{lem:subdiff} below shows that the output $\zeta_T$ in Algorithm~\ref{algo:SMD},
given as a convex combination of $\cL'_z(x_t,y_t,z_t,\lambda,\xi'_t)$ for $t=0,...,T-1$,
is an approximate subgradient of $\psi(z)$. Note that since $\psi(z)$ is an one-dimensional convex function, the exact subgradient vector $\psi'(z)$ is nonnegative if and only if $z\geq z^*$. Later we will show that the approximation $\zeta_T$ of $\psi'(z)$ also
provides us with information to search for $z^*$.

\begin{lemma}
    Under Assumptions~\ref{assumption1} and~\ref{assumption2}, and conditioning on~\eqref{eq:Ycontainy}, for any $\alpha\in(0,1)$, if we set
    \begin{align}
        T\geq\max\left\{\tfrac{1296}{\alpha^2}D_{V,x_0}^2M_z^2,\tfrac{32}{\alpha}(z^*-z)^2\Tilde{M}_z^4\right\}\epsilon^{-2} \quad \text{and} \quad \gamma_t=\tfrac{D_{V,x_0}}{M_z\sqrt{T}}
        \label{Tgamma1}
    \end{align}
    in Algorithm~\ref{algo:SMD}, where $D_{V,x_0}$ and $M_z$ are defined in~\eqref{def:D} and~\eqref{def:M^2}, respectively, and
    \begin{align}
        \Tilde{M}_z^4:=\tfrac{8c^2}{z^4}\left[M_f^4+16(L_fD_X+1)^4\right]
        \label{def:tildeMz}
    \end{align}
    is the upper bound for $\mathrm{Var}(\cL'_z(x,y,z,\lambda,\xi))$, then
    \begin{align*}
        \psi^*\geq\psi(z)+\zeta_T(z^*-z)-\tfrac{\epsilon}{2}
    \end{align*}
    holds with probability at least $1-\alpha$, where $z^*$ is defined in~\eqref{def:yzstar} and $\psi^*:=\lim_{z\to z^*}\psi(z)=\inf_{z>0}\psi(z)$.
    \label{lem:subdiff}
\end{lemma}
\begin{proof}
    For any $z>0$, we let $u_z^*=(x_z^*,y_z^*,\lambda_z^*)$ be the solution for problem~\eqref{problem:inner} and denote $\Tilde{u}_z:=(x^*,y^*,\lambda_z^*)\in U$ (see~\eqref{def:yzstar}). Then for any $(x,y,\lambda)=u\in U$, we have
    \begin{align}
        \psi^*\overset{(a)}{=}&\lim_{\Tilde{z}\to z^*}\min_{\Tilde{x}\in X,\Tilde{y}\in Y}\max_{\Tilde{\lambda}\in\Lambda}L(\Tilde{x},\Tilde{y},\Tilde{z},\Tilde{\lambda})\overset{(b)}{=}\lim_{\Tilde{z}\to z^*}\max_{\Tilde{\lambda}\in\Lambda}L(x^*,y^*,\Tilde{z},\Tilde{\lambda})\overset{(c)}{=}\lim_{\Tilde{z}\to z^*}L(x^*,y^*,\Tilde{z},\lambda) \nonumber\\
        \overset{(d)}{\geq}&L(x,y,z,\lambda)+L'_x(x,y,z,\lambda)^{\top}(x^*-x)+L'_y(x,y,z,\lambda)(y^*-y)+\lim_{\Tilde{z}\to z^*}L'_z(x,y,z,\lambda)(\Tilde{z}-z) \nonumber\\
        \overset{(e)}{=}&L(x,y,z,\lambda_z^*)-L'_\lambda(x,y,z,\lambda)(\lambda_z^*\!-\!\lambda)+L'_x(x,y,z,\lambda)^{\top}(x^*\!-x)+L'_y(x,y,z,\lambda)(y^*\!-y)+L'_z(x,y,z,\lambda)(z^*\!-z) \nonumber\\
        =&L(x,y,z,\lambda_z^*)+\Bar{l}_z'(x,y,\lambda)^{\top}(\Tilde{u}_z-u)+L'_z(x,y,z,\lambda)(z^*-z) \nn\\
        \overset{(f)}{\geq}&\psi(z)+\Bar{l}_z'(x,y,\lambda)^{\top}(\Tilde{u}_z-u)+L'_z(x,y,z,\lambda)(z^*-z),
        \label{eq:subdiff}
    \end{align}
    where $(a)$ follows from the definition of $\psi^*$, $(b)$ holds because the infimum of $\max_{\Tilde{\lambda}\in\Lambda}L(x,y,z,\Tilde{\lambda})$ for $x\in X,y\in\R,z>0$ is approached as  $(x,y,z)\to (x^*,y^*,z^*)$, $(c)$ is satisfied due to $y^*=\E[F(x^*,\xi)]$ by~\eqref{def:yzstar}, $(d)$ utilizes the joint convexity of $L(x,y,z,\lambda)$ with respect to $(x,y,z)$, $(e)$ holds due to the linearity of $L(x,y,z,\lambda)$ with respect to $\lambda$, and $(f)$ follows from $L(x,y,z,\lambda_z^*)\geq\min_{x\in X,y\in Y}L(x,y,z,\lambda_z^*)=L(x_z^*,y_z^*,z,\lambda_z^*)=\psi(z)$. For any $z>0$, by taking the convex combination of the inequalities~\eqref{eq:subdiff} with $u=u_t$ and the weights $\gamma_t/(\tsum_{t=0}^{T-1}\gamma_t)$ for $t=0,1,...,T-1$, we conclude that
    \begin{align}
        \psi^*\geq&\psi(z)+\left(\tsum_{t=0}^{T-1}\gamma_t\right)^{-1}\tsum_{t=0}^{T-1}\gamma_t\Bar{l}'_z(u_t)^{\top}(\Tilde{u}_z-u_t)+\left[\left(\tsum_{t=0}^{T-1}\gamma_t\right)^{-1}\tsum_{t=0}^{T-1}\gamma_tL'_z(x_t,y_t,z,\lambda_t)\right](z^*-z), \nn\\
        \geq&\psi(z)-\left(\tsum_{t=0}^{T-1}\gamma_t\right)^{-1}\left[\max_{u\in U}\tsum_{t=0}^{T-1}\gamma_t\Bar{l}'_z(u_t)^{\top}(u_t-u)\right]+\E[\zeta_T\mid u_0,...,u_{T-1}](z^*-z),
        \label{eq:subdiff_iter}
    \end{align}
    where $\E[\zeta_T|u_0,...,u_{T-1}]$ is taken with respect to the randomness of the samples $\xi'_0,...,\xi'_{T-1}$.

    Below we provide the probability guarantees for two events. Let us first consider the following event
    \begin{align}
        A:=\left\{\max_{u\in U}\left(\tsum_{t=0}^{T-1}\gamma_t\right)^{-1}\left(\tsum_{t=0}^{T-1}\gamma_t\Bar{l}'_z(u_t)^{\top}(u_t-u)\right)\leq\tfrac{\epsilon}{4}\right\}.
        \label{def:eventA}
    \end{align}
    For any $\alpha\in(0,1)$, since $T\geq 1296D_{V,x_0}^2M_z^2/(\alpha^2\epsilon^2)$ and $\gamma_t=D_{V,x_0}/(M_z\sqrt{T})$ hold in Algorithm~\ref{algo:SMD} by~\eqref{Tgamma1}, according to Lemma~\ref{lem:SMDconverge}, we have
    \begin{align}
        \left(\tsum_{t=0}^{T-1}\gamma_t\right)^{-1}\E\left[\max_{u\in U}\tsum_{t=0}^{T-1}\gamma_t\Bar{l}'_z(u_t)^{\top}(u_t-u)\right]\leq\tfrac{\alpha\epsilon}{8}.
        \label{eq1lem2algo}
    \end{align}
    Then, it follows that
    \begin{align}
        \P(A)\geq1-\tfrac{4}{\epsilon}\E\left[\max_{u\in U}\left(\tsum_{t=0}^{T-1}\gamma_t\right)^{-1}\left(\tsum_{t=0}^{T-1}\gamma_t\Bar{l}'_z(u_t)^{\top}(u_t-u)\right)\right]\geq1-\tfrac{4}{\epsilon}\times\tfrac{\alpha\epsilon}{8}=1-\tfrac{\alpha}{2},
        \label{eq:prob1}
    \end{align}
    where the first inequality uses Markov's inequality, and the second inequality follows from~\eqref{eq1lem2algo}. When the event $A$ occurs, according to~\eqref{eq:subdiff_iter}, we have
    \begin{align}
        \psi^*\geq\psi(z)+\E[\zeta_T\mid u_0,...,u_{T-1}](z^*-z)-\tfrac{\epsilon}{4}.
        \label{eq2lem2algo}
    \end{align}
    
    Now we consider the second event
    \begin{align}
        B:=\left\{\zeta_T(z^*-z)\leq\E[\zeta_T\mid\cF_{T-1}](z^*-z)+\tfrac{\epsilon}{4}\right\},
        \label{def:eventB}
    \end{align}
    where for any $s=0,...,T-1$,
    \begin{align}
        \cF_s:=\sigma(u_t\mid t=0,...,s)
        \label{def:calF}
    \end{align}
    is defined as the $\sigma$-algebra associated with the probability space supporting the random process in Algorithm~\ref{algo:SMD}. Given any $x\in X,y\in Y,z>0$ and $\lambda\in\Lambda$, we have
    \begin{align}
        &\mathrm{Var}\left(\cL'_z(x,y,z,\lambda,\xi)\right)=\mathrm{Var}\left(-\tfrac{c}{z^2}(F(x,\xi)-y)_+^2\right)\leq\tfrac{c^2}{z^4}\E\left[(F(x,\xi)-y)_+^4\right] \nn\\
        &\leq\tfrac{8c^2}{z^4}\left\{\E\left[(F(x,\xi)-f(x))_+^4\right]+(f(x)-y)^4\right\}\leq\tfrac{8c^2}{z^4}\left[M_f^4+16(L_fD_X+1)^4\right]=\Tilde{M}_z^4,
        \label{eq2.5lem2algo}
    \end{align}
    where the second inequality utilizes $(F(x,\xi)-y)_+\leq (F(x,\xi)-f(x))_++|f(x)-y|$ and Young's inequality, the third inequality follows from Assumption~\ref{assumption1},~\eqref{eq:Ycontainy} and~\eqref{eq:ybound}, and the last equality applies the definition of $\Tilde{M}_z^4$ in~\eqref{def:tildeMz}. Then it follows that
    \begin{align}
        &\mathrm{Var}(\zeta_T\mid\cF_{T-1})=\mathrm{Var}\left(\tfrac{\tsum_{t=0}^{T-1}\gamma_t\cL'_z(x_t,y_t,z,\lambda_t,\xi'_t)}{\tsum_{t=0}^{T-1}\gamma_t}\mid\cF_{T-1}\right)=\tfrac{\tsum_{t=0}^{T-1}\gamma_t^2\mathrm{Var}\left(\cL'_z(x_t,y_t,z,\lambda_t,\xi'_t)\mid\cF_{T-1}\right)}{\left(\tsum_{t=0}^{T-1}\gamma_t\right)^2} \nn\\
        &\leq\tfrac{\tsum_{t=0}^{T-1}\gamma_t^2}{\left(\tsum_{t=0}^{T-1}\gamma_t\right)^2}\max_{t=0,...,T-1}\left\{\mathrm{Var}\left(\cL'_z(x_t,y_t,z,\lambda_t,\xi'_t)\mid\cF_{T-1}\right)\right\}\leq\tfrac{1}{T}\Tilde{M}_z^4\leq\tfrac{\alpha\epsilon^2}{32(z^*-z)^2},
        \label{eq3lem2algo}
    \end{align}
    where the second equality holds since $\xi_t',t=0,...,T-1$ are independent samples, the second inequality utilizes~\eqref{eq2.5lem2algo},$\gamma_t=D_{V,x_0}/(M_z\sqrt{T})$ and the fact that the samples $\xi_t'$ are independent of $\xi_t$ for $t=0,...,T-1$, and the third inequality follows from $T\geq 32(z^*-z)^2\Tilde{M}_z^4/(\alpha\epsilon^2)$ by~\eqref{Tgamma1}. Therefore, we have
    \begin{align}
        \P(B\mid\cF_{T-1})&=1-\P\left(\zeta_T(z^*-z)-\E[\zeta_T](z^*-z)>\tfrac{\epsilon}{4}\mid\cF_{T-1}\right)\geq 1-\tfrac{\mathrm{Var}(\zeta_T(z^*-z)\mid\cF_{T-1})}{\mathrm{Var}(\zeta_T(z^*-z)\mid\cF_{T-1})+\epsilon^2/16} \nn\\
        &=1-\tfrac{(z^*-z)^2\mathrm{Var}(\zeta_T\mid\cF_{T-1})}{(z^*-z)^2\mathrm{Var}(\zeta_T\mid\cF_{T-1})+\epsilon^2/16}\geq 1-\tfrac{\alpha\epsilon^2/32}{\alpha\epsilon^2/32+\epsilon^2/16}=1-\tfrac{\alpha}{\alpha+2}\geq 1-\tfrac{\alpha}{2},
        \label{eq:prob2}
    \end{align}
    where the first inequality utilizes Cantelli's inequality, and the second inequality uses~\eqref{eq3lem2algo}. Since the event $A$ is determined by $u_0,u_1,...,u_{T-1}$, then in view of~\eqref{def:calF}, we have $A\in\cF_{T-1}$, and it follows from~\eqref{eq:prob2} and the law of total expectations that
    \begin{align}
        \P(B\mid A)=\E[\P(B\mid\cF_{T-1})\mid A]\geq 1-\tfrac{\alpha}{2}.
        \label{eq:prob2BonA}
    \end{align}
    
    Finally, when both the events $A$ and $B$ occur, in view of~\eqref{eq2lem2algo} and~\eqref{def:eventB}, we have
    \begin{align*}
        \psi^*\geq\psi(z)+\left[\E[\zeta_T\mid u_0,...,u_{T-1}](z^*-z)+\tfrac{\epsilon}{4}\right]-\tfrac{\epsilon}{2}\geq\psi(z)+\zeta_T(z^*-z)-\tfrac{\epsilon}{2},
    \end{align*}
    which occurs with probability at least $\P(A)\P(B\mid A)\geq(1-\alpha/2)^2\geq 1-\alpha$ by concluding~\eqref{eq:prob1} and~\eqref{eq:prob2BonA}.
\end{proof}

\vgap

We remark here that the samples $\xi'_t$ are necessary to be independent of $\xi_t$, due to the computation of the conditional probability $\P(B\mid A)$. The event $A$ relies on how accurately we solve the inner problem~\eqref{problem:inner}, which occurs with probability at least $1-\alpha/2$ in~\eqref{eq:prob1} when $T\geq \cO(\alpha^{-2}(1+z^{-2})\epsilon^{-2})$ holds, while the event $B$ relies on how precisely we estimate $\E[\zeta_T\mid\cF_{T-1}]$, which occurs with probability at least $1-\alpha/2$ in~\eqref{eq:prob2} when $T\geq \cO(\alpha^{-1}(z^*-z)^2z^{-4}\epsilon^{-2})$ is satisfied.

Lemma~\ref{lem:subdiff} 
provides a certain probability guarantee for $\psi^*\geq\psi(z)+\zeta_T(z^*-z)-\epsilon/2$ which will enable us to develop a probabilistic line search procedure to find $z^*$. Specifically, for any $z>0$ such that
$\psi(z)-\psi^*>\epsilon/2$ holds ($z\neq z^*$), suppose that $\psi^*\geq\psi(z)+\zeta_T(z^*-z)-\epsilon/2$ is satisfied, it follows that $\zeta_T(z-z^\star)\geq\psi(z)-\psi^\star-\epsilon/2>0$, thus we have
\begin{align}
    z > z^* \quad \Leftrightarrow \quad \zeta_T>0.
    \label{relation_zeta_z}
\end{align}
Therefore, for any $z>0$ such that $\psi(z)-\psi^*>\epsilon/2$, the positivity or negativity of $\zeta_T$ computed by Algorithm~\ref{algo:SMD} 
indicates whether $z^*>z$ or $z^*<z$. We will show the details of our search procedure based on~\eqref{relation_zeta_z} in the next subsection.

\subsection{Probabilistic Bisection Method with Simple Stopping Criterion}

In view of~\eqref{l2minminmaxproblem} and the relation in~\eqref{relation_zeta_z}, the outer layer of our method solves the following one-dimensional convex minimization problem
\begin{align}
    \inf_{z>0}\psi(z)
    \label{problem:outer}
\end{align}
using a bisection procedure, utilizing Algorithm~\ref{algo:SMD} as a subroutine to compute the approximate subgradients of $\psi(z)$. We present this probabilistic bisection method in Algorithm~\ref{algo:PB}.

\begin{algorithm}
\caption{The Probabilistic Bisection Method}
\label{algo:PB}
\begin{algorithmic}
\State{\textbf{Input:} $u_0 = (x_0, y_0, 0)\in U$ with $y_0=\tsum_{i=1}^mF(x_0,\xi_i)\big/m$, $\theta>0$, the initial interval $[a_0,b_0]$ with $a_0=0$ and $b_0\geq\max\{2M_f,\theta\}$.}
\For{$k=0,1,...$}
\State{Set $z_k=(a_k+b_k)/2$.}
\State{Set $(\Bar{u}_k,\zeta_k)$ as output of Algorithm~\ref{algo:SMD} applied to problem~\eqref{problem:inner} with input $u_0,z_k,T_k,\{\gamma_{k,t}\}_{t=0}^{T_k-1}$.}
\If{$\zeta_k>0$}
\begin{equation}
    a_{k+1}=a_k,b_{k+1}=z_k.
    \label{eq1:algo}
\end{equation}
\Else
\begin{equation}
    a_{k+1}=z_k,b_{k+1}=b_k.
    \label{eq2:algo}
\end{equation}
\EndIf
\EndFor
\end{algorithmic}
\end{algorithm}

Algorithm~\ref{algo:PB} conducts the bisection search for $z^*$, starting from the initial interval $[a_0,b_0]$ with $a_0=0$ and $b_0\geq \max\{2M_f,\theta\}$, where $M_f$ can be found in Assumption~\ref{assumption1} and $\theta$ is an input of the algorithm. Below we first explain the reasoning behind the setting $b_0\geq 2M_f$. By the definition of $z^*$ in~\eqref{def:yzstar} and  Assumption~\ref{assumption1}, we have
\begin{align}
    z^*=2\E^{\tfrac{1}{2}}\left[(F(x^*,\xi)-\E[F(x^*,\xi)])_+^2\right]\leq2\E^{\tfrac{1}{4}}\left[(F(x^*,\xi)-\E[F(x^*,\xi)])_+^4\right]\leq 2M_f.
    \label{eq:z*bound}
\end{align}
Therefore, picking $b_0\geq 2M_f$ ensures that $z^*\in [a_0,b_0]$. Moreover, since the value of $M_f$ depends on the specific problem instance~\eqref{l2problem} and could be arbitrarily close to $0$, we set $\theta$ as the other lower bound for $b_0$. Then we iteratively shrink the interval $[a_k,b_k]$ that brackets $z^*$ through probabilistic bisections. Based on Lemma~\ref{lem:subdiff} and the discussion above the relation~\eqref{relation_zeta_z}, the relations $z_k>z^*$ and $\zeta_k>0$ are equivalent under certain probability guarantee
when $\psi(z_k)-\psi^*>\epsilon/2$ holds. Therefore,
if $\zeta_k>0$, we set the next interval $[a_{k+1},b_{k+1}]=[a_k,z_k]$. Otherwise, if $\zeta_k \leq 0$, we take $[a_{k+1},b_{k+1}]=[z_k,b_k]$. Thus, $z^*\in[a_k,b_k]$ holds under certain probability guarantee, until $\psi(z_k)-\psi^*\leq\epsilon/2$ is satisfied.

Notice that in Algorithm~\ref{algo:PB}, we do not specify any termination criterion. We introduce a simple stopping criterion in this section. Specifically, Algorithm~\ref{algo:PB} terminates when
\begin{align}
    b_{k'+1}<\theta
    \label{simplestopping}
\end{align}
is satisfied for some $k'\in\{0,1,...\}\cup\{+\infty\}$.
An enhanced stopping criterion will be discussed in the next subsection. Throughout the remaining part of this section,
we say an outer iteration in Algorithm~\ref{algo:PB} occurs whenever $k$ increases by $1$, or equivalently, Algorithm~\ref{algo:SMD} is called. 

To carry out the convergence analysis for Algorithm~\ref{algo:PB}, we first state in Lemma~\ref{lem:opt} below two sufficient conditions for attaining an $\epsilon$-optimal solution for problem~\eqref{l2problem}. Later we will show in Theorem~\ref{thm:converge1} that at least one of these conditions holds under certain probability guarantees.

\begin{lemma}
    Under Assumption~\ref{assumption2}, and conditioning on~\eqref{eq:Ycontainy}, suppose that we apply the stopping criterion~\eqref{simplestopping} and take
    \begin{align}
        \theta=\tfrac{\epsilon}{c}, \quad T_k\geq\tfrac{81D_{V,x_0}^2M_{z_k}^2}{\epsilon^2} \quad \text{and} \quad \gamma_{k,t}=\tfrac{D_{V,x_0}}{M_{z_k}\sqrt{T_k}}
        \label{Tgamma3}
    \end{align}
    in Algorithm~\ref{algo:PB}, where $D_{V,x_0}$ and $M_{z_k}$ are defined in~\eqref{def:D} and~\eqref{def:M^2}, respectively. Then
    an $\epsilon$-optimal solution for problem~\eqref{l2problem}, i.e.,
    a point $\bar x \in X$ s.t. $\E[h(\bar x)]-h(x^*)\leq\epsilon$, is attained under either of the following two conditions:
    \begin{align*}
        \textit{(a)} \ \psi(z_k)-\psi^*\leq\tfrac{\epsilon}{2}, \qquad \textit{(b)} \ \text{Algorithm~\ref{algo:PB} terminates with} \ z^*\leq z_k=b_{k+1}<\theta.
    \end{align*}
    \label{lem:opt}
\end{lemma}
\begin{proof}
    Let $\Bar{u}_k=(\Bar{x}_k,\Bar{y}_k,\Bar{\lambda}_k)$ in Algorithm~\ref{algo:PB}. According to Lemma~\ref{lem:SMDconverge}, if $T_k$ and $\gamma_{k,t}$ are set to~\eqref{Tgamma3}, then
    \begin{align}
        \E[h(\Bar{x}_k)]-\psi(z_k)&=\E\left[\inf_{y\in Y,z>0}L(\Bar{x}_k,y,z,0)-\min_{x\in X,y\in Y}\max_{\lambda\in\Lambda}L(x,y,z_k,\lambda)\right] \nn\\
        &\leq\E\left[\max_{\lambda\in\Lambda}L(\Bar{x}_k,\Bar{y}_k,z_k,\lambda)-\min_{x\in X,y\in Y}L(x,y,z_k,\Bar{\lambda}_k)\right]\leq\tfrac{\epsilon}{2},
        \label{eq1lem3algo}
    \end{align}
    where the equality follows from Lemma~\ref{lem:phi}, $L(\Bar{x}_T,y,z,0)=\phi(\Bar{x}_T,y,z)$ and the definition of $\psi(z)$ in~\eqref{l2minminmaxproblem}, the first inequality utilizes $z_k>0$ and the feasibility of $(\Bar{x}_T,\Bar{y}_T,\Bar{\lambda}_T)$, and the second inequality follows from Lemma~\ref{lem:SMDconverge}. Therefore, when $\psi(z_k)-\psi^*\leq\epsilon/2$ holds, we have
    \begin{align}
        \E[h(\Bar{x}_k)]-h(x^*)=\left\{\E[h(\Bar{x}_k)]-\psi(z_k)\right\}+\left[\psi(z_k)-\psi^*\right]\leq\tfrac{\epsilon}{2}+\tfrac{\epsilon}{2}=\epsilon,
        \label{eq:opt1}
    \end{align}
    where the inequality follows from~\eqref{eq1lem3algo}. When Algorithm~\ref{algo:PB} terminates with $z^*\leq z_k=b_{k+1}<\epsilon/c$, we have
    \begin{align}
        \E[h(\Bar{x}_k)]-h(x^*)&\leq \E\left[\max_{\lambda\in\Lambda}L(\Bar{x}_k,\Bar{y}_k,z_k,\lambda)\right]-h(x^*)\leq\left[\min_{x\in X,y\in Y}L(x,y,z_k,\Bar{\lambda}_k)+\tfrac{\epsilon}{2}\right]-h(x^*) \nn\\
        &\leq L(x^*,y^*,z_k,0)-\lim_{z\to z^*}L(x^*,y^*,z,0)+\tfrac{\epsilon}{2}=\tfrac{c}{4}\left(\tfrac{z^{*2}}{z_k}-z^*\right)+\tfrac{c}{4}(z_k-z^\star)+\tfrac{\epsilon}{2} \nn\\
        &\leq\tfrac{c}{4}z_k+\tfrac{c}{4}z_k+\tfrac{\epsilon}{2}<\tfrac{\epsilon}{2}+\tfrac{\epsilon}{2}=\epsilon,
        \label{eq:opt2}
    \end{align}
    where the first inequality follows from Lemma~\ref{lem:phi} and $\phi(\Bar{x}_k,\Bar{y}_k,z_k)=L(\Bar{x}_k,\Bar{y}_k,z_k,0)$, the second inequality applies~\eqref{eq1lem3algo}, the third inequality holds due to $x^* \in X$, $y^*=\E[F(x^*,\xi)]\in Y$ by~\eqref{def:yzstar} and~\eqref{eq:Ycontainy}, Lemma~\ref{lem:phi}, as well as $\lim_{z\to z^*}\phi(x^*,y^*,z)=\lim_{z\to z^*}L(x^*,y^*,z,0)$, and the rest of equations follow from~\eqref{def:yzstar} and $z^*\leq z_k<\epsilon/c$. Combining~\eqref{eq:opt1} and~\eqref{eq:opt2}, $\Bar{x}_k$ is an $\epsilon$-optimal solution for problem~\eqref{l2problem} in both cases.
\end{proof}
\vgap

Here we add some remarks to explain condition (b) in Lemma~\ref{lem:opt}, which can be divided into three relations. First, $b_{k+1}<\theta$ is the stopping criterion~\eqref{simplestopping} in Algorithm~\ref{algo:PB}. Second, when Algorithm~\ref{algo:PB} terminates with $b_{k+1}<\theta$, we must have $z_k=b_{k+1}$; otherwise, if $z_k=a_{k+1}$, then $b_{k+1}=b_k$ by~\eqref{eq2:algo}, and Algorithm~\ref{algo:PB} should have stopped before the outer iteration $k$ because $b_k<\theta$. Moreover, we will show that $z^*\in[a_{k+1},b_{k+1}]$ holds under certain probability guarantee in Lemma~\ref{lem:induction1} below, which indicates $z^*\leq b_{k+1}=z_k$.

Before moving to the convergence of Algorithm~\ref{algo:PB}, we provide a probability guarantee result for the bisection process in Lemma~\ref{lem:induction1} below, by extending the probability guarantee in Lemma~\ref{lem:subdiff} across multiple outer iterations in Algorithm~\ref{algo:PB}. For this purpose, let us assume that in Algorithm~\ref{algo:PB}, an $\epsilon$-optimal solution $\Bar{x}_{k_0}$ for problem~\eqref{l2problem}, i.e., $\E[h(\bar{x}_{k_0})]-h(x^*)\leq\epsilon$ is found in some outer iteration $k_0\in\{0,1,...\}\cup\{+\infty\}$. 

\begin{lemma}
    Under Assumptions~\ref{assumption1} and~\ref{assumption2}, and conditioning on~\eqref{eq:Ycontainy}, for any $\alpha\in(0,1)$, assume that the algorithmic parameters in Algorithm~\ref{algo:PB} are set to
    \begin{align}
        \theta=\tfrac{\epsilon}{c}, \quad T_k\geq\max\left\{\left(\tfrac{36}{\alpha}D_{V,x_0}M_{z_k}\right)^2,\tfrac{32}{\alpha}\Bar{M}_{z_k}^2\right\}\epsilon^{-2} \quad \text{and} \quad \gamma_{k,t}=\tfrac{D_{V,x_0}}{M_{z_k}\sqrt{T_k}},
        \label{Tgammaleminduct1}
    \end{align}
    where
    \begin{align}
        \Bar{M}_z^2:=\tfrac{8c^2}{z^2}\left[M_f^4+16(L_fD_X+1)^4\right],
        \label{def:barMz}
    \end{align}
    and $D_{V,x_0}$ and $M_{z_k}$ are defined in~\eqref{def:D} and~\eqref{def:M^2}, respectively. Then, for any $k\in\{0,...,k_0\}$, with probability at least $(1-\alpha)^k$, we have
    \begin{align}
        z^*\in[a_s,b_s] \quad \text{and} \quad z_s\geq\tfrac{z^*}{2} \quad \text{for} \quad s\in\{0,...,k\}.
        \label{eq:induction1}
    \end{align}
    \label{lem:induction1}
\end{lemma}
\begin{proof}
    First, in view of the parameter setting in~\eqref{Tgammaleminduct1}, the condition~\eqref{Tgamma3} in Lemma~\ref{lem:opt} is satisfied. Then by Lemma~\ref{lem:opt}, we must have $\psi(z_k)-\psi^*>\epsilon/2$ for $k\in\{0,...,k_0-1\}$.
    
    Now we prove the statement~\eqref{eq:induction1} through an inductive argument. Clearly, since $z^*\in[a_0,b_0]$ and $z_0=b_0/2\geq z^*/2$,~\eqref{eq:induction1} holds for $k=0$. Next, assume that~\eqref{eq:induction1} is true for some $k\in\{0,...,k_0-1\}$, it then suffices to show that $z^*\in[a_{k+1},b_{k+1}]$ and $z_{k+1}\geq z^*/2$ are satisfied with probability at least $1-\alpha$. 
    Prior to this, we first demonstrate that when~\eqref{eq:induction1} holds for $k$,
    \begin{align}
        \psi^*\geq\psi(z_k)+\zeta_k(z^*-z_k)-\tfrac{\epsilon}{2}
        \label{eq:induction2}
    \end{align}
    is satisfied with probability at least $1-\alpha$. In view of~\eqref{Tgamma1} in Lemma~\ref{lem:subdiff},~\eqref{eq:induction2} holds if
    \begin{align}
        T_k\geq\max\left\{\left(\tfrac{36}{\alpha}D_{V,x_0}M_{z_k}\right)^2,\tfrac{32}{\alpha}(z^*-z_k)^2\Tilde{M}_{z_k}^4\right\}\epsilon^{-2} \quad \text{and} \quad \gamma_k=\tfrac{D_{V,x_0}}{M_{z_k}\sqrt{T_k}}.
        \label{leminductionTgamma}
    \end{align}
    Comparing~\eqref{leminductionTgamma} with the parameter settings in~\eqref{Tgammaleminduct1}, it suffices to show $\Bar{M}_{z_k}^2\geq(z^*-z_k)^2\Tilde{M}_{z_k}^4$ (see the definitions of $\Bar{M}_z^2$ and $\Tilde{M}_z^4$ in~\eqref{def:barMz} and~\eqref{def:tildeMz} respectively).
    First observe that 
    \begin{equation}
        \tfrac{(z^*-z_k)^2}{z_k^2}\leq 1.
        \label{eq:z*zk}
    \end{equation}
    To see~\eqref{eq:z*zk}, we consider two cases. Suppose that $z^*\leq z_k$, then we have $(z^*-z_k)^2\leq z_k^2$ and~\eqref{eq:z*zk} follows. Now we suppose $z^*>z_k$. Since $z_k\geq z^*/2$ holds by our inductive hypothesis in~\eqref{eq:induction1}, it follows that $(z^*-z_k)^2\leq (z^*/2)^2$ and hence $(z^*-z_k)^2/z_k^2\leq (z^*/2)^2/z_k^2\leq 1$. In view of the definitions of $\Bar{M}_z^2$ and $\Tilde{M}_z^4$ in~\eqref{def:barMz} and~\eqref{def:tildeMz}, it follows from~\eqref{eq:z*zk} that
    \begin{align}
        \Bar{M}_{z_k}^2=\tfrac{8c^2}{z_k^2}\left[M_f^4+16(L_fD_X+1)^4\right]\geq (z^*-z_k)^2\tfrac{8c^2}{z_k^4}\left[M_f^4+16(L_fD_X+1)^4\right]=(z^*-z_k)^2\Tilde{M}_{z_k}^4.
        \label{Mzbartilde_relation}
    \end{align}
    Hence,~\eqref{eq:induction2} holds with probability at least $1-\alpha$. As a result of $\psi(z_k)-\psi^*>\epsilon/2$ and~\eqref{eq:induction2}, the relations $z_k>z^*$ and $\zeta_k>0$ are equivalent (see the discussion above~\eqref{relation_zeta_z}). Then, by construction of the bisection process (see the discussion below Algorithm~\ref{algo:PB}), $z^*\in[a_{k+1},b_{k+1}]$ is satisfied, which implies $z_{k+1}=(a_{k+1}+b_{k+1})/2\geq b_{k+1}/2\geq z^*/2$. Therefore, we complete the induction.
\end{proof}
\vgap

With the help of Lemma~\ref{lem:opt} and Lemma~\ref{lem:induction1}, we are now ready to establish the convergence of Algorithm~\ref{algo:PB}.
\begin{theorem}
    Under Assumptions~\ref{assumption1} and~\ref{assumption2}, for any $\alpha_0\in (0,1)$ and $\alpha\in(0,1-\alpha_0)$, let
    \begin{align}
        K=\left\lceil\log_2\left\{\tfrac{8cb_0}{\epsilon}\left[\tfrac{\beta^2+4(L_fD_X+1)^2}{\max\{z^*/2,\epsilon/2c\}^2}+1\right]\right\}\right\rceil
        \label{Kthm1}
    \end{align}
    denote the maximum number of outer iterations in Algorithm~\ref{algo:PB}. Also assume that the stopping criterion~\eqref{simplestopping} is applied in Algorithm~\ref{algo:PB}, and the algorithmic parameters in Algorithm~\ref{algo:PB} are set to
    \begin{align}
        m=\left\lceil\tfrac{\beta^2}{\alpha_0}\right\rceil, \quad \theta=\tfrac{\epsilon}{c}, \quad T_k=\left\lceil\max\left\{\left(\tfrac{36K}{\alpha}D_{V,x_0}M_{z_k}\right)^2,\tfrac{32K}{\alpha}\Bar{M}_{z_k}^2\right\}\epsilon^{-2}\right\rceil \quad \text{and} \quad \gamma_{k,t}=\tfrac{D_{V,x_0}}{M_{z_k}\sqrt{T_k}},
        \label{Tgammathm1}
    \end{align}
    where $D_{V,x_0},M_{z_k}$ and $\Bar{M}_{z_k}$ are defined in~\eqref{def:D},~\eqref{def:M^2} and~\eqref{def:barMz}, respectively.
    Then an $\epsilon$-optimal solution for problem~\eqref{l2problem} can be attained within at most $m+2K\Tilde{T}$ number of oracle calls with probability at least $1-\alpha_0-\alpha$, where
    \begin{align*}
        \Tilde{T}=\left\lceil\max\left\{\left(\tfrac{36K}{\alpha}D_{V,x_0}M_{z_{\min}}\right)^2,\tfrac{32K}{\alpha}\Bar{M}_{z_{\min}}^2\right\}\epsilon^{-2}\right\rceil \quad \text{and} \quad z_{\min}=\max\left\{\tfrac{z^*}{2},\tfrac{\epsilon}{2c}\right\}.
    \end{align*}
    \label{thm:converge1}
\end{theorem}
\begin{proof}
    We first condition on the event~\eqref{eq:Ycontainy}. In view of the parameter setting in~\eqref{Tgammathm1} and according to Lemma~\ref{lem:induction1}, with probability at least $(1-\alpha/K)^{k_0}$, we have
    \begin{align}
        z^*\in[a_k,b_k] \quad \text{and} \quad z_k\geq\tfrac{z^*}{2} \quad \text{for} \quad k\in\{0,...,k_0\}.
        \label{eq:thm1induction1}
    \end{align}
    Now let us assume~\eqref{eq:thm1induction1} holds, and let $k'$ be defined in~\eqref{simplestopping}. Then, in view of the discussion right after Lemma~\ref{lem:opt}, we must have 
    \begin{align}
        z_{k'}&=b_{k'+1}, \label{eq:zkbk+1}\\
        b_k &\geq\theta, \, \forall k\in\{0,...,k'\}. \label{eq:bktheta}
    \end{align}
    Below we show that the following statements are satisfied:
    \begin{align}
        (1) \ k'\geq k_0; \qquad (2) \ z^*\in [a_k,b_k] \quad \text{and} \quad z_k\geq\max\left\{\tfrac{z^*}{2},\tfrac{\epsilon}{2c}\right\} \quad \text{for} \quad k\in\{0,...,k_0\}.
        \label{zregion}
    \end{align}
    As a result of~\eqref{zregion}, $T_k$ defined in~\eqref{Tgammathm1} satisfies $T_k \le \Tilde{T}$ for $k\in\{0,...,k_0\}$, since $M_z^2$ and $\Bar{M}_z^2$ defined in~\eqref{def:M^2} and~\eqref{def:barMz}, respectively, monotonically decrease with respect to $z$.
    
    To verify the first claim in~\eqref{zregion}, suppose for contraction that $k'<k_0$. Then since $z^*\in[a_s,b_s]$ holds for $s\in\{0,...,k_0\}$ by~\eqref{eq:thm1induction1}, we have $z^*\in[a_{k'+1},b_{k'+1}]$, which in view of \eqref{eq:zkbk+1} and $b_{k'+1}<\theta$, implies  
     $z^*\leq z_{k'}=b_{k'+1}<\theta$. 
     This observation together with Lemma~\ref{lem:opt} implies that, an $\epsilon$-optimal solution for problem~\eqref{l2problem} has been found in the outer iteration $k'$ or earlier, i.e., $k_0 \le k'$ and thus lead to a contradiction. Furthermore, the second claim in~\eqref{zregion}
     follows from \eqref{eq:thm1induction1},
     the previous conclusion $k'\geq k_0$, and the observation 
     \begin{align}
         z_k\geq\tfrac{\epsilon}{2c} \quad \text{for} \quad k\in\{0,...,k'\},
         \label{eq:zkboundeps}
     \end{align}
     which directly follows from $z_k=(a_k+b_k)/2\geq b_k/2\geq\theta/2=\epsilon/(2c),\forall k\in\{0,...,k'\}$ by utilizing~\eqref{eq:bktheta}.
    
    It now remains to establish an upper bound on $k_0$. First, let us consider the case $z^*\geq\epsilon/(2c)$. In this case, we utilize the Lipschitz continuity of $\psi(z)$, i.e., for any $z>0$ and $(x,y,\lambda) \in X\times Y \times \Lambda$, 
    \begin{align}
        \left|L'_z(x,y,z,\lambda)\right|&=\left|-\tfrac{c}{z^2}\E\left[(F(x,\xi)-y)_+^2\right]+\tfrac{c}{4}\right|\leq\tfrac{2c}{z^2}\E\left[(F(x,\xi)-f(x))_+^2\right]+\tfrac{2c}{z^2}(f(x)-y)^2+\tfrac{c}{4} \nn\\
        &\leq\tfrac{2c}{z^2}\beta^2+\tfrac{2c}{z^2}[2(L_fD_X+1)]^2+\tfrac{c}{4}=\tfrac{2c}{z^2}\left[\beta^2+4(L_fD_X+1)^2\right]+\tfrac{c}{4},
        \label{eq1thm1}
    \end{align}
    where the first inequality follows from $(F(x,\xi)-y)_+\leq (F(x,\xi)-f(x))_++|f(x)-y|$ and Young's inequality, and the second inequality applies~\eqref{assum:fbddvar},~\eqref{eq:Ycontainy} and~\eqref{eq:ybound}. Then, for $k\in\{0,...,k_0-1\}$, we have
    \begin{align}
        \tfrac{\epsilon}{2}&<\psi(z_k)-\psi(z^*)=\min_{x\in X,y\in Y}\max_{\lambda\in\Lambda}L(x,y,z_k,\lambda)-L(x^*,y^*,z^*,\lambda^*)\leq \max_{\lambda\in\Lambda}L(x^*,y^*,z_k,\lambda)-L(x^*,y^*,z^*,\lambda^*) \nn\\
        &=L(x^*,y^*,z_k,0)-L(x^*,y^*,z^*,0)\leq\left(\tfrac{2c\left[\beta^2+4(L_fD_X+1)^2\right]}{\max\{z^*/2,\epsilon/2c\}^2}+\tfrac{c}{4}\right)|z_k-z^*| \nn\\
        &\leq\left(\tfrac{2c\left[\beta^2+4(L_fD_X+1)^2\right]}{\max\{z^*/2,\epsilon/2c\}^2}+\tfrac{c}{4}\right)\tfrac{b_0}{2^{k+1}},
        \label{eq2thm1}
    \end{align}
    where the first equality follows from the definition of $\psi(z)$ in~\eqref{l2minminmaxproblem}, the second equality holds because $y^*=\E[F(x^*,\xi)]$, the third inequality utilizes $z_k>\max\{z^*/2,\epsilon/(2c)\}$ by~\eqref{zregion} and the Lipschitz constant of $L(x,y,z,\lambda)$ with respect to $z$ over $[\epsilon/2c,\infty)$ given by~\eqref{eq1thm1}, and the fourth inequality is satisfied due to
    \begin{align*}
        |z_k-z^*|\leq\tfrac{b_k-a_k}{2}=\tfrac{(b_0-a_0)/2^k}{2}=\tfrac{b_0}{2^{k+1}}
    \end{align*}
    by $z_k=(a_k+b_k)/2$ and $z^*\in[a_k,b_k]$ in~\eqref{zregion}. Taking $k=k_0-1$ in~\eqref{eq2thm1}, we have
    \begin{align}
        k_0\leq\log_2\left\{\tfrac{4cb_0}{\epsilon}\left[\tfrac{\beta^2+4(L_fD_X+1)^2}{\max\{z^*/2,\epsilon/2c\}^2}+\tfrac{1}{8}\right]\right\}.
        \label{eq3thm1}
    \end{align}
    Now we consider the second case $z^*<\epsilon/(2c)$. In this case, for $k\in\{0,...,k_0-1\}$, it follows from~\eqref{eq:bktheta} and~\eqref{zregion} that 
    $
    b_k-a_k\geq\theta-z^*>\theta-\epsilon/(2c)=\epsilon/(2c).
    $
    Using the fact that $b_k-a_k=(b_0-a_0)/2^k=b_0/2^k$ and taking $k=k_0-1$, we have
    \begin{align}
        \tfrac{b_0}{2^{k_0-1}}=b_{k_0-1}-a_{k_0-1}>\tfrac{\epsilon}{2c} \quad \Rightarrow \quad k_0<\log\left(\tfrac{4cb_0}{\epsilon}\right).
        \label{eq4thm1}
    \end{align}
    Combining the two cases~\eqref{eq3thm1} and~\eqref{eq4thm1}, we obtain
    \begin{align}
        k_0+1\leq\log_2\left\{\tfrac{8cb_0}{\epsilon}\left[\tfrac{\beta^2+4(L_fD_X+1)^2}{\max\{z^*/2,\epsilon/2c\}^2}+1\right]\right\}\leq K,
        \label{deriveK}
    \end{align}
    and thus~\eqref{zregion} holds with probability at least $(1-\alpha/K)^{k_0}>(1-\alpha/K)^K\geq 1-\alpha$.
    
    Finally, we conclude the proof by noting that, when $m\geq\beta^2/\alpha_0$ samples are used, the event~\eqref{eq:Ycontainy} that we previously conditioned on occurs with probability at least $1-\alpha_0$.
\end{proof}
\vgap

By Theorem~\ref{thm:converge1}, the oracle complexity for Algorithm~\ref{algo:PB} applied to  problem~\eqref{l2problem} is bounded by
\begin{align}
    \Tilde{\cO}\left(\left(C_1\max\{z^*,\epsilon\}^{-2}+C_2\right)\epsilon^{-2}\right),
    \label{complexityresult}
\end{align}
where
\begin{align}
    C_1\!:=\!M_f^4\!+\!\left(\!L_G^2\!+\!L_f^2\beta^2\!+\!L_f^4D_X^2\!+\!L_f^2D_X^2\sigma_f^2\!\right)\!\left(\!D_{v_X\!,x_0}^2\!+\!L_f^2D_X^2\!\right) \quad \text{and} \quad C_2\!:=\!\left(\!\beta^2\!+\!L_f^2D_X^2\!+\!\sigma_f^2\!\right)\!\left(\!D_{v_X\!,x_0}^2\!+\!L_f^2D_X^2\!\right).
    \label{def:C1C2}
\end{align}
Note that the bound in~\eqref{complexityresult} depends on $z^*=2\E^{1/2}[(F(x^*,\xi)-\E[F(x^*,\xi)])_+^2]$, which is twice of the semideviation of $F(x,\xi)$ at the optimal solution $x^*$. Therefore, the oracle complexity is instance-dependent with respect to $x^*$. Moreover, since $z^*\leq\cO(1)$ by~\eqref{eq:z*bound}, we have $\max\{z^*,\epsilon\}^{-2}\geq\cO(1)$, thus the oracle complexity in~\eqref{complexityresult} can be simplified as
\begin{align}
    \tilde{\cO}\left(\max\{z^*,\epsilon\}^{-2}\epsilon^{-2}\right),
    \label{simplifycomplexityresult}
\end{align}
which will be shown to be nearly optimal in Section~\ref{section:lower}.

As a remark, we should emphasize that, Theorem~\ref{thm:converge1} only guarantees that with probability at least $1-\alpha$, we can attain an $\epsilon$-optimal solution $\Bar{x}_{k_0}$ for problem~\eqref{l2problem} in some outer iteration $k_0\leq K-1$. However, Algorithm~\ref{algo:PB} may not terminate in the outer iteration $k_0$. Suppose that Algorithm~\ref{algo:PB} continues in the outer iteration $k_0$, then since $\psi(z_{k_0})-\psi^*>\epsilon/2$ is not guaranteed, the relation $z_{k_0}>z^*\Leftrightarrow\zeta_{k_0}>0$ in~\eqref{relation_zeta_z} may not hold. In this case, we cannot further guarantee $z^*\in[a_{k_0+1},b_{k_0+1}]$ with probability at least $1-\alpha/K$. As a result, the violation of the condition $z^*\in[a_k,b_k]$ for $k>k_0$ may also lead to the violation of the lower bound on $z_k$ in~\eqref{zregion}. Therefore, $z_k$ may continue to decrease when $k\geq k_0$, until Algorithm terminates in the outer iteration $k'\geq k_0$ with $z_{k'}\in\cO(\epsilon)$. In this case, we can only guarantee $T_k\leq\Tilde{\cO}(\epsilon^{-4})$ for $k\in\{0,...,k'\}$, and the worst-case complexity for solving problem~\eqref{l2problem} using Algorithm~\ref{algo:PB} will be $\Tilde{\cO}(\epsilon^{-4})$.

In the next subsection, we will design an enhanced stopping criterion for Algorithm~\ref{algo:PB}. Whenever it occurs that $z_k\leq z^*$ continues to decrease in some outer iterations $k\geq k_0$, this new stopping criterion will be triggered to terminate Algorithm~\ref{algo:PB} to maintain the complexity result in~\eqref{simplifycomplexityresult}.

\subsection{Enhanced Stopping Criterion}

The new stopping criterion for Algorithm~\ref{algo:PB} is defined as follows. Let $K > 0$ be the maximum number of outer iterations,  Algorithm~\ref{algo:PB} will terminate when at least one of the following two conditions is satisfied 
\begin{align}
    &(1) \quad b_{k+1}=z_k<\theta, \label{originalstop} \\
    &(2) \quad \hat{\psi}(b_k)-\hat{\psi}(z_k)<\tfrac{\epsilon}{4} \qquad \text{and} \qquad [a_{k+1},b_{k+1}] =[a_k,z_k],
    \label{stopcriteria}
\end{align}
for some outer iteration $k\in\{0,...,K-2\}$. Here $\hat{\psi}$ is an estimation of the function value $\psi$
given by
\begin{align}
    \hat{\psi}(z_k)=\tfrac{1}{T'_k}\tsum_{i=1}^{T'_k}\cL(\Bar{x}_k,\Bar{y}_k,z_k,\Bar{\lambda}_k,\xi''_{k,i}),
    \label{eq:hatpsi}
\end{align}
where $\xi''_{k,i},i=1,...,T'_k$ are independent of the samples $\xi_{k,t},t=0,...,T_k-1$ and $\xi'_{k,t},t=0,...,T_k-1$ used for computing $\Bar{u}_k=(\Bar{x}_k,\Bar{y}_k,\Bar{\lambda}_k)$ and $\zeta_k$ in Algorithm~\ref{algo:SMD}, respectively. Clearly, we have $\E[\hat{\psi}(z_k)|\Bar{u}_k]=L(\Bar{x}_k,\Bar{y}_k,z_k,\Bar{\lambda}_k)$.
Note that for any $k>0$, there exists $\Bar{k}<k$ such that $z_{\Bar{k}}=b_k$, hence we only need to estimate $\psi(z_k)$ for $k\ge 0$, and to compute $\hat{\psi}(b_0)$ similarly according to~\eqref{eq:hatpsi} before Algorithm~\ref{algo:PB} starts.
Also observe that the condition~\eqref{originalstop} is equivalent to the previous stopping criterion~\eqref{simplestopping}, while ~\eqref{stopcriteria} is first introduced here to improve the worst-case complexity $\tilde{\cO}(\epsilon^{-4})$ for Algorithm~\ref{algo:PB}. 

\vgap

The following result shows the quality of the estimation $\hat{\psi}$.

\begin{lemma}
    Under Assumptions~\ref{assumption1} and~\ref{assumption2}, and conditioning on~\eqref{eq:Ycontainy}, for any $\alpha\in (0,1)$, suppose that 
    \begin{align}
        T_k\geq\left(\tfrac{144D_{V,x_0}M_{z_k}}{\alpha\epsilon}\right)^2,\gamma_{k,t}=\tfrac{D_{V,x_0}}{M_{z_k}T_k} \qquad \text{and} \qquad T'_k\geq\tfrac{1024}{\alpha\epsilon^2}\left\{\tfrac{8c^2}{z_k^2}\left[M_f^2+16(L_fD_X+1)^2\right]+\beta^2\right\},
        \label{stopT}
    \end{align}
    for each outer iteration of Algorithm~\ref{algo:PB}, where $D_{V,x_0}$ and $M_{z_k}$ are defined in~\eqref{def:D} and~\eqref{def:M^2}, respectively, and $T'_k$ is used in~\eqref{eq:hatpsi}. Then for any outer iteration $k$, we have
    \begin{align*}
        \left|\hat{\psi}(z_k)-\psi(z_k)\right|\leq\tfrac{\epsilon}{8}
    \end{align*}
    with probability at least $1-\alpha$.
    \label{lem:stop1}
\end{lemma}
\begin{proof}
    For any outer iteration $k$ in Algorithm~\ref{algo:PB} and any $z_k>0$, to provide the probability guarantee for $|\hat{\psi}(z_k)-\psi(z_k)|\leq\epsilon/8$, we consider the following two events
    \begin{align}
        C:=\left\{\left|\E\left[\hat{\psi}(z_k)\mid\Bar{u}_k\right]-\psi(z_k)\right|\leq\tfrac{\epsilon}{16}\right\} \quad \text{and} \quad D:=\left\{\left|\hat{\psi}(z_k)-\E\left[\hat{\psi}(z_k)\mid\Bar{u}_k\right]\right|\leq\tfrac{\epsilon}{16}\right\},
        \label{def:eventCD}
    \end{align}
    respectively.
    
    Let us first consider the event $C$. Let $u_{k,t}$ denote the iterates in Algorithm~\ref{algo:SMD} called by Algorithm~\ref{algo:PB} in the outer iteration $k$. For any $\alpha>0$, utilizing $T_k\geq [144D_{V,x_0}M_{z_k}/(\alpha\epsilon)]^2$ and $\gamma_{k,t}=D_{V,x_0}/(M_{z_k}T_{z_k})$ by~\eqref{stopT} in Lemma~\ref{lem:SMDconverge}, we have (see the definition of $\Bar{l}'_z$ in~\eqref{def:barlz})
    \begin{align}
        \E\left[\left|\E\left[\hat{\psi}(z_k)\mid\Bar{u}_k\right]-\psi(z_k)\right|\right]\leq\left(\tsum_{t=0}^{T_k-1}\gamma_{k,t}\right)^{-1}\E\left[\max_{u\in U}\tsum_{t=0}^{T_k-1}\gamma_{k,t}\Bar{l}'_{z_k}(u_{k,t})^{\top}(u_{k,t}-u)\right]\leq\tfrac{\alpha\epsilon}{32},
        \label{eq:eventCbound}
    \end{align}
    then by Markov's inequality,
    \begin{align}
        \P(C)=\P\left(\left|\E\left[\hat{\psi}(z_k)\mid\Bar{u}_k\right]-\psi(z_k)\right|\leq\tfrac{\epsilon}{16}\right)\geq 1-\tfrac{16}{\epsilon}\times\tfrac{\alpha\epsilon}{32}=1-\tfrac{\alpha}{2}.
        \label{stopprob1}
    \end{align}
    
    Now we consider the event D. For any given $x\in X,y\in Y,z>0$ and $\lambda\in\Lambda$, we have
    \begin{align}
        &\mathrm{Var}\left(\cL(x,y,z,\lambda,\xi)\right)=\mathrm{Var}\left(\tfrac{c}{z}(F(x,\xi)-y)_+^2+\lambda F(x,\xi)\right)\leq\tfrac{2c^2}{z^2}\E\left[(F(x,\xi)-y)_+^4\right]+2\E\left[(F(x,\xi)-f(x))^2\right] \nn\\
        &\leq\tfrac{16c^2}{z^2}\left\{\E\left[(F(x,\xi)-\E[F(x,\xi)])_+^4\right]+(\E[F(x,\xi)]-y)^4\right\}+2\beta^2\leq\tfrac{16c^2}{z^2}\left[M_f^2+16(L_fD_X+1)^2\right]+2\beta^2,
        \label{eq:stopvar}
    \end{align}
    where the first inequality follows from the Cauchy--Schwarz inequality and $\lambda\in[0,1]$, the second inequality applies $(F(x,\xi)-y)_+\leq (F(x,\xi)-f(x))_++|f(x)-y|$,~\eqref{assum:fbddvar} and Young's inequality, and the third inequality follows from~\eqref{eq:Ycontainy}, ~\eqref{eq:ybound} and Assumption~\ref{assumption1}. Similar to~\eqref{def:calF}, for any $s=0,...,T_k-1$, let us define the $\sigma$-algebra $\cF_{k,s}$ associated with the probability space supporting the random process in the outer iteration $k$ as
    \begin{align}
        \cF_{k,s}:=\sigma(u_{k,t}\mid t=0,...,s).
        \label{def:cFks}
    \end{align}
    Then we have
    \begin{align}
        \P(D\mid\cF_{k,T_k-1})&=\P\left(\left|\hat{\psi}(z_k)-\E\left[\hat{\psi}(z_k)\mid\Bar{u}_k\right]\right|\leq\tfrac{\epsilon}{16}\mid\cF_{k,T_k-1}\right)\geq 1-\tfrac{\mathrm{Var}\left(\hat{\psi}(z_k)\mid\cF_{k,T_k-1}\right)}{(\epsilon/16)^2} \nn\\
        &\geq 1-\tfrac{256}{\epsilon^2T'_k}\mathrm{Var}\left(\cL(\Bar{x}_k,\Bar{y}_k,z_k,\Bar{\lambda}_k,\xi)\mid\cF_{k,T_k-1}\right)\geq 1-\tfrac{\alpha}{2},
        \label{stopprob2}
    \end{align}
    where the first inequality utilizes Chebyshev's inequality and the independence of the samples $\xi_{k,t},t=0,...,T_k-1$ and $\xi''_{k,i},i=1,...,T'_k$, the second inequality follows from~\eqref{eq:hatpsi}, and the third inequality applies the bound on $T'_k$ in~\eqref{stopT} and uses~\eqref{eq:stopvar}. It follows from~\eqref{stopprob2} that
    \begin{align}
        \P(D\mid C)\geq 1-\tfrac{\alpha}{2}.
        \label{eq:probDonC}
    \end{align}
    Combining~\eqref{stopprob1} and~\eqref{eq:probDonC}, we conclude $\P(|\hat{\psi}(z_k)-\psi(z_k)|\leq\epsilon/8)\geq \P(C)\P(D\mid C)\geq (1-\alpha/2)^2\geq 1-\alpha$.
\end{proof}
\vgap

Lemma~\ref{lem:stop1} shows that, for any particular outer iteration $k$, we can guarantee $|\hat{\psi}(z_k)-\psi(z_k)|\leq\cO(\epsilon)$ with probability at least $1-\alpha$ when $T_k,T'_k\geq\cO(\alpha^{-2}(1+z_k^{-2})\epsilon^{-2})$ are both satisfied. 
We now extend the probability guarantee in Lemma~\ref{lem:stop1} and that in Lemma~\ref{lem:subdiff} across multiple outer iterations.

\begin{lemma}
    Under Assumptions~\ref{assumption1} and~\ref{assumption2}, and conditioning on~\eqref{eq:Ycontainy}, for any $\alpha\in(0,1)$, suppose that the algorithmic parameters in Algorithm~\ref{algo:PB} are set to
    \begin{align}
        \theta=\tfrac{\epsilon}{c}, \quad T_k\geq\max\left\{\left(\tfrac{216}{\alpha}D_{V,x_0}M_{z_k}\right)^2,\tfrac{48}{\alpha}\Bar{M}_{z_k}^2\right\}\epsilon^{-2} \quad \text{and} \quad \gamma_{k,t}=\tfrac{D_{V,x_0}}{M_{z_k}\sqrt{T_k}},
        \label{Tkstop2}
    \end{align}
    where $D_{V,x_0}$, $M_{z_k}$ and $\Bar{M}_{z_k}$ are defined in~\eqref{def:D},~\eqref{def:M^2} and~\eqref{def:barMz} respectively. Moreover, assume that $\hat{\psi}(b_0)$ and $\hat{\psi}(z_k)$ for $k\in\{0,1,...\}$ are computed in Algorithm~\ref{algo:PB} using~\eqref{eq:hatpsi} with
    \begin{align}
        T'_k\geq\tfrac{1536}{\alpha}\left\{\tfrac{8c^2}{z_k^2}\left[M_f^2+16(L_fD_X+1)^2\right]+\beta^2\right\}\epsilon^{-2}.
        \label{Tk'stop2}
    \end{align}
    Then for any $k\in\{0,1,...\}$, with probability at least $(1-\alpha)^{k+1}$, we have:
    \begin{itemize}
        \item [(1)] $z^*\in [a_s,b_s]$ and $z_s\geq z^*/2$ for $s\in\{0,...,\min\{k_0,k\}\}$;
        \item [(2)] $|\hat{\psi}(b_0)-\psi(b_0)|\leq\epsilon/8$ and $|\hat{\psi}(z_s)-\psi(z_s)|\leq\epsilon/8$ for $s\in\{0,...,k-1\}$.
    \end{itemize}
    \label{lem:stop2}
\end{lemma}
\begin{proof}
    We construct this proof using inductive arguments, similar to those in Lemma~\ref{lem:induction1}. First, we clearly have $z^*\in [a_0,b_0]$ and $z_0=b_0/2\geq z^*/2$. Moreover, in view of the parameter settings in~\eqref{Tkstop2} and~\eqref{Tk'stop2}, as well as the condition~\eqref{stopT} in Lemma~\ref{lem:stop1}, we have $\P(|\hat{\psi}(b_0)-\psi(b_0)|\leq\epsilon/8)\geq 1-\alpha$. Thus, (1) and (2) are satisfied with probability at least $1-\alpha$ for $k=0$.

    Now, assume that (1) and (2) are true for any $k\in\{0,1,...\}$. Below we consider two cases. First, let us suppose $k<k_0$. In this case, it suffices to show that $z^*\in [a_{k+1},b_{k+1}],z_{k+1}\geq z^*/2$ and $|\hat{\psi}(z_{k})-\psi(z_k)|\leq\epsilon/8$ hold with probability at least $1-\alpha$. Considering the outer iteration $k$ in Algorithm~\ref{algo:PB}, we can show that
    \begin{align}
        \left(\tsum_{t=0}^{T_k-1}\gamma_{k,t}\right)^{-1}\E\left[\max_{u\in U}\tsum_{t=0}^{T_k-1}\gamma_{k,t}\Bar{l}'_{z_k}(u_{k,t})^{\top}(u_{k,t}-u)\right]\leq\tfrac{\epsilon}{16}
        \label{eventAC}
    \end{align}
    holds with probability at least $1-\alpha/3$, by adopting the same approach as in~\eqref{eq:eventCbound}-\eqref{stopprob1} and using~\eqref{Tkstop2}. Notice that the events $A$ and $C$ defined in~\eqref{def:eventA} and~\eqref{def:eventCD}, respectively, occur when~\eqref{eventAC} is satisfied. Moreover, each one of the events $B$ and $D$, defined in~\eqref{def:eventB} and ~\eqref{def:eventCD}, respectively, occurs with probability at least $1-\alpha/3$, conditioned on~\eqref{eventAC}. These results can be established by using the same approach as in~\eqref{eq2.5lem2algo}-\eqref{eq:prob2} and~\eqref{eq:stopvar}-\eqref{eq:probDonC}, respectively, and applying the parameter settings in~\eqref{Tkstop2} and~\eqref{Tk'stop2}. As a result, letting $E$ denote the event in~\eqref{eventAC}, we have
    \begin{align}
        \P(A\cap B\cap C\cap D)\geq\P(B\cap D\cap E)= \P(E)\P(B\mid E)\P(D\mid E)\geq \left(1-\tfrac{\alpha}{3}\right)^3\geq1-\alpha,
    \end{align}
    where the equality follows from the independence of the events $B$ and $D$ given~\eqref{eventAC}, since the samples $\xi'_{k,t},t=0,...,T_k-1$ used to compute $\zeta_k$ are independent of the samples $\xi''_{k,i},i=1,...,T'_k$ used to compute $\hat{\psi}(z_k)$. Next, in view of Lemma~\ref{lem:subdiff} and Lemma~\ref{lem:stop1}, when the events $A,B,C$ and $D$ all occur, the following relations hold:
    \begin{align*}
        \psi^*\geq\psi(z_k)+\zeta_k(z^*-z_k)-\tfrac{\epsilon}{2} \quad \text{and} \quad |\hat{\psi}(z_k)-\psi(z_k)|\leq\epsilon/8.
    \end{align*}
    According to the inductive arguments in Lemma~\ref{lem:induction1}, the first relation above, combined with the assumptions $z^*\in [a_k,b_k]$ and $z_k\geq z^*/2$, implies that $z^*\in [a_{k+1},b_{k+1}]$ and $z_{k+1}\geq z^*/2$. Thus, we complete the case $k<k_0$. Now we suppose $k\geq k_0$. In this case, we only need to show the event (2), i.e., $|\hat{\psi}(z_{k})-\psi(z_k)|\leq\epsilon/8$ holds with probability at least $1-\alpha$, which directly follows from Lemma~\ref{lem:stop1} and the parameter settings in~\eqref{Tkstop2} and~\eqref{Tk'stop2}.  
\end{proof}
\vgap

With the help of Lemma~\ref{lem:stop2}, we are ready to show the reliability of our enhanced stopping criterion~\eqref{originalstop}-\eqref{stopcriteria} in Theorem~\ref{thm:stop} below.

\begin{theorem}
    Under Assumptions~\ref{assumption1} and~\ref{assumption2}, for any $\alpha_0\in (0,1)$ and $\alpha\in(0,1-\alpha_0)$, suppose that the algorithmic parameters in Algorithm~\ref{algo:PB} are set to
    \begin{align}
        m=\left\lceil\tfrac{\beta^2}{\alpha_0}\right\rceil, \quad \theta=\tfrac{\epsilon}{c}, \quad T_k=\left\lceil\max\left\{\left(\tfrac{216K}{\alpha}D_{V,x_0}M_{z_k}\right)^2,\tfrac{48K}{\alpha}\Bar{M}_{z_k}^2\right\}\epsilon^{-2}\right\rceil \quad \text{and} \quad \gamma_{k,t}=\tfrac{D_{V,x_0}}{M_{z_k}\sqrt{T_k}},
        \label{Tkthm2}
    \end{align}
    where $D_{V,x_0}$, $M_{z_k}$ and $\Bar{M}_{z_k}$ are defined in~\eqref{def:D},~\eqref{def:M^2} and~\eqref{def:barMz} respectively. Moreover, assume that the maximum number of outer iterations in Algorithm~\ref{algo:PB} is set to
    \begin{align}
        K=\left\lceil\log_2\left\{\tfrac{8cb_0}{\epsilon}\left[\tfrac{\beta^2+4(L_fD_X+1)^2}{(\epsilon/2c)^2}+1\right]\right\}\right\rceil,
        \label{Kthm2}
    \end{align}
    and the stopping criterion~\eqref{originalstop}-\eqref{stopcriteria} is applied, where $\hat{\psi}(\cdot)$ is computed using~\eqref{eq:hatpsi} with
    \begin{align}
        T'_k=\left\lceil\tfrac{1536K}{\alpha}\left\{\tfrac{8c^2}{z_k^2}\left[M_f^2+16(L_fD_X+1)^2\right]+\beta^2\right\}\epsilon^{-2}\right\rceil.
        \label{Tk'thm2}
    \end{align}
    Then Algorithm~\ref{algo:PB} must terminate in some outer iteration $\tilde{k}\leq K-1$, and with probability at least $1-\alpha_0-\alpha$, an $\epsilon$-optimal solution for problem~\eqref{l2problem} is attained in some outer iteration $k_0\leq\tilde{k}$, and we have $z_k\geq\max\{z^*/4,\epsilon/(2c)\}$ for $k\in\{0,...,\tilde{k}\}$.
    \label{thm:stop}
\end{theorem}
\begin{proof}
    This proof is built on Lemma~\ref{lem:stop2} and Theorem~\ref{thm:converge1}, and we still condition on~\eqref{eq:Ycontainy}. First, notice that under $k=k_0$, the statement (1) in Lemma~\ref{lem:stop2} is equivalent to the inductive statement~\eqref{eq:induction1} in Lemma~\ref{lem:induction1}, which,  through~\eqref{zregion} in Theorem~\ref{thm:converge1}, implies the following observations: (a) the stopping criterion~\eqref{originalstop} can only be triggered in some outer iteration $k'\geq k_0$, (b) $z_k\geq\epsilon/(2c)$ holds for $k\in\{0,...,k'\}$, and (c) $k_0\leq K-1$ is satisfied. Consequently, in view of the parameter settings in~\eqref{Tkthm2} and~\eqref{Tk'thm2}, taking $k=K-1$ in Lemma~\ref{lem:stop2}, with probability at least $(1-\alpha/K)^K\geq 1-\alpha$, we have
    \begin{align}
        z^*\in [a_s,b_s] \quad \text{and} \quad z_s\geq z^*/2 \quad \text{for} \quad s\in\{0,...,k_0\},
        \label{statement1thm2}
    \end{align}
    as well as
    \begin{align}
        |\hat{\psi}(b_0)-\psi(b_0)|\leq\epsilon/8 \quad \text{and} \quad |\hat{\psi}(z_s)-\psi(z_s)|\leq\epsilon/8 \quad \text{for} \quad s\in\{0,...,K-2\}.
        \label{statement2thm2}
    \end{align}

    Below we discuss the reliability of the stopping criterion~\eqref{stopcriteria} under two cases. Let us first consider any outer iteration $k\in\{0,...,k_0-1\}$, in which $\psi(z_k)-\psi^*>\epsilon/2$ holds by Lemma~\ref{lem:opt}. We show that the stopping criterion~\eqref{stopcriteria} will not be triggered. Below we consider two subcases.
    \begin{itemize}
        \item (a) $z_k>z^*$. In this subcase, we have
        \begin{align}
            \psi(b_k)-\psi(z_k)\geq\psi'(z_k)(b_k-z_k)\geq\psi'(z_k)(z_k-z^*)\geq\psi(z_k)-\psi^*>\tfrac{\epsilon}{2},
            \label{eq1thmstop}
        \end{align}
        where the first and third inequalities hold due to the convexity of $\psi(z)$, and the second inequality follows from $\psi'(z_k)>0$ and $b_k-z_k=z_k-a_k\geq z_k-z^*$. Since we have $|\hat{\psi}(b_k)-\psi(b_k)|\leq\epsilon/8$ and $|\hat{\psi}(z_k)-\psi(z_k)|\leq\epsilon/8$ by~\eqref{statement2thm2}, in view of~\eqref{eq1thmstop}, $\hat{\psi}(b_k)-\hat{\psi}(z_k)>\epsilon/4$ holds. Thus, the stopping criterion~\eqref{stopcriteria} is not triggered.
        \item (b) $z_k<z^*$. In this subcase, we have $[a_{k+1},b_{k+1}]=[z_k,b_k]$ that brackets $z^*$, and hence the stopping criterion~\eqref{stopcriteria} is not triggered.
    \end{itemize}

    Notice that the previous discussion combined with $K-1\geq k_0$ and $k'\geq k_0$ implies $\tilde{k}\geq k_0$. Now, for the outer iterations $k\in\{k_0,...,\tilde{k}\}$, we show that $z_k\geq z^*/4$ is satisfied. Since $z^*\in[a_{k_0},b_{k_0}]$ holds by~\eqref{statement1thm2}, let us assume $[a_{k+1},b_{k+1}]=[a_k,z_k]$ for $k=k_0,k_0+1,...$, until $k=\Bar{k}\leq\tilde{k}$. It follows that either $[a_{\Bar{k}+1},b_{\Bar{k}+1}]=[z_{\Bar{k}},b_{\Bar{k}}]$ holds or $\bar{k}=\tilde{k}$. Below we consider two subcases.
    \begin{itemize}
        \item (a) $b_{\Bar{k}}\geq z^*$. In this subcase, since $b_{k+1}=z_k\leq b_k$ for $k\in\{k_0,...,\Bar{k}-1\}$, we have $b_k\geq b_{\Bar{k}}\geq z^*$ for $k\in\{k_0,...,\Bar{k}\}$. Thus, $z_k=(a_k+b_k)/2\geq z^*/2$ holds for $k\in\{k_0,...,\Bar{k}\}$. If $\bar{k}<\tilde{k}$, then  $[a_{\Bar{k}+1},b_{\Bar{k}+1}]=[z_{\Bar{k}},b_{\Bar{k}}]$, which implies $z_k\geq a_{\Bar{k}+1}=z_{\bar{k}}\geq z^*/2$ for $k\in\{\Bar{k}+1,...,\tilde{k}\}$.
        \item (b) $b_{\Bar{k}}<z^*$. In this subcase, since $b_{k_0}\geq z^*$ holds and $b_k$ monotonically decreases for $k\in\{k_0,...,\Bar{k}\}$, there exists $k_1\in\{k_0+1,...,\Bar{k}\}$ such that $b_{k_1}<z^*$ and $b_{k_1-1}\geq z^*$. Notice that $z_k=(a_k+b_k)/2\geq b_{k_1-1}/2\geq z^*/2$ holds for $k\in\{k_0,...,k_1-1\}$. Now we consider the outer iteration $k_1$, where $z_{k_1}=(a_{k_1}+b_{k_1})/2\geq b_{k_1}/2=[(a_{k_1-1}+b_{k_1-1})/2]/2\geq b_{k_1-1}/4\geq z^*/4$. Due to $z_{k_1}<b_{k_1}<z^*$ and the convexity of $\psi$, $\psi(z_{k_1})>\psi(b_{k_1})$ holds. Then, since $|\hat{\psi}(b_{k_1})-\psi(b_{k_1})|\leq\epsilon/8$ and $|\hat{\psi}(z_{k_1})-\psi(z_{k_1})|\leq\epsilon/8$ by~\eqref{statement2thm2}, we have $\hat{\psi}(b_{k_1})-\hat{\psi}(z_{k_1})<\epsilon/4$. Now we consider two scenarios: Suppose that $[a_{k_1+1},b_{k_1+1}]=[a_{k_1},z_{k_1}]$, then the stopping criterion~\eqref{stopcriteria} is triggered and $\tilde{k}=k_1$; otherwise, we have $[a_{k_1+1},b_{k_1+1}]=[z_{k_1},b_{k_1}]$, then $z_k\geq a_{k_1+1}=z_{k_1}\geq z^*/4$ holds for $k\in\{k_1+1,...,\tilde{k}\}$.
    \end{itemize}

    We conclude this proof by combining the discussions above.
\end{proof}
\vgap

Theorem~\ref{thm:stop} implies that the enhanced stopping criterion~\eqref{originalstop}-\eqref{stopcriteria} helps maintain the oracle complexity $\Tilde{\cO}(\max\{z^*,\epsilon\}^{-2}\epsilon^{-2})$ in~\eqref{simplifycomplexityresult} for Algorithm~\ref{algo:PB} applied to problem~\eqref{l2problem}. It is worth mentioning that we finally adopt the solution $\Bar{x}_k$ (see $\Bar{u}_k$ in Algorithm~\ref{algo:PB} where $\Bar{u}_k=(\Bar{x}_k,\Bar{y}_k,\Bar{\lambda}_k)$), in which $k=\argmin_{k=0,...,\tilde{k}}\hat{\psi}(z_k)$.

\subsection{High Probability Guarantees}
\label{sec:highprob}

In this subsection, we establish high probability guarantees for attaining an $\epsilon$-optimal solution for problem~\eqref{l2problem} using Algorithm~\ref{algo:PB}.

\subsubsection{High Probability Guarantees under Sub-Gaussian Assumptions}

Our first approach is to make the following sub-Gaussian assumption.

\begin{assumption}
    Let $G(x,\xi)=(F(x,\xi)-\E[F(x,\xi)])_+^2$, then we have
    \begin{align*}
        &\textit{(a)} \ \E\left[\exp\left\{\left\|F'(x,\xi)\right\|_{*,X}^2\big/\sigma_1^2\right\}\right]\leq\exp\{1\}, \qquad \textit{(b)} \ \E\left[\exp\left\{[F(x,\xi)-f(x)]^2\big/\sigma_2^2\right\}\right]\leq\exp\{1\}, \\
        &\textit{(c)} \ \E\left[\exp\left\{\left\|G'(x,\xi)\right\|_{*,X}^2\big/\sigma_3^2\right\}\right]\leq\exp\{1\}, \qquad \textit{(d)} \ \E\left[\exp\left\{(F(x,\xi)-f(x))_+^4\big/\sigma_4^4\right\}\right]\leq\exp\{1\}.
    \end{align*}
    for any $x\in X$.
    \label{assumption3}
\end{assumption}
Note that Assumption~\ref{assumption3} is sufficient to imply the conditions~\eqref{assum:fbddvar} and~\eqref{eq:Flips}, as well as Assumptions~\ref{assumption1} and~\ref{assumption2}. To present the convergence of Algorithm~\ref{algo:PB} under Assumption~\ref{assumption3}, we first define
\begin{align}
    \sigma_z^2:=4\max\left\{c_{z1}\sigma_1^2,c_{z2}\sigma_2^2,c_{z3}\sigma_3^2,c_{z4}\right\},
    \label{def:sigma}
\end{align}
where
\begin{align}
    c_{z1}:=\tfrac{96c^2}{z^2}(L_fD_X+1)^2+6, \ c_{z2}:=\tfrac{12c^2}{z^2}(L_f^2+1)+2, \ c_{z3}:=\tfrac{3c^2}{z^2}, \ c_{z4}:=\left(\tfrac{48c^2}{z^2}+8\right)(L_fD_X+1)^2+3.
    \label{def:c1234}
\end{align}
Also note that we defer several results with their proofs, built upon Assumption~\ref{assumption3}, to Appendix A. Based on these results, we establish the high probability convergence for Algorithm~\ref{algo:PB} in Theorem~\ref{thm:converge1highprob} below, which improves the result in Theorem~\ref{thm:converge1}. The proof of Theorem~\ref{thm:converge1highprob} is also deferred to Appendix A. As a remark, a similar improvement can also be made to Theorem~\ref{thm:stop}, and we omit the details here.

\begin{theorem}
    Under Assumption~\ref{assumption3}, for any $\alpha_0\in (0,1)$ and $\alpha\in(0,1-\alpha_0)$, let
    \begin{align*}
        K=\left\lceil\log_2\left\{\tfrac{8cb_0}{\epsilon}\left[\tfrac{\beta^2+4(L_fD_X+1)^2}{\max\{z^*/2,\epsilon/2c\}^2}+1\right]\right\}\right\rceil
    \end{align*}
    denote the maximum number of the outer iterations in Algorithm~\ref{algo:PB}. Also assume that the stopping criterion~\eqref{simplestopping} is applied in Algorithm~\ref{algo:PB}, and the algorithmic parameters in Algorithm~\ref{algo:PB} are set to
    \begin{align}
        m=\left\lceil3\ln\!\left(\!\tfrac{2}{\alpha_0}\!\right)\!\sigma_2^2\right\rceil, \quad \theta=\tfrac{\epsilon}{c}, \quad T_k=\left\lceil\max\left\{4\!\left[9\!+\!5\ln\!\left(\!\tfrac{2K}{\alpha}\!\right)\!\right]^2\!D_{V,x_0}^2\sigma_{z_k}^2,48\ln\!\left(\!\tfrac{2K}{\alpha}\!\right)\!\Bar{\sigma}_{z_k}^2\right\}\epsilon^{-2}\right\rceil \quad \text{and} \quad \gamma_{k,t}=\tfrac{D_{V,x_0}}{\sigma_{z_k}\sqrt{T_k}},
        \label{Tgammakthmhighprob}
    \end{align}
    where
    \begin{align*}
        \Bar{\sigma}_z^2:=\tfrac{32c^2}{z^2}\max\left\{\sigma_4^4,32(L_fD_X+1)^4+\beta^4\right\}
    \end{align*}
    and $D_{V,x_0}$ and $\sigma_{z_k}$ are defined in~\eqref{def:D} and~\eqref{def:sigma}, respectively. Then an $\epsilon$-optimal solution for problem~\eqref{l2problem} can be attained within at most $m+2K\Tilde{T}$ number of oracle calls with probability at least $1-\alpha_0-\alpha$, where
    \begin{align*}
        \Tilde{T}=\left\lceil\max\left\{4\left[9+5\ln\left(\tfrac{2K}{\alpha}\right)\right]^2D_{V,x_0}^2\sigma_{z_{\min}}^2,48\ln\left(\tfrac{2K}{\alpha}\right)\Bar{\sigma}_{z_{\min}}^2\right\}\epsilon^{-2}\right\rceil \quad \text{and} \quad z_{\min}=\max\left\{\tfrac{z^*}{2},\tfrac{\epsilon}{2c}\right\}.
    \end{align*}
    \label{thm:converge1highprob}
\end{theorem}

\subsubsection{High Probability Guarantees under Multiple Sample Trajectories and Robust Distance Approximation}
\label{sec:highprobRDA}

Instead of relying on Assumption~\ref{assumption3}, we establish our high probability convergence guarantee through an alternative approach. Specifically, we derive high probability bounds for the events $A$ and $B$ defined in~\eqref{def:eventA} and~\eqref{def:eventB} of Lemma~\ref{lem:subdiff}, respectively. For the event $A$, we independently execute Algorithm~\ref{algo:SMD} $m_1$ times, selecting in each run a high-quality estimation of $\max_{u\in U}(\tsum_{t=0}^{T-1}\gamma_t)^{-1}(\tsum_{t=1}^{T-1}\gamma_t\bar{l}'_z(u_t)^{\top}(u_t-u))$ via a robust distance approximation (RDA) procedure with $m_2$ independent sample trajectories, and then take the minimum of these $m_1$ estimations. For the event $B$, we again apply the RDA procedure with $m_3$ independent sample trajectories to estimate $\zeta_T$.

In order to obtain a high probability bound for the event $A$, in each outer iteration $k$ of Algorithm~\ref{algo:PB}, we independently invoke Algorithm~\ref{algo:SMD} $m_1$ times with the same input $\{u_0,z_k,T_k,\{\gamma_{k,t}\}_{t=0}^{T_k-1}\}$. In the $i$-th call ($i\in\{1,2,...,m_1\}$) of Algorithm~\ref{algo:SMD}, the iterates are denoted as $u^i_{k,t},t=0,1,...,T_k$. For each $t$, we generate the stochastic subgradient vector $l'_{z_k}(u^i_{k,t},\xi^{i,j}_{k,t}), j=1,2,...,m_2$, using $m_2$ independent samples. Following the definition of event $A$ in~\eqref{def:eventA}, our goal is to obtain a reliable estimation of
\begin{align}
    S^i_{k,t}:=\max_{u\in U}\left(\tsum_{t=0}^{T_k-1}\gamma_{k,t}\right)^{-1}\left(\tsum_{t=0}^{T_k-1}\gamma_{k,t}\Bar{l}'_{z_k}\left(u^i_{k,t}\right)^{\top}\left(u^i_{k,t}-u\right)\right).
    \label{def:RDAobject1}
\end{align}
To this end, we utilize the following robust distance approximation (RDA) procedure, which was first introduced in~\cite{nemirovskij1983problem} and later also analyzed in~\cite[Algorithm 2]{hsu2016loss}, as shown in Algorithm~\ref{algo:RDA}.

\begin{algorithm}
\caption{Robust Distance Approximation (RDA)}
\label{algo:RDA}
\begin{algorithmic}
\State{\textbf{Input:} $w_1,...,w_N$.}
\For{$j=1,...,N$}
\State{$r_j=\min\left\{r\geq 0:\left|B(w_j,r)\cap\{w_j\}_{j=1}^N\right|\geq N/2\right\}$}
\EndFor
\State{Set $j^*=\argmin_{j\in\{1,...,N\}}r_j$.}
\State{\textbf{Output:} $\hat{w}=w_{j^*}$}
\end{algorithmic}
\end{algorithm}
\vgap

For each $i\in\{1,...,m_1\}$, taking $N=m_2$ and
\begin{align}
    w_j=\hat{S}^{i,j}_{k,t}:=\max_{u\in U}\left(\tsum_{t=0}^{T_k-1}\gamma_{k,t}\right)^{-1}\left(\tsum_{t=0}^{T-1}\gamma_{k,t}l'_{z_k}\left(u^i_{k,t},\xi^{i,j}_{k,t}\right)^{\top}\left(u^i_{k,t}-u\right)\right),\quad j=1,2,...,N
    \label{def:RDAobject2}
\end{align}
as the inputs of Algorithm~\ref{algo:RDA}, it then outputs $\hat{w}_i=w_{j_i^*}=\hat{S}^{i,j_i^*}_{k,t}$, which serves as a high-quality approximation of $S^i_{k,t}$ defined in~\eqref{def:RDAobject1} with high probability. Consequently, we select
\begin{align}
    i^*=\argmin_{i\in\{1,2,...,m_1\}}\hat{w}_i.
    \label{def:i*RDA}
\end{align}
We will show that the event $\{S^{i^*}_{k,t}\leq\epsilon/4\}$, corresponding to the event $A$ defined in~\eqref{def:eventA}, occurs with high probability.

Next, we derive a high probability bound for the event $B$ defined in~\eqref{def:eventB}. Considering the $i^*$-th call of Algorithm~\ref{algo:SMD}, it outputs the scalars $\zeta^j_k,j=1,2,...,m_3$ using $m_3$ independent sample trajectories $\{\xi'^{i^*,j}_{k,t}\}_{t=1}^{T_k},j=1,2,...,m_3$, i.e.,
\begin{align}
    \zeta^j_k=\left(\tsum_{t=0}^{T_k-1}\gamma_{k,t}\right)^{-1}\tsum_{t=0}^{T_k-1}\gamma_{k,t}\cL'_{z_k}\left(x^{i^*}_{k,t},y^{i^*}_{k,t},z_k,\lambda^{i^*}_{k,t},\xi'^{i^*,j}_{k,t}\right).
    \label{def:zetakRDA}
\end{align}
Taking $N=m_3$ and $w_j=\zeta^j_k$ as the inputs to Algorithm~\ref{algo:RDA}, it then outputs $\hat{w}=w_{j^*}=\zeta^{j^*}_k$. We will show that the event $B$ (see~\eqref{def:eventB}) taking $\zeta_T=w_{j^*}$ also occurs with high probability.

Finally, we apply Algorithm~\ref{algo:RDA} to compute $y_0$, one of the inputs in Algorithm~\ref{algo:PB}, using $m_0$ independent sample trajectories, i.e., $\{\xi^j_i\}_{i=1}^m,j=1,2,...,m_0$. The detailed derivations of the high probability bounds for the events $A$ and $B$ are deferred to Appendix B. Using these results, we establish the high probability convergence of Algorithm~\ref{algo:PB} in Theorem~\ref{thm:highprobRDA}, whose proof is also provided in Appendix B. Similar to Theorem~\ref{thm:converge1highprob}, Theorem~\ref{thm:highprobRDA} strengthens the result in Theorem~\ref{thm:converge1}. To facilitate our analysis, we assume that there exist constants $c_{2,*}$ and $c_{*,2}$ such that $\|\cdot\|_*\leq c_{2,*}\|\cdot\|_2$ and $\|\cdot\|_2\leq c_{*,2}\|\cdot\|_*$.

\begin{theorem}
    Under Assumptions~\ref{assumption1} and~\ref{assumption2}, for any $\alpha_0\in (0,1)$ and $\alpha\in(0,1-\alpha_0)$, let
    \begin{align*}
        K=\left\lceil\log_2\left\{\tfrac{8cb_0}{\epsilon}\left[\tfrac{\beta^2+4(L_fD_X+1)^2}{\max\{z^*/2,\epsilon/2c\}^2}+1\right]\right\}\right\rceil
    \end{align*}
    denote the maximum number of outer iterations in Algorithm~\ref{algo:PB}. Also assume that the stopping criterion~\eqref{simplestopping} is applied in Algorithm~\ref{algo:PB}, the algorithmic parameters in Algorithm~\ref{algo:PB} are set to
    \begin{align*}
        m=\left\lceil27\beta^2\right\rceil, \quad \theta=\tfrac{\epsilon}{c}, \quad T_k=\left\lceil\max\left\{\left[216(c_{2,*}c_{*,2}\!+\!1)D_{V,x_0}M_{z_k}\right]^2,\left(36\Bar{M}_{z_k}\right)^2\right\}\epsilon^{-2}\right\rceil \quad \text{and} \quad \gamma_{k,t}=\tfrac{D_{V,x_0}}{M_{z_k}\sqrt{T_k}},
    \end{align*}
    where $D_{V,x_0},M_{z_k}$ and $\Bar{M}_{z_k}$ are defined in~\eqref{def:D},~\eqref{def:M^2} and~\eqref{def:barMz}, respectively, and the parameters $m_1,m_2$ and $m_3$ described in Section~\ref{sec:highprobRDA} are set to
    \begin{align*}
        m_0=\left\lceil 18\ln\left(\tfrac{1}{\alpha_0}\right)\right\rceil, \quad m_1=\left\lceil\log_2\left(\tfrac{4K}{\alpha}\right)\right\rceil, \quad m_2=\left\lceil 18\ln\left(\tfrac{4m_1K}{\alpha}\right)\right\rceil, \quad \text{and} \quad m_3= \left\lceil18\ln\left(\tfrac{2K}{\alpha}\right)\right\rceil,
    \end{align*}
    Then an $\epsilon$-optimal solution for problem~\eqref{l2problem} can be attained within at most $mm_0+[m_1(m_2+1)+m_3]K\Tilde{T}$ number of oracle calls with probability at least $1-\alpha_0-\alpha$, where
    \begin{align*}
        \Tilde{T}=\left\lceil\max\left\{\left[216(c_{2,*}c_{*,2}\!+\!1)D_{V,x_0}M_{z_{\min}}\right]^2,\left(36\Bar{M}_{z_{\min}}\right)^2\right\}\epsilon^{-2}\right\rceil \quad \text{and} \quad z_{\min}=\max\left\{\tfrac{z^*}{2},\tfrac{\epsilon}{2c}\right\}.
    \end{align*}
    \label{thm:highprobRDA}
\end{theorem}

\section{Lower Complexity Bound}
\label{section:lower}

In this section, we provide a lower complexity bound for the mean-upper-semideviation problem~\eqref{l2problem}, which almost matches the convergence rate~\eqref{simplifycomplexityresult} of Algorithm~\ref{algo:PB} in the previous section. Let us first illustrate why the lower complexity bound for problem~\eqref{l2problem} is instance-dependent and might be worse than $\Omega(\epsilon^{-2})$ (see~\cite{nemirovskij1983problem}). Suppose that in problem~\eqref{l2problem}, the mean $\E[F(x,\xi)]$ is identical for any $x\in X$ and $c=1$, then problem~\eqref{l2problem} is equivalent to minimize the semideviation $g^{1/2}(x)=\E^{1/2}[(F(x,\xi)-\E[F(x,\xi)])_+^2]$ over $X$. Since the square root function $f_0(x)=x^{1/2}$ is concave and monotonically increasing over $[0,\infty)$, we have 
\begin{align}
    g^{\tfrac{1}{2}}(x)-g^{\tfrac{1}{2}}(x^*)=f_0(g(x))-f_0(g(x^*))\geq\nabla f_0(g(x))(g(x)-g(x^*))
    \label{eq:lower_relation}
\end{align}
for any $x\in X$, where $x^*=\argmin_{x\in X}g(x)$. By the classical complexity theory (see~\cite{nemirovskij1983problem}), under proper assumptions, for any solution $\Tilde{x} \in X$ generated by any algorithm utilizing $O(1/\epsilon^2)$ number of samples, we must have $g(\Tilde{x})-g(x^*)>\epsilon$ for some worst case problem instances. Then, in view of~\eqref{eq:lower_relation}, we must have $g^{1/2}(x)-g^{1/2}(x^*)>\nabla f_0(g(x))\epsilon$. However, $f_0(z)$ is not uniformly Lipschitz continuous in $[0,\infty)$ and its gradient blows up as $z\to 0$. Suppose that for some problem instance, $g(x^*)$ is close to $0$, then when $g(x)$ is close to $g(x^*)$, $\nabla f_0(g(x))$ becomes large. This gives us an idea why 
obtaining an $\epsilon$-optimal solution for problem $\min_{x\in X}g^{1/2}(x)$ may require more samples than that for problem $\min_{x\in X}g(x)$.

In order to formally present the lower complexity result for solving problem~\eqref{l2problem}, we consider a specific class of instances, which is inspired by~\cite{nemirovskij1983problem}. For notational simplicity, we denote $D=D_X$ and pick any constant $L>0$. Then for any $s_0\geq 0$, we define a class of stochastic functions parameterized by $t$ given by
\begin{align*}
    f_{t}(x):=t\left(d_t-x\right)
\end{align*}
over the domain $X:=\left[-D/2,D/2\right]$, where 
\begin{align}
    d_t:=\begin{cases}
        -\tfrac{D}{2}, & \text{if} \ t<0, \\
        \tfrac{D}{2}, & \text{if} \ t>0
    \end{cases}
    \label{eq:d}
\end{align}
and the parameter $t$ satisfies
\begin{align}
   t \in  \left(-\tfrac{1}{2}L^{\tfrac{1}{2}}[2(D+s_0)]^{-\tfrac{1}{2}}, 0\right) \bigcup \left(0,\tfrac{1}{2}L^{\tfrac{1}{2}}[2(D+s_0)]^{-\tfrac{1}{2}}\right).
    \label{eq:t}
\end{align}
It is worth noting that the constant $s_0$ in~\eqref{eq:t} will play an important role
on the construction of the lower bound (see discussions after Lemma~\ref{lem:lower2}).
Now we discuss some properties of the function $f_t$. First, notice that $f_t(x)\geq 0$ holds for any $x\in X$. Furthermore, for any $t>0$ and $x\in X$, $f_{t}(x)=f_{-t}(-x)$ holds, thus the pair of functions $f_{t}$ and $f_{-t}$ are symmetric to each other with respect to $x=0$. 

Below we define a stochastic first-order oracle $\mathcal{O}_f$ that returns a pair of unbiased estimators $(F_{t}(x,\xi),F'_{t}(x,\xi))$ for $(f_{t}(x),f'_{t}(x))$. Specifically, $\mathcal{O}_f$ returns
\begin{align}
    \left(F_{t}(x,\xi),F'_{t}(x,\xi)\right):=\begin{cases}
        \left(-\tfrac{|t|s_0}{2(1-\alpha)},0\right) & \text{with prob.} \ 1-\alpha, \\
        \tfrac{1}{\alpha}\left(f_{t}(x)+\tfrac{|t|s_0}{2},f_{t}'(x)\right) & \text{with prob.} \ \alpha,
    \end{cases}
    \label{def:Of}
\end{align}
where
\begin{align}
    \alpha:=|t|^{\tfrac{4}{3}}L^{-\tfrac{2}{3}}[2(D+s_0)]^{\tfrac{2}{3}}\in\left(0,\tfrac{1}{2}\right).
    \label{def:alpha}
\end{align}
Here, the relation $\alpha\in(0,1/2)$ follows from~\eqref{eq:t}. Note that by~\eqref{def:Of}, with probability $1-\alpha$, $\cO_f$ returns constant information that is independent of the decision variable $x$.

Moreover, we define the upper semideviation $g_{t}(x):=\E[G_{t}(x,\xi)]$ with 
\begin{align*}
    G_{t}(x,\xi):=[F_{t}(x,\xi)-f_{t}(x)]_+^2.
\end{align*}
Notice that for any $t>0$ and $x\in X$, we have
\begin{align}
    g_{t}(x)=\alpha\left[\tfrac{1}{\alpha}\left(f_t(x)+\tfrac{|t|s_0}{2}\right)-f_t(x)\right]^2=\alpha\left[\tfrac{1}{\alpha}\left(f_{-t}(-x)+\tfrac{|t|s_0}{2}\right)-f_{-t}(-x)\right]^2=g_{-t}(-x),
    \label{eq:gsymmetry}
\end{align}
where the second inequality follows from $f_t(x)=f_{-t}(-x)$. Thus, the pair of functions $g_{t}$ and $g_{-t}$ are symmetric to each other with respect to $x=0$.

Finally, we are ready to construct a class of instances for problem~\eqref{l2problem} as
\begin{align}
    \min_{x\in X} \left\{h_{t}(x):=f_{t}(x)+g_{t}^{\tfrac{1}{2}}(x)\right\}.
    \label{problem:h_lower}
\end{align}
It follows that the pair of functions $h_{t}$ and $h_{-t}$ are also symmetric to each other with respect to $x=0$, since for any $t>0$ and $x\in X$, we have $h_t(x)=f_{t}(x)+g_{t}^{1/2}(x)=f_{-t}(-x)+g_{-t}^{1/2}(-x)=h_{-t}(-x)$. Note that problem~\eqref{problem:h_lower} is given in the form of the $L_2$ risk minimization problem, i.e., problem~\eqref{l2problem} with $c=1$. Below we show that~\eqref{problem:h_lower} satisfies the general assumptions we make for problem~\eqref{l2problem}.
\begin{lemma}
    The class of problems~\eqref{problem:h_lower} satisfies the Lipschitz continuity condition~\eqref{assum:flipschitz}, the bounded variances conditions~\eqref{assum:fgradbddvar} and~\eqref{assum:fbddvar}, and Assumptions~\ref{assumption1} and~\ref{assumption2} on the semideviations. Specifically, considering $\|\cdot\|$ as the $l_2$-norm, we have:
    \begin{itemize}
        \item [(a)] In~\eqref{assum:flipschitz}, $L_f=2^{-3/2}L^{1/2}(D+s_0)^{-1/2}$;
        \item [(b)] In~\eqref{assum:fgradbddvar}, $\sigma_f^2=(2^{-3}+2^{-5/3})L(D+s_0)^{-1}$;
        \item [(c)] In~\eqref{assum:fbddvar}, $\beta^2=(2^{-3}+2^{-5/3})L(D+s_0)$;
        \item [(d)] In Assumption~\ref{assumption1}, $M_f^4=2^{-2}L^2(D+s_0)^2$;
        \item [(e)] In Assumption~\ref{assumption2}, $L_G^2=L^2$.
    \end{itemize}
    \label{lem:lower1}
\end{lemma}
\begin{proof}
    For~\eqref{assum:flipschitz}, we have
    \begin{align*}
        \left\|f'_{t}(x)\right\|_2=|t|\leq\tfrac{1}{2}L^{\tfrac{1}{2}}[2(D+s_0)]^{-\tfrac{1}{2}}=L_f
    \end{align*}
    by~\eqref{eq:t}. For~\eqref{assum:fgradbddvar}, we have
    \begin{align*}
        &\E\left[\left\|F'_{t}(x,\xi)-f'_{t}(x)\right\|_2^2\right]=(1-\alpha)(0+t)^2+\alpha \left(-\tfrac{1}{\alpha}t+t\right)^2=(1-\alpha)t^2+\tfrac{(1-\alpha)^2}{\alpha}t^2\overset{(a)}{\leq}t^2+\tfrac{1}{\alpha}t^2 \\
        &\overset{(b)}{=}t^2+|t|^{\tfrac{2}{3}}L^{\tfrac{2}{3}}[2(D+s_0)]^{-\tfrac{2}{3}}\overset{(c)}{\leq}\tfrac{1}{4}L[2(D+s_0)]^{-1}+2^{-\tfrac{2}{3}}L[2(D+s_0)]^{-1}=\sigma_f^2,
    \end{align*}
    where (a) and (b) utilize~\eqref{def:alpha}, and (c) follows from~\eqref{eq:t}. Before proceeding to~\eqref{assum:fbddvar}, we first notice that
    \begin{align}
        &\left\{\tfrac{1}{\alpha}\left[f_{t}(x)+\tfrac{|t|s_0}{2}\right]-f_{t}(x)\right\}^2\overset{(a)}{<}\left\{\tfrac{1}{\alpha}\left[f_{t}(x)+|t|s_0\right]-f_{t}(x)-|t|s_0\right\}^2=\left(\tfrac{1}{\alpha}-1\right)^2\left[f_{t}(x)+|t|s_0\right]^2 \nn\\
        &\overset{(b)}{\leq}\tfrac{(1-\alpha)^2}{\alpha^2}\left[|t|D+|t|s_0\right]^2\overset{(c)}{\leq}\tfrac{1}{\alpha^2}t^2(D+s_0)^2,
        \label{eq:subresult1}
    \end{align}
    where (a) and (c) follow from~\eqref{def:alpha}, and (b) is due to $0\leq f_{t}(x)\leq |t|D$ over $X$. Then, for~\eqref{assum:fbddvar}, we have
    \begin{align*}
        &\E\left[(F_{t}(x,\xi)-f_{t}(x))^2\right]=(1-\alpha)\left[-\tfrac{|t|s_0}{2(1-\alpha)}-f_{t}(x)\right]^2+\alpha\left\{\tfrac{1}{\alpha}\left[f_{t}(x)+\tfrac{|t|s_0}{2}\right]-f_{t}(x)\right\}^2 \\
        &\overset{(a)}{<}\!(1-\alpha)\left[|t|s_0+f_{t}(x)\right]^2+\tfrac{1}{\alpha}t^2(D+s_0)^2\overset{(b)}{<}\left(1+\tfrac{1}{\alpha}\right)t^2(D+s_0)^2\overset{(c)}{=}\left\{1+|t|^{-\tfrac{4}{3}}L^{\tfrac{2}{3}}[2(D+s_0)]^{-\tfrac{2}{3}}\right\}t^2\left(D+s_0\right)^2 \\
        &=\left\{t^2+|t|^{\tfrac{2}{3}}L^{\tfrac{2}{3}}[2(D+s_0)]^{-\tfrac{2}{3}}\right\}\left(D+s_0\right)^2\overset{(d)}{\leq}\left(2^{-3}+2^{-\tfrac{5}{3}}\right)L(D+s_0)=\beta^2,
    \end{align*}
    where (a) follows from~\eqref{def:alpha} and~\eqref{eq:subresult1}, (b) utilizes~\eqref{def:alpha} and $0\leq f_{t}(x)\leq |t|D$ over $X$, (c) directly applies~\eqref{def:alpha}, and (d) follows from~\eqref{eq:t}. For Assumption~\ref{assumption1}, we have
    \begin{align*}
        &\E\left[(F_{t}(x,\xi)-f_{t}(x))_+^4\right]=\alpha\left\{\tfrac{1}{\alpha}\left[f_{t}(x)+\tfrac{|t|s_0}{2}\right]-f_{t}(x)\right\}^4\overset{(a)}{\leq}\alpha\left[\tfrac{1}{\alpha^2}t^2(D+s_0)^2\right]^2=\tfrac{1}{\alpha^3}t^4(D+s_0)^4 \\
        &\overset{(b)}{=}\left\{t^{-4}L^2[2(D+s_0)]^{-2}\right\}t^4(D+s_0)^4=\tfrac{1}{4}L^2(D+s_0)^2=M_f^4,
    \end{align*}
    where (a) follows from~\eqref{eq:subresult1}, and (b) directly applies~\eqref{def:alpha}. For Assumption~\ref{assumption2}, we have
    \begin{align}
        &\E\left[\left\|G_{t}'(x,\xi)\right\|_2^2\right]=\E\left[\left\{2\left[F_{t}(x,\xi)-f_{t}(x)\right]_+\left[F'_{t}(x,\xi)-f'_{t}(x)\right]\right\}^2\right] \nn\\
        &=\alpha\left\{2\left\{\tfrac{1}{\alpha}\left[f_{t}(x)+\tfrac{|t|s_0}{2}\right]-f_{t}(x)\right\}\left[\tfrac{1}{\alpha}f_{t}'(x)-f_{t}'(x)\right]\right\}^2=4\alpha\left\{\tfrac{1}{\alpha}\left[f_{t}(x)+\tfrac{|t|s_0}{2}\right]-f_{t}(x)\right\}^2\left[\tfrac{1}{\alpha}f_{t}'(x)-f_{t}'(x)\right]^2 \nn\\
        &\overset{(a)}{\leq} 4\alpha\left[\tfrac{1}{\alpha^2}t^2(D+s_0)^2\right]\left[\tfrac{1}{\alpha}f'_{t}(x)\right]^2=\tfrac{4}{\alpha^3}t^4(D+s_0)^2\overset{(b)}{=}\left\{4t^{-4}L^2[2(D+s_0)]^{-2}\right\}t^4(D+s_0)^2=L^2=L_G^2,
        \label{lower:verifyassump2}
    \end{align}
    where (a) follows from~\eqref{def:alpha} and~\eqref{eq:subresult1}, and (b) directly applies~\eqref{def:alpha}.
\end{proof}
\vgap

The result below illustrates the relation between the optimality gaps for $\min_{x\in X}h_{t}(x)$ in ~\eqref{problem:h_lower} and the corresponding problem of minimizing the upper semideviation part only, i.e., $\min_{x\in X}g_{t}(x)$.
\begin{lemma}
    For any problem instance from~\eqref{problem:h_lower} with parameter $t$, the following results hold:
    
    (a) $g_{t}(x)$ and $h_{t}(x)$ are both monotone and minimized at $x^*=d_t$ (see the definition of $d_t$ in~\eqref{eq:d}).
    
    (b) For any $x\in X$ such that $h_{t}(x)-h_{t}(x^*)\leq\epsilon$, we have $g_{t}(x)-g_{t}(x^*)\leq\max\{2^{3/2}g_{t}^{1/2}(x^*)\epsilon,8\epsilon^2\}$.
    \label{lem:lower2}
\end{lemma}
\begin{proof}
    It suffice to consider the case $t>0$ only, because the case $t<0$ can be shown similarly, due to the symmetry of the functions $f_{t},g_{t}$ and $h_{t}$ with respect to $x=0$. For any $t>0$, $f_{t}(x)$ monotonically decreases over $X$ and is uniquely minimized at $x=x^*=D/2$. Note that $d_t=D/2$ since $t>0$, and hence we have
    \begin{align}
        g_{t}(x)&=\E\left[\left[F_{t}(x,\xi)-f_{t}(x)\right]_+^2\right]=\alpha\left\{\tfrac{1}{\alpha}\left[f_{t}(x)+\tfrac{|t|s_0}{2}\right]-f_{t}(x)\right\}^2=\tfrac{(1-\alpha)^2}{\alpha}\left[f_{t}(x)+\tfrac{|t|}{2(1-\alpha)} s_0\right]^2 \nn\\
        &=\tfrac{(1-\alpha)^2}{\alpha}\left[|t|\left(\tfrac{D}{2}-x\right)+|t|\tfrac{s_0}{2(1-\alpha)}\right]^2=\tfrac{(1-\alpha)^2}{\alpha}t^2\left[\tfrac{D}{2}+\tfrac{s_0}{2(1-\alpha)}-x\right]^2.
        \label{eq:gvalue_lower}
    \end{align}
    It is easy to see from above relation that $g_{t}(x)$ also monotonically decreases over $X$ and is uniquely minimized at $x=x^*=D/2$. Therefore, by~\eqref{problem:h_lower}, $h_{t}$ monotonically decreases over $X$ and is uniquely minimized at $x=x^*=D/2$. Consider any $\Bar{x}\in X$  ($\Bar{x}\neq x^*$)  such that $0<h_{t}(\Bar{x})-h_{t}(x^*)\leq\epsilon$ for some $\epsilon>0$ and denote $\nu_{\Bar{x}}=g_{t}(\Bar{x})-g_{t}(x^*)>0$.
     Let $f_0(x):=x^{1/2}$ for $x \in [0,\infty)$.  We have
    \begin{align*}
        \epsilon&\geq h_{t}(\Bar{x})-h_{t}(x^*)=f_0(g_{t}(\Bar{x}))+f_{t}(\Bar{x})-f_0(g_{t}(x^*))-f_{t}(x^*)\geq f_0(g_{t}(\Bar{x}))-f_0(g_{t}(x^*)) \nn\\
        &\geq\nabla f_0(g_{t}(\Bar{x}))(g_{t}(\Bar{x})-g_{t}(x^*))=\nabla f_0(g_{t}(x^*)+\nu_{\Bar{x}})\nu_{\Bar{x}}=\tfrac{\nu_{\Bar{x}}}{2\left[g_{t}(x^*)+\nu_{\Bar{x}}\right]^{1/2}}\geq\tfrac{\nu_{\Bar{x}}}{2\left[2\max\{g_{t}(x^*),\nu_{\Bar{x}}\}\right]^{1/2}},
    \end{align*}
    where the first equality follows from~\eqref{problem:h_lower}, the second inequality follows from the monotonicity of $f_{t}$, i.e., $f(\Bar{x})\geq f(x^*)$, and the third inequality follows from the concavity of $f_0$. It then follows from the above equation that $\nu_{\Bar{x}}\leq 2[2g_{t}(x^*)]^{1/2}\epsilon$ holds if $g_{t}(x^*)\geq\nu_{\Bar{x}}$ is satisfied, and $\nu_{\Bar{x}}\leq 8\epsilon^2$ holds if $g_{t}(x^*)<\nu_{\Bar{x}}$, thus we have
    \begin{align}
        \nu_{\Bar{x}}\leq\max\left\{2^{\tfrac{3}{2}}g_{t}^{\tfrac{1}{2}}(x^*)\epsilon,8\epsilon^2\right\}.
        \label{eq:nuepsilon_lower}
    \end{align}
\end{proof}
\vgap

In view of Lemma~\ref{lem:lower2}, to have $g_{t}(\Bar{x})-g_{t}(x^*)\leq\max\{2^{3/2}g_{t}^{1/2}(x^*)\epsilon,8\epsilon^2\}=:\nu$ is necessary for $h_{t}(\Bar{x})-h_{t}(x^*)\leq\epsilon$. Thus, in order to attain an $\epsilon$-optimal solution for problem~\eqref{problem:h_lower}, we need to at least find a $\nu$-optimal solution for problem $\min_{x\in X}g_{t}(x)$. Note also that by~\eqref{eq:gvalue_lower}, we have $g_{t}(x^*)=0$ when $s_0=0$ holds, and $g_{t}(x^*)\to\infty$ when $s_0\to\infty$. Therefore, different choices of $s_0\in[0,\infty)$ for the class of problems~\eqref{problem:h_lower} result in different instance values $g_{t}(x^*)\in[0,\infty)$.

\vgap

The result below provides a lower complexity bound for solving the class of problems $\min_{x\in X}g_{t}(x)$ based on the argument in~\cite[Section 5.3]{nemirovskij1983problem}.
\begin{lemma}
    Any method with a deterministic rule for the sample trajectories requires at least $N\geq\Omega\left(L^2D^2\nu^{-2}\right)$ calls to the stochastic first-order oracle $\cO_f$ to find a $\nu$-optimal solution for some problem instance from the class of problem instances $\min_{x\in X}g_{t}(x)$.
    \label{lem:lower3}
\end{lemma}
\begin{proof}
    First, we let $\mathcal{A}_g$ be any method with a deterministic rule for the sample trajectories, such that for any problem instance $\min_{x\in X}g_{t}(x)$ with $t$ satisfying~\eqref{eq:t}, given any initial point $x_0\in X$, $\mathcal{A}_g$ outputs a $\nu$-optimal solution $x \in X$, i.e., $\E[g_{t}(x)]-g_{t}(x^*)\leq\nu$, after $N$ queries of the stochastic first-order oracle $\mathcal{O}_f$.
    
    Now, for any $t>0$ that satisfies~\eqref{eq:t}, we consider a pair of problem instances $\min_{x\in X}g_{t}(x)$ and $\min_{x\in X}g_{-t}(x)$. For each of the two instances, in view of the definition of $\cO_f$ in~\eqref{def:Of}, the probability that the method $\mathcal{A}_g$ receives the constant information from all of the $N$ queries of $\mathcal{O}_f$ is
    \begin{align}
        (1-\alpha)^N=\left\{1-|t|^{\tfrac{4}{3}}L^{-\tfrac{2}{3}}[2(D+s_0)]^{\tfrac{2}{3}}\right\}^N,
        \label{def:pxic_lower}
    \end{align}
    and the information $(F_{t}(x,\xi),F'_{t}(x,\xi))=(-ts_0/[2(1-\alpha)],0)$ in each query of $\cO_f$ is the same for the two instances. Therefore, for each of the two instances, given the initial point $x_0$, $\mathcal{A}_g$ outputs the same solution $\Bar{x}\in X$ since it receives the same constant information from $N$ queries of $\cO_f$.

    Below we first show that, either the problem instance $g=g_{t}$ or $g=g_{-t}$ satisfies
    \begin{align}
        g(\Bar{x})-\min_{x\in X}g(x)\geq g_{t}(0)-g_{t}(d_t).
        \label{eq:glowerbound}
    \end{align}
    We will show the above result by considering two cases. Let us first suppose that $\Bar{x}\leq 0$. In this case, we have $g_{t}(\Bar{x})-\min_{x\in X}g_{t}(x)\geq g_{t}(0)-g_{t}(d_t)$, which follows from Lemma~\ref{lem:lower2}(a). Now we suppose that $\Bar{x}>0$, and we have
    \begin{align*}
        g_{-t}(\Bar{x})-\min_{x\in X}g_{-t}(x)=g_{-t}(\Bar{x})-g_{-t}(d_{-t})=g_{t}(-\Bar{x})-g_{t}(d_t)\geq g_{t}(0)-g_{t}(d_t),
    \end{align*}
    where the first equality and the inequality follow from Lemma~\ref{lem:lower2}(a), and the second equality follows from the facts that $d_t=-d_{-t}$ (see~\eqref{eq:d}), and that $g_{t}$ is symmetric to $g_{-t}$ with respect to $x=0$ (see~\eqref{eq:gsymmetry}).
    
    Now we focus on the problem instance $\min_{x\in X}g(x)$. It follows from~\eqref{eq:glowerbound} that
    \begin{align}
        &g(\Bar{x})-\min_{x\in X}g(x)\geq g_{t}(0)-g_{t}(d_t)\overset{(a)}{=}\tfrac{(1-\alpha)^2}{\alpha}t^2\left\{\left[\tfrac{D}{2}+\tfrac{s_0}{2(1-\alpha)}\right]^2-\left[\tfrac{s_0}{2(1-\alpha)}\right]^2\right\} \nn\\
        &\overset{(b)}{>}\tfrac{1}{16\alpha}t^2\left[\left(D+\tfrac{s_0}{1-\alpha}\right)^2-\left(\tfrac{s_0}{1-\alpha}\right)^2\right]\overset{(c)}{=}2^{-\tfrac{14}{3}}|t|^{\tfrac{2}{3}}L^{\tfrac{2}{3}}(D+s_0)^{-\tfrac{2}{3}}\left[\left(D+\tfrac{s_0}{1-\alpha}\right)^2-\left(\tfrac{s_0}{1-\alpha}\right)^2\right] \nn\\
        &=2^{-\tfrac{14}{3}}|t|^{\tfrac{2}{3}}L^{\tfrac{2}{3}}(D+s_0)^{-\tfrac{2}{3}}D\left(D+\tfrac{2}{1-\alpha}s_0\right)\overset{(d)}{\geq}2^{-\tfrac{14}{3}}|t|^{\tfrac{2}{3}}L^{\tfrac{2}{3}}D\left(D+\tfrac{2}{1-\alpha} s_0\right)^{\tfrac{1}{3}},
        \label{eq1lower}
    \end{align}
    where (a) applies~\eqref{eq:gvalue_lower}, (b) and (c) follow from~\eqref{def:alpha}, and (d) holds due to $2/(1-\alpha)>1$. For any $\nu$ such that $0<\nu\leq 2^{-20/3}LD$, we take
    \begin{align}
        t=2^{\tfrac{17}{2}}\nu^{\tfrac{3}{2}}L^{-1}D^{-\tfrac{3}{2}}\left(D+\tfrac{2}{1-\alpha}s_0\right)^{-\tfrac{1}{2}}\leq\tfrac{1}{2}L^{\tfrac{1}{2}}[2(D+s_0)]^{-\tfrac{1}{2}},
        \label{eq:tnu_lower}
    \end{align}
    which satisfies~\eqref{eq:t}, then it follows from~\eqref{eq1lower} that
    \begin{align}
        g(\Bar{x})-\min_{x\in X}g(x)>2\nu.
        \label{eq2lower}
    \end{align}
    
    For the problem instance $\min_{x\in X}g(x)$, suppose that given the initial point $x_0\in X$, the method $\mathcal{A}_g$ outputs a $\nu$-optimal solution $\Tilde{x}$, i.e., $\E[g(\Tilde{x})]-g(x^*)\leq\nu$, after $N$ queries of $\cO_f$. Let $\Xi^N$ denote the set of all possible sample trajectories by the $N$ queries of $\cO_f$, where $\tsum_{\xi^N\in\Xi^N}\P(\xi^N)=1$. For any $\xi^N\in\Xi^N$, we suppose $\mathcal{A}_g$ outputs the solution $x_{\xi^N}$. Specifically, let $\xi_c^N\in\Xi^N$ denote the sample trajectory where $\mathcal{A}_g$ receives the constant information from all of the $N$ queries of $\cO_f$. Note that we have $\P(\xi_c^N)=(1-\alpha)^N$ and $x_{\xi_c^N}=\Bar{x}$. It then follows from~\eqref{def:pxic_lower} and~\eqref{eq2lower} that
    \begin{align*}
        \nu&\geq\E[g(\Tilde{x})]-\min_{x\in X}g(x)=\tsum_{\xi^N\in\Xi^N}\P(\xi^N)\left[g(x_{\xi^N})-\min_{x\in X}g(x)\right]\geq\P(\xi_c^N)\left[g(x_{\xi_c^N})-\min_{x\in X}g(x)\right] \\
        &=(1-\alpha)^N\left[g(\Bar{x})-\min_{x\in X}g(x)\right]>(1-\alpha)^N\times 2\nu\geq 2\nu(1-N\alpha),
    \end{align*}
    which indicates $1-N\alpha<1/2$. Thus, utilizing~\eqref{def:alpha} and~\eqref{eq:tnu_lower}, we have
    \begin{align}
        N>\tfrac{1}{2\alpha}=\tfrac{1}{2}|t|^{-\tfrac{4}{3}}L^{\tfrac{2}{3}}[2(D+s_0)]^{-\tfrac{2}{3}}>2^{-13}L^2D^2\nu^{-2}.
        \label{eq:Nnu_lower}
    \end{align}
\end{proof}
\vgap

Based on the previous results, we provide an instance-dependent lower complexity bound for problem~\eqref{l2problem}.
\begin{theorem}
    Any method with a deterministic rule for the sample trajectories requires at least
    \begin{align*}
        N\geq\Omega\left(L_G^2D_X^2\left(\max\left\{g^{\tfrac{1}{2}}(x^*),\epsilon\right\}\right)^{-2}\epsilon^{-2}\right)
    \end{align*}
    calls to the stochastic first-order oracle to find an $\epsilon$-optimal solution for some instance of the mean-upper-semideviation problem~\eqref{l2problem} that satisfies~\eqref{assum:flipschitz},~\eqref{assum:fgradbddvar},~\eqref{assum:fbddvar} and Assumptions~\ref{assumption1} and~\ref{assumption2}, where $L_G$ is defined in Assumption~\ref{assumption2} and $D_X := \max_{x_1, x_2\in X} \|x_1-x_2\|$.
    \label{thm:lower}
\end{theorem}
\begin{proof}
    The proof is based on Lemma~\ref{lem:lower1}, Lemma~\ref{lem:lower2} and Lemma~\ref{lem:lower3}. We have shown in Lemma~\ref{lem:lower1} that the class of problem instances~\eqref{problem:h_lower} satisfies~\eqref{assum:flipschitz},~\eqref{assum:fgradbddvar},~\eqref{assum:fbddvar} and Assumptions~\ref{assumption1} and~\ref{assumption2}. Let $\mathcal{A}_h$ be any method with a deterministic rule for the sample trajectories, such that given any initial point $x_0\in X$, for any instance $\min_{x\in X}h_{t}(x)$ from the class of problem instances~\eqref{problem:h_lower}, $\mathcal{A}_h$ attains an $\epsilon$-optimal solution $\Tilde{x}$ after $N$ queries of the stochastic first-order oracle $\mathcal{O}_f$. According to Lemma~\ref{lem:lower2} and the discussion that follows, $\Tilde{x}$ is at least a $\max\{2^{3/2}g_{t}^{1/2}(x^*)\epsilon,8\epsilon^2\}$-optimal solution for the problem instance $\min_{x\in X}g_{t}(x)$. It then directly follows from Lemma~\ref{lem:lower3} that, for some instance from the class of problem instances $\min_{x\in X}g_{t}(x)$, we have
    \begin{align*}
        N>2^{-13}L^2D^2\left\{\max\left\{2^{\tfrac{3}{2}}g_{t}^{\tfrac{1}{2}}(x^*),8\epsilon\right\}\right\}^{-2}\epsilon^{-2},
    \end{align*}
    where $L^2=L_G^2$ holds by~\eqref{lower:verifyassump2} in Lemma~\ref{lem:lower1}, and $D=D_X$.
\end{proof}
\vgap

Notice that the instance-dependent value $z^*$ (see~\eqref{def:yzstar}) in the oracle complexity result~\eqref{simplifycomplexityresult} satisfies $z^*=2g^{1/2}(x^*)$, and that $g^{1/2}(x^*)$ also appears in the lower complexity bound in Theorem~\ref{thm:lower}. Therefore, we conclude that the oracle complexity result in~\eqref{simplifycomplexityresult} is nearly optimal up to the logarithmic factor and some other constant factors. Nevertheless, we should acknowledge that the lower complexity bound constructed in Theorem~\ref{thm:lower} does not explicitly include the problem parameters $\beta,L_f,\sigma_f$ and $M_f$ that appear in the oracle complexity result~\eqref{complexityresult}  (see~\eqref{def:C1C2}), though their relationships with $L_G$ and $D_X$ have been illustrated in Lemma~\ref{lem:lower1} for the specific problem instance. It is worth noting that although the sub-Gaussian assumptions (see Assumption~\ref{assumption3}) used to derive the high probability result in Section~\ref{sec:highprob} are not satisfied for the class of problem instances~\eqref{problem:h_lower}, the lower complexity bound still nearly matches the high probability complexity results established in Section~\ref{sec:highprobRDA}, which no longer rely on these assumptions.
\section{General $L_p$ Risk Minimization Problem}

In this section, we extend our prior results to the general $L_p$ risk minimization problem~\eqref{lpproblem} with any $p>1$. Specifically, we show our lifting formulation, convergence results and a lower complexity bound for problem~\eqref{lpproblem}.

\subsection{Lifting Formulation for $L_p$ Risk Measure}

Suppose that the conditions~\eqref{assum:flipschitz},~\eqref{assum:fgradbddvar} and~\eqref{assum:fbddvar} still hold, and we condition on~\eqref{eq:Ycontainy}. Similar to the reformulation~\eqref{eq:l2equiform} for the $L_2$ problem~\eqref{l2problem}, we provide the following convex-concave stochastic saddle point reformulation that is equivalent to the $L_p$ problem~\eqref{lpproblem} by using the lifting strategy:
\begin{align}
    \inf_{x\in X,y\in Y,z>0}\max_{\lambda\in\Lambda}L(x,y,z,\lambda)=\tfrac{c}{z^{p-1}}\E\left[(F(x,\xi)-y)_+^p\right]+y+c(p-1)p^{-\tfrac{p}{p-1}}z+\lambda(\E[F(x,\xi)]-y),
    \label{lp:minmax}
\end{align}
where $Y$ is defined in~\eqref{def:ydelta}, and $\Lambda=[0,1]$. Moreover, the saddle point $(x^*,y^*,z^*,\lambda^*)$ of problem~\eqref{lp:minmax} satisfies:
\begin{align}
    x^*\in X, \ y^*=\E[F(x^*,\xi)]\in Y, \ z^*=p^{\tfrac{1}{p-1}}\E^{\tfrac{1}{p}}\left[(F(x^*,\xi)-\E[F(x^*,\xi)])_+^p\right] \quad \text{and} \quad \lambda^*\in\Lambda,
    \label{lp:def:yzstar}
\end{align}
where $x^*$ is the optimal solution for the $L_p$ problem~\eqref{lpproblem}.

Next, similar to Assumptions~\ref{assumption1} and~\ref{assumption2} for the $L_2$ problem~\eqref{l2problem}, we make the following two assumptions for the $L_p$ problem~\eqref{lpproblem}.

\begin{assumption}
    $\E[(F(x,\xi)-\E[F(x,\xi)])_+^{2p}]\leq M_f^{2p}$ for any $x\in X$.
    \label{assumption4}
\end{assumption}
\begin{assumption}
    Define $G(x,\xi):=(F(x,\xi)-\E[F(x,\xi)])_+^p$, then $\E[\|G'(x,\xi)\|_{\star,X}^2]\leq L_G^2$ for any $x\in X$.
    \label{assumption5}
\end{assumption}

\subsection{Convergence Results}

To solve problem~\eqref{lp:minmax}, we still consider the min-min-max formulation in~\eqref{l2minminmaxproblem} and apply Algorithm~\ref{algo:PB}. First, we redefine the vector in~\eqref{def:lzsubgrad} as
\begin{align}
    l'_z(x,y,\lambda,\xi):=\left[
    \begin{array}{c}
    \cL'_x(x,y,z,\lambda,\xi) \\
    \cL'_y(x,y,z,\lambda,\xi) \\
    -\cL'_\lambda(x,y,z,\lambda,\xi)
    \end{array}
    \right]=\left[
    \begin{array}{c}
    \tfrac{cp}{z^{p-1}}(F(x,\xi)-y)_+^{p-1}F'(x,\xi)+\lambda F'(x,\xi) \\
    -\tfrac{cp}{z^{p-1}}(F(x,\xi)-y)_+^{p-1}+1-\lambda \\
    -F(x,\xi)+y
    \end{array}
    \right].
    \label{lp:def:lzsubgrad}
\end{align}
Recall that the set $U=X\times Y\times\Lambda$. Moreover, compared to the definition of $M_z^2$ in~\eqref{def:M^2}, we redefine $M_z^2$ as
\begin{align}
    M_z^2:=&\tfrac{3c^2}{z^{2p-2}}\max\{2^{p-2},1\}^2\left\{L_G^2+p^2(L_f^2+1)M_f^{2p-2}+p^2\left[2^{2p}(L_f^2+\sigma_f^2)+2^{2p-2}\right](L_fD_X+1)^{2p-2}\right\} \nn\\
    &+12(L_f^2+\sigma_f^2)+8(L_fD_X+1)^2+2\beta^2+3.
    \label{lp:def:Mz}
\end{align}
For any $u=(x,y,\lambda)\in U$, one can verify that $\E[\|l'_z(x,y,\lambda,\xi)\|_*^2]\leq M_z^2$ by following a similar argument as in the proof of Lemma~\ref{lem:Lipsboundnum} and applying the following elementary inequalities: (1) $(a+b)^r\leq 2^{r-1}(a^r+b^r)$ for any $a,b>0$ and $r\geq 1$; (2) $(a+b)^r\leq a^r+b^r$ for any $a,b>0$ and $r\in (0,1)$. Notice that $M_z^2$ in~\eqref{lp:def:Mz} is in the order of $\cO(\max\{z^{-2p+2},1\})$, which, by applying $p=2$, matches the order $\cO(\max\{z^{-2},1\})$ in~\eqref{def:M^2}. Below we establish the convergence for Algorithm~\ref{algo:PB} applied to the $L_p$ problem~\eqref{lpproblem}.

\begin{theorem}
    Under Assumptions~\ref{assumption4} and~\ref{assumption5}, for any $\alpha_0\in (0,1)$ and $\alpha\in(0,1-\alpha_0)$, let
    \begin{align}
        K:=\left\lceil\log_2\left\{\tfrac{cb_0}{\epsilon}\left\{\tfrac{(p-1)2^{p+1}\left[M_f^p+2^p\left(L_fD_X+1\right)^p\right]}{\max\{z^*/2,p^{1/(p-1)}\epsilon/(4c)\}^p}+16p^{-\tfrac{1}{p-1}}\right\}\right\}\right\rceil
        \label{lp:def:K}
    \end{align}
    denote the maximum number of the outer iterations in Algorithm~\ref{algo:PB}. Also assume that the stopping criterion~\eqref{simplestopping} is applied in Algorithm~\ref{algo:PB}, and the algorithmic parameters in Algorithm~\ref{algo:PB} are set or modified to
    \begin{align}
        m\!=\!\left\lceil\!\tfrac{\beta^2}{\alpha_0}\!\right\rceil, \theta\!=\!\tfrac{p^{1\!/\!(p\!-\!1)}\epsilon}{2c}, b_0\!\geq\!\max\left\{p^{\tfrac{1}{p\!-\!1}}\!M_f,\theta\right\}, T_k\!=\!\left\lceil\!\max\left\{\left(\tfrac{36K}{\alpha}\!D_{V,x_0}\!M_{z_k}\right)^2,\tfrac{32K}{\alpha}\!\Bar{M}_{z_k}^2\right\}\epsilon^{-2}\!\right\rceil \text{ and } \gamma_{k,t}\!=\!\tfrac{D_{V,x_0}}{M_{z_k}\!\sqrt{T_k}},
        \label{lp:parathm1}
    \end{align}
    where
    \begin{align}
        \Bar{M}_z^2:=\tfrac{c^2}{z^{2p-2}}(p-1)^22^{2p-1}\left[M_f^{2p}+2^{2p}(L_fD_X+1)^{2p}\right],
        \label{lp:def:BarMz}
    \end{align}
    and $D_{V,x_0}$ and $M_{z_k}$ are defined in~\eqref{def:D} and~\eqref{lp:def:Mz}, respectively. Then an $\epsilon$-optimal solution for problem~\eqref{lpproblem} can be attained within at most $m+2K\Tilde{T}$ number of oracle calls with probability at least $1-\alpha_0-\alpha$, where
    \begin{align*}
        \Tilde{T}=\left\lceil\max\left\{\left(\tfrac{36K}{\alpha}D_{V,x_0}M_{z_{\min}}\right)^2,\tfrac{32K}{\alpha}\Bar{M}_{z_{\min}}^2\right\}\epsilon^{-2}\right\rceil \quad \text{and} \quad z_{\min}=\max\left\{\tfrac{z^*}{2},\tfrac{p^{1/(p-1)}\epsilon}{4c}\right\}.
    \end{align*}
    \label{lp:thm:converge1}
\end{theorem}
\begin{proof}
    This proof is similar to that of Theorem~\ref{thm:converge1}, so we only highlight the key differences.
    
    First, notice that the parameter $\theta=p^{1/(p-1)}\epsilon/(2c)$ in~\eqref{lp:parathm1} generalizes the choice $\theta=\epsilon/c$ for the special case $p=2$. This is because, in the general $L_p$ setting, one can follow the same reasoning as in~\eqref{eq:opt2} of Lemma~\ref{lem:opt} to show that $\E[h(\Bar{x}_k)]-h(x^*)\leq\epsilon$ holds when Algorithm~\ref{algo:PB} terminates with $z^*\leq z_k=b_{k+1}<p^{1/(p-1)}\epsilon/(2c)$.
    
    Second, we set $b_0\geq\max\{p^{1/(p-1)}M_f,\theta\}$ as in~\eqref{lp:parathm1}, since $z^*\leq p^{1/(p-1)}M_f$ follows directly from the definition of $z^*$ in~\eqref{lp:def:yzstar} and Assumption~\ref{assumption4}.
    
    Next, observe that the definition of $\Bar{M}_z^2$ in~\eqref{lp:def:BarMz} differs from that in~\eqref{def:barMz} of Lemma~\ref{lem:induction1}. To understand this difference, recall from~\eqref{Mzbartilde_relation} the relationship between $\Bar{M}_{z_k}^2$ and $\tilde{M}_{z_k}^4$, where $\tilde{M}_z^4$, defined in~\eqref{def:tildeMz} of Lemma~\ref{lem:subdiff}, serves as an upper bound for $\mathrm{Var}(\cL'_z(x,y,z,\lambda,\xi))$ (see~\eqref{eq2.5lem2algo}). In the $L_p$ case, however, $\Bar{M}_{z_k}^2$ corresponds to the upper bound for the variance of the generalized $\cL'_z(x,y,z,\lambda,\xi)$.
    
    Finally, $K$ defined in~\eqref{lp:def:K} is a natural extension of~\eqref{Kthm1} in Theorem~\ref{thm:converge1}. The difference arises because we follow a similar way to~\eqref{eq1thm1} to bound the generalized $|L'_z(x,y,z,\lambda)|$ in the $L_p$ case, and adopt the generalized parameter choice $\theta=p^{1/(p-1)}\epsilon/(2c)$.
\end{proof}
\vgap

We also provide the convergence guarantees for Algorithm~\ref{algo:PB} when the enhanced stopping criterion~\eqref{originalstop}-\eqref{stopcriteria} is adopted.

\begin{theorem}
    Under Assumptions~\ref{assumption4} and~\ref{assumption5}, for any $\alpha_0\in (0,1)$ and $\alpha\in(0,1-\alpha_0)$, suppose that the algorithmic parameters in Algorithm~\ref{algo:PB} are set or modified to
    \begin{align*}
        m\!=\!\left\lceil\!\tfrac{\beta^2}{\alpha_0}\!\right\rceil, \theta\!=\!\tfrac{p^{1\!/\!(p\!-\!1)}\epsilon}{2c}, b_0\!\geq\!\max\left\{p^{\tfrac{1}{p\!-\!1}}\!M_f,\theta\right\}, T_k\!=\!\left\lceil\!\max\left\{\left(\tfrac{216K}{\alpha}\!D_{V,x_0}\!M_{z_k}\right)^2,\tfrac{48K}{\alpha}\!\Bar{M}_{z_k}^2\right\}\epsilon^{-2}\!\right\rceil \text{ and } \gamma_{k,t}\!=\!\tfrac{D_{V,x_0}}{M_{z_k}\!\sqrt{T_k}},
    \end{align*}
    where $D_{V,x_0},M_{z_k}$ and $\Bar{M}_{z_k}$ are defined in~\eqref{def:D},~\eqref{lp:def:Mz} and~\eqref{lp:def:BarMz}, respectively. Moreover, assume that the maximum number of outer iterations in Algorithm~\ref{algo:PB} is set to
    \begin{align*}
        K=\left\lceil\log_2\left\{\tfrac{cb_0}{\epsilon}\left\{\tfrac{(p-1)2^{p+1}\left[M_f^p+2^p\left(L_fD_X+1\right)^p\right]}{\max\{z^*/2,p^{1/(p-1)}\epsilon/(4c)\}^p}+16p^{-\tfrac{1}{p-1}}\right\}\right\}\right\rceil
    \end{align*}
    and the stopping criterion~\eqref{originalstop}-\eqref{stopcriteria} is applied, where $\hat{\psi}(\cdot)$ is computed using~\eqref{eq:hatpsi} with
    \begin{align}
        T'_k=\left\lceil\tfrac{1536K}{\alpha}\left\{\tfrac{c^2}{z_k^{2p-2}}2^{2p-1}\left[M_f^{2p}+2^{2p}(L_fD_X+1)^{2p}\right]+\beta^2\right\}\epsilon^{-2}\right\rceil.
        \label{lp:T'kthmstop}
    \end{align}
    Then Algorithm~\ref{algo:PB} must terminate in some outer iteration $\tilde{k}\leq K-1$, and with probability at least $1-\alpha_0-\alpha$, an $\epsilon$-optimal solution for problem~\eqref{lpproblem} is attained in some outer iteration $k_0\leq\tilde{k}$, and $z_k\geq\max\{z^*/4,p^{1/(p-1)}\epsilon/(4c)\}$ for $k\in\{0,...,\tilde{k}\}$.
    \label{lp:thm:stop}
\end{theorem}
\begin{proof}
    This proof is similar to that of Theorem~\ref{thm:stop}; we therefore highlight only the key differences. Note that the parameters $\theta,b_0$ and $K$, which differ from those in the case $p=2$, have already been discussed in the proof of Theorem~\ref{lp:thm:converge1}. Therefore, we focus here on the choice of $T'_k$ in~\eqref{lp:T'kthmstop}, which generalizes~\eqref{Tk'thm2} in Theorem~\ref{thm:stop}. Specifically, this extension follows from the derivations in~\eqref{eq:stopvar} of Lemma~\ref{lem:stop1}, utilizing a similar approach to obtain an upper bound for the variance of the generalized $\cL(x,y,z,\lambda,\xi)$.
\end{proof}
\vgap

Notice that by Theorem~\ref{lp:thm:converge1} and Theorem~\ref{lp:thm:stop}, the oracle complexity for Algorithm~\ref{algo:PB} applied to problem~\eqref{lpproblem} is bounded by
\begin{align}
    \tilde{\cO}\left(\max\{z^*,\epsilon\}^{-2p+2}\epsilon^{-2}\right),
    \label{lp:simplifycomplexity}
\end{align}
which reduces to the oracle complexity result~\eqref{simplifycomplexityresult} for solving the $L_2$ problem~\eqref{l2problem} by setting $p=2$. We will show that~\eqref{lp:simplifycomplexity} is nearly optimal in the next subsection. Moreover, for the $L_p$ problem~\eqref{lpproblem}, the high probability convergence can be established under certain sub-Gaussian assumptions, or by using multiple sample trajectories together with the robust distance approximation, which are similar to the results in Theorem~\ref{thm:converge1highprob} and Theorem~\ref{thm:highprobRDA} for the $L_2$ problem~\eqref{l2problem}, respectively.

\subsection{Lower Complexity Bound}
\label{lp:subseclowerbound}

In order to provide a lower complexity result for solving the $L_p$ problem~\eqref{lpproblem}, we construct a specific class of instances, which is similar to that in Section~\ref{section:lower}. For notational simplicity, we denote $D=D_X$ and pick any constant $L>0$. Then for any $p>1$ and $s_0\geq 0$, a class of stochastic functions parameterized by $t$ is defined as
\begin{align*}
    f_{t}(x):=t\left(d_t-x\right)
\end{align*}
over the domain $X:=[-D/2,D/2]$, where 
\begin{align*}
    d_t:=\begin{cases}
        -\tfrac{D}{2}, & \text{if} \ t<0, \\
        \tfrac{D}{2}, & \text{if} \ t>0
    \end{cases}
\end{align*}
and the parameter $t$ satisfies
\begin{align*}
    t \in  \left(-\tfrac{1}{2}L^{\tfrac{1}{p}}\left[p^{\tfrac{1}{p-1}}(D+s_0)\right]^{\tfrac{1-p}{p}}, 0\right) \bigcup \left(0,\tfrac{1}{2}L^{\tfrac{1}{p}}\left[p^{\tfrac{1}{p-1}}(D+s_0)\right]^{\tfrac{1-p}{p}}\right).
\end{align*}
Moreover, we define a stochastic first-order oracle $\mathcal{O}_f$ that returns a pair of unbiased estimators $(F_{t}(x,\xi),F'_{t}(x,\xi))$ for $(f_{t}(x),f'_{t}(x))$. Specifically, $\mathcal{O}_f$ returns
\begin{align*}
    \left(F_{t}(x,\xi),F'_{t}(x,\xi)\right):=\begin{cases}
        \left(-\tfrac{|t|s_0}{2(1-\alpha)},0\right) & \text{with prob.} \ 1-\alpha, \\
        \tfrac{1}{\alpha}\left(f_{t}(x)+\tfrac{|t|s_0}{2},f_{t}'(x)\right) & \text{with prob.} \ \alpha,
    \end{cases}
\end{align*}
where
\begin{align*}
    \alpha := |t|^{\tfrac{2p}{2p-1}}L^{-\tfrac{2}{2p-1}}\left[p^{\tfrac{1}{p-1}}(D+s_0)\right]^{\tfrac{2p-2}{2p-1}}\in\left(0,\tfrac{1}{2}\right).
\end{align*}
Moreover, we define the upper semideviation of order $p$ as $g_t(x):=\E[G_t(x,\xi)]$, with
\begin{align*}
    G_{t}(x,\xi):=[F_{t}(x,\xi)-f_{t}(x)]_+^p.
\end{align*}
Finally, we construct a class of instances for the $L_p$ problem~\eqref{lpproblem} as
\begin{align}
    \min_{x\in X}\left\{h_{t}(x):=f_{t}(x)+g_{t}^{\tfrac{1}{p}}(x)\right\}.
    \label{lp:problem:h_lower}
\end{align}
With the help of the class of problem instances~\eqref{lp:problem:h_lower}, we are able to provide an instance-dependent lower complexity bound for problem~\eqref{lpproblem}.
\begin{theorem}
    Any method with a deterministic rule for the sample trajectories requires at least
    \begin{align*}
        N\geq\Omega\left(L_G^2D_X^2\left(\max\left\{g^{\tfrac{1}{p}}(x^*),\epsilon\right\}\right)^{-2p+2}\epsilon^{-2}\right)
    \end{align*}
    calls to the stochastic first-order oracle to find an $\epsilon$-optimal solution for some instance of the mean-upper-semideviation problem~\eqref{lpproblem} that satisfies~\eqref{assum:flipschitz},~\eqref{assum:fgradbddvar},~\eqref{assum:fbddvar} and Assumptions~\ref{assumption4} and~\ref{assumption5}, where $L_G$ is defined in Assumption~\ref{assumption5} and $D_X := \max_{x_1, x_2\in X} \|x_1-x_2\|$.
    \label{lp:thm:lower}
\end{theorem}
\begin{proof}
    This proof is similar to that of Theorem~\ref{thm:lower}, and hence the details are omitted.
\end{proof}
\vgap

Observe that the instance-dependent value $z^*$ (see~\eqref{lp:def:yzstar}) in the oracle complexity result~\eqref{lp:simplifycomplexity} satisfies $z^*=p^{1/(p-1)}g^{1/p}(x^*)$, and that $g^{1/p}(x^*)$ also appears in the lower complexity bound in Theorem~\ref{lp:thm:lower}. Therefore, the oracle complexity result in~\eqref{lp:simplifycomplexity} is nearly optimal up to the logarithmic factor and some other constant factors.
\section{Conclusion}

In this paper, we propose a novel lifting formulation for the $L_p$ risk minimization problem, and transform it to a convex-concave stochastic saddle point problem. Since the objective function is not uniformly Lipschitz continuous, we design a two-layer probabilistic bisection method to solve the stochastic saddle point problem and achieve an instance-dependent stochastic oracle complexity. Moreover, we develop an instance-dependent lower complexity bound for the $L_p$ risk minimization problem, which nearly matches the oracle complexity that we obtain and shows the near optimality of our method. For the future directions, it is interesting to conduct numerical experiments based on real-world datasets and evaluate the empirical performance of our method.
\section*{Appendix A: High Probability Guarantees under Sub-Gaussian Assumptions}

To show high probability convergence for Algorithm~\ref{algo:PB}, in view of Lemma~\ref{lem:subdiff}, we are essentially seeking a high probability guarantee for the relation $\psi^*\geq\psi(z_k)+\zeta_T(z^*-z_k)-\epsilon/2$ in each outer iteration $k$ of Algorithm~\ref{algo:PB}. To this end, under the sub-Gaussian Assumption~\ref{assumption3}, we first present some properties of $l'_z(x,y,\lambda,\xi)$ (see the definition in~\eqref{def:lzsubgrad}) and $\cL'_z(x,y,z,\lambda,\xi)$ (see the definition in~\eqref{eq:subgrad1}) 
in Lemma~\ref{lem:subgaussian} and 
Lemma~\ref{lem:zsubgaussian}, respectively.

\begin{lemma}
    Under Assumption~\ref{assumption3}, and conditioning on~\eqref{eq:Ycontainy}, for any $x\in X,y\in Y,z>0$ and $\lambda\in\Lambda$, we have
    \begin{align*}
        \E\left[\exp\left\{\left\|l'_z(x,y,\lambda,\xi)\right\|_*^2\Big/\sigma_z^2\right\}\right]\leq\exp\{1\},
    \end{align*}
    where $\sigma_z$ is defined in~\eqref{def:sigma}.
    \label{lem:subgaussian}
\end{lemma}
\begin{proof}
    In view of~\eqref{def:lzsubgrad}, $\|l'_z(x,y,\lambda,\xi)\|_*^2$ consists of three components, i.e.,
    \begin{align*}
        \left\|l'_z(x,y,\lambda,\xi)\right\|_*^2=\left\|\cL'_x(x,y,z,\lambda,\xi)\right\|_{*,X}^2+\left(\cL'_y(x,y,z,\lambda,\xi)\right)^2+\left(\cL'_\lambda(x,y,z,\lambda,\xi)\right)^2.
    \end{align*}
    Then, for any $x\in X,y\in Y,z>0$ and $\lambda\in\Lambda$, based on the procedures that we bound these three components in~\eqref{eq:Lygradbound},~\eqref{eq:Llamgradbound} and~\eqref{eq:Lxgradbound} respectively, we can derive
    \begin{align}
        \left\|l'_z(x,y,\lambda,\xi)\right\|_*^2=c_{z1}\left\|F'(x,\xi)\right\|_{*,X}^2+c_{z2}[F(x,\xi)-f(x)]^2+c_{z3}\left\|G'(x,\xi)\right\|_{*,X}^2+c_{z4},
        \label{eq1lem4algo}
    \end{align}
    where $c_{z1},c_{z2},c_{z3}$ and $c_{z4}$ are defined in~\eqref{def:c1234}. Therefore, we have
    \begin{align*}
        &\E\!\left[\exp\!\left\{\!\tfrac{\left\|l'_z(x,y,\lambda,\xi)\right\|_*^2}{\sigma_z^2}\!\right\}\right]\!=\!\E\!\left[\exp\!\left\{\!\tfrac{c_{z1}\left\|F'(x,\xi)\right\|_{*,X}^2}{\sigma_z^2}\!\right\}\!\times\!\exp\!\left\{\!\tfrac{c_{z2}[F(x,\xi)-f(x)]^2}{\sigma_z^2}\!\right\}\!\times\!\exp\!\left\{\!\tfrac{c_{z3}\left\|G'(x,\xi)\right\|_{*,X}^2}{\sigma_z^2}\!\right\}\!\times\!\exp\!\left\{\!\tfrac{c_{z4}}{\sigma_z^2}\!\right\}\right] \nn\\
        &\leq\E^{\tfrac{1}{4}}\!\left[\exp\!\left\{\!\tfrac{4c_{z1}\left\|F'(x,\xi)\right\|_{*,X}^2}{\sigma_z^2}\!\right\}\right]\!\times\!\E^{\tfrac{1}{4}}\!\left[\exp\!\left\{\!\tfrac{4c_{z2}[F(x,\xi)-f(x)]^2}{\sigma_z^2}\!\right\}\right]\!\times\!\E^{\tfrac{1}{4}}\!\left[\exp\!\left\{\!\tfrac{4c_{z3}\left\|G'(x,\xi)\right\|_{*,X}^2}{\sigma_z^2}\!\right\}\right]\!\times\!\E^{\tfrac{1}{4}}\!\left[\exp\!\left\{\!\tfrac{4c_{z4}}{\sigma_z^2}\!\right\}\right] \nn\\
        &\leq\E^{\tfrac{1}{4}}\left[\exp\left\{\tfrac{\left\|F'(x,\xi)\right\|_{*,X}^2}{\sigma_1^2}\right\}\right]\times\E^{\tfrac{1}{4}}\left[\exp\left\{\tfrac{[F(x,\xi)-f(x)]^2}{\sigma_2^2}\right\}\right]\times\E^{\tfrac{1}{4}}\left[\exp\left\{\tfrac{\left\|G'(x,\xi)\right\|_{*,X}^2}{\sigma_3^2}\right\}\right]\times\E^{\tfrac{1}{4}}\left[\exp\left\{1\right\}\right] \nn\\
        &\leq\left[\left(\exp\{1\}\right)^{\tfrac{1}{4}}\right]^4=\exp\{1\},
    \end{align*}
    where the equality follows from~\eqref{eq1lem4algo}, the first inequality utilizes H\"{o}lder's inequality, the second inequality applies~\eqref{def:sigma}, and the third inequality follows from Assumption~\ref{assumption3}.
\end{proof}
\vgap

\begin{lemma}
    Under Assumption~\ref{assumption3}, and conditioning on~\eqref{eq:Ycontainy}, we define
    \begin{align}
        \Tilde{\sigma}_z^4:=\tfrac{32c^2}{z^4}\max\left\{\sigma_4^4,32(L_fD_X+1)^4+\beta^4\right\},
        \label{def:tildesigma}
    \end{align}
    then for any $x\in X,y\in Y,z>0$ and $\lambda\in\Lambda$, we have
    \begin{align*}
        \E\left[\exp\left\{\left(\cL'_z(x,y,z,\lambda,\xi)-L'_z(x,y,z,\lambda)\right)^2\Big/\Tilde{\sigma}_z^4\right\}\right]\leq\exp\{1\}.
    \end{align*}
    \label{lem:zsubgaussian}
\end{lemma}
\begin{proof}
    In view of the definition of $\cL'_z(x,y,z,\lambda,\xi)$ in~\eqref{eq:subgrad1} and $L'_z(x,y,z,\lambda)=\E[\cL'_z(x,y,z,\lambda,\xi)]$, we have
    \begin{align}
        &\left(\cL'_z(x,y,z,\lambda,\xi)-L'_z(x,y,z,\lambda)\right)^2=\left\{-\tfrac{c}{z^2}(F(x,\xi)-y)_+^2+\tfrac{c}{z^2}\E\left[F(x,\xi)-y)_+^2\right]\right\}^2 \nn\\
        &\leq\tfrac{2c^2}{z^4}\left\{(F(x,\xi)-y)_+^4+\E^2\left[(F(x,\xi)-y)_+^2\right]\right\} \nn\\
        &\leq\tfrac{2c^2}{z^4}\left\{8(F(x,\xi)-f(x))_+^4+8(f(x)-y)^4+\left\{2\E\left[(F(x,\xi)-f(x))_+^2\right]+2(f(x)-y)^2\right\}^2\right\} \nn\\
        &\leq\tfrac{2c^2}{z^4}\left\{8(F(x,\xi)-f(x))_+^4+128(L_fD_X+1)^4+\left[2\beta^2+8(L_fD_X+1)^2\right]^2\right\} \nn\\
        &\leq\tfrac{16c^2}{z^4}\left[(F(x,\xi)-f(x))_+^4+32(L_fD_X+1)^4+\beta^4\right],
        \label{eq:eq1lem4'algo}
    \end{align}
    where the first, second and fourth inequalities follow from Young's inequality, the second inequality also applies $(F(x,\xi)-y)_+\leq (F(x,\xi)-f(x))_++|f(x)-y|$, and the third inequality utilizes~\eqref{assum:fbddvar},~\eqref{eq:Ycontainy} and~\eqref{eq:ybound}. Therefore, we have
    \begin{align*}
        &\E\left[\exp\left\{\left(\cL'_z(x,y,z,\lambda,\xi)-L'_z(x,y,z,\lambda)\right)^2\Big/\Tilde{\sigma}_z^4\right\}\right] \\
        &\leq\E\left[\exp\left\{\tfrac{16c^2}{z^4}(F(x,\xi)-f(x))_+^4\big/\Tilde{\sigma}_z^4\right\}\times\exp\left\{\tfrac{16c^2}{z^4}\left[32(L_fD_X+1)^4+\beta^4\right]\Big/\Tilde{\sigma}_z^4\right\}\right] \\
        &\leq\E^{\tfrac{1}{2}}\left[\exp\left\{\tfrac{32c^2}{z^4}(F(x,\xi)-f(x))_+^4\big/\Tilde{\sigma}_z^4\right\}\right]\times\E^{\tfrac{1}{2}}\left[\exp\left\{\tfrac{32c^2}{z^4}\left[32(L_fD_X+1)^4+\beta^4\right]\Big/\Tilde{\sigma}_z^4\right\}\right] \\
        &\leq\E^{\tfrac{1}{2}}\left[\exp\left\{(F(x,\xi)-f(x))_+^4\big/\sigma_4^4\right\}\right]\times\E^{\tfrac{1}{2}}\left[\exp\{1\}\right]\leq\left[\left(\exp\{1\}\right)^{\tfrac{1}{2}}\right]^2=\exp\{1\},
    \end{align*}
    where the first inequality follows from~\eqref{eq:eq1lem4'algo}, the second inequality utilizes Cauchy--Schwarz inequality, the third inequality applies~\eqref{def:tildesigma}, and the fourth inequality follows from Assumption~\ref{assumption3}.
\end{proof}
\vgap

With the help of Lemma~\ref{lem:subgaussian} and Lemma~\ref{lem:zsubgaussian}, we are now able to provide the high probability guarantee for the relation $\psi^*\geq\psi(z)+\zeta_T(z^*-z)-\epsilon/2$, following a similar approach to that in Lemma~\ref{lem:subdiff}.

\begin{lemma}
    Under Assumption~\ref{assumption3}, and conditioning on~\eqref{eq:Ycontainy}, for any $\alpha\in(0,1)$, if we set
    \begin{align}
        T\geq\max\left\{4\left[9+5\ln\left(\tfrac{2}{\alpha}\right)\right]^2D_V^2\sigma_z^2,48\ln\left(\tfrac{2}{\alpha}\right)(z^*-z)^2\Tilde{\sigma}_z^4\right\}\epsilon^{-2} \quad \text{and} \quad \gamma_t=\tfrac{D_V}{\sigma_z\sqrt{T}}
        \label{Tgamma2}
    \end{align}
    in Algorithm~\ref{algo:SMD}, where $D_V,\sigma_z,\Tilde{\sigma}_z$ are defined in~\eqref{def:D},~\eqref{def:sigma} and~\eqref{def:tildesigma}, respectively, then
    \begin{align*}
        \psi^*\geq\psi(z)+\zeta_T(z^*-z)-\tfrac{\epsilon}{2}
    \end{align*}
    holds with probability at least $1-\alpha$.
    \label{lem:highprobsubgrad}
\end{lemma}
\begin{proof}
    The structure of this proof is similar as the proof of Lemma~\ref{lem:subdiff}. Below we only show the differences of how we provide the high probability guarantees for the events $A$ and $B$ defined in~\eqref{def:eventA} and~\eqref{def:eventB} respectively, i.e., $P(A)\geq 1-\alpha/2$ and $\P(B\mid A)\geq 1-\alpha/2$.
    
    Let us first consider the event A defined in~\eqref{def:eventA}, i.e.,
    \begin{align*}
        A=\left\{\max_{u\in U}\left(\tsum_{t=0}^{T-1}\gamma_t\right)^{-1}\left(\tsum_{t=0}^{T-1}\gamma_t\Bar{l}'_z(u_t)^{\top}(u_t-u)\right)\leq\tfrac{\epsilon}{4}\right\}.
    \end{align*}
    To show $\P(A)\geq 1-\alpha/2$, we first take $\eta=1+\ln(2/\alpha)$ and define the following event
    \begin{align}
        A':=\left\{\tsum_{t=0}^{T-1}\gamma_t^2\left\|l_z'(u_t,\xi)\right\|_*^2\leq\eta\sigma_z^2\tsum_{t=0}^{T-1}\gamma_t^2\right\},
        \label{def:eventA'}
    \end{align}
    and we will show that $\P(A')\geq 1-\alpha/2$ and $\P(A\mid A')=1$ are satisfied. Let $\Delta_t=\gamma_t^2/(\tsum_{t=0}^{T-1}\gamma_t^2)$, then we have
    \begin{align}
        \E\left[\exp\left\{\tsum_{t=0}^{T-1}\Delta_t\left\|l_z'(u_t,\xi)\right\|_*^2\Big/\sigma_z^2\right\}\right]\leq\tsum_{t=0}^{T-1}\Delta_t\E\left[\exp\left\{\left\|l_z'(u_t,\xi)\right\|_*^2\Big/\sigma_z^2\right\}\right]\leq\tsum_{t=0}^{T-1}\Delta_t\exp\{1\}=\exp\{1\},
        \label{eq2lem5algo}
    \end{align}
    where the first inequality follows from Jensen's inequality, and the second inequality holds by Lemma~\ref{lem:subgaussian}. It then directly follows from~\eqref{eq2lem5algo} that
    \begin{align}
        \E\left[\exp\left\{\tsum_{t=0}^{T-1}\gamma_t^2\left\|l_z'(u_t,\xi)\right\|_*^2\Big/\left(\sigma_z^2\tsum_{t=0}^{T-1}\gamma_t^2\right)\right\}\right]\leq\exp\{1\}.
        \label{eq3lem5algo}
    \end{align}
    Now we consider the probability that the event $A'$ occurs, and we have
    \begin{align}
        &\P(A')\geq1-\exp\left\{\tsum_{t=0}^{T-1}\gamma_t^2\left\|l_z'(u_t,\xi)\right\|_*^2\Big/\left(\sigma_z^2\tsum_{t=0}^{T-1}\gamma_t^2\right)\right\}\Big/\exp\{\eta\}\geq 1-\exp\{1-\eta\}=1-\tfrac{\alpha}{2},
        \label{eq4lem5algo}
    \end{align}
    where the first inequality follows from Chernoff bound, and the second inequality utilizes~\eqref{eq3lem5algo}. According to the convergence analysis of the SMD method in~\cite[Lemma 3.1]{nemirovski2009robust} or~\cite[Lemma 4.6]{lan2020first}, the result~\eqref{eq1lem1algo} holds. Now we suppose that the event $A'$ occurs, and it follows from~\eqref{eq1lem1algo} that
    \begin{align}
        \left(\tsum_{t=0}^{T-1}\gamma_t\right)^{-1}\max_{u\in U}\tsum_{t=0}^{T-1}\gamma_t\Bar{l}'_z(u_t)^{\top}(u_t-u)\leq\left(\tsum_{t=0}^{T-1}\gamma_t\right)^{-1}\left(2D_V^2+\tfrac{5}{2}\eta\sigma_z^2\tsum_{t=0}^{T-1}\gamma_t^2\right)=\left(2+\tfrac{5}{2}\eta\right)\tfrac{D_V\sigma_z}{\sqrt{T}}\leq\tfrac{\epsilon}{4},
        \label{eq5lem5algo}
    \end{align}
    where the first inequality utilizes~\eqref{def:D},~\eqref{eq1lem1algo} and~\eqref{def:eventA'}, the equality follows from $\gamma_t=D_V/(\sigma_z\sqrt{T})$ in~\eqref{Tgamma2}, and the second inequality follows from $T\geq4[9+5\ln(2/\alpha)]^2D_V^2\sigma_z^2\epsilon^{-2}$ in~\eqref{Tgamma2} and $\eta=1+\ln(2/\alpha)$. In view of~\eqref{eq5lem5algo}, the event $A$ occurs when the event $A'$ occurs, thus $\P(A\mid A')=1$ holds, and by utilizing~\eqref{eq4lem5algo}, it follows that
    \begin{align}
        \P(A)\geq\P(A,A')=\P(A')\P(A\mid A')\geq 1-\tfrac{\alpha}{2}.
        \label{eq:probAnew}
    \end{align}
    
    Now we consider the event B. We first recall the definitions of the event $B$ and the $\sigma$-algebra $\cF_s$ in~\eqref{def:eventB} and~\eqref{def:calF}, respectively as
    \begin{align*}
        B=\left\{\zeta_T(z^*-z)\leq\E[\zeta_T\mid\cF_{T-1}](z^*-z)+\tfrac{\epsilon}{4}\right\} \qquad \text{and} \qquad \cF_s=\sigma(u_t\mid t=0,...,s).
    \end{align*}
    To show $\P(B\mid A)\geq 1-\alpha/2$, we apply~\cite[Lemma 4.1]{lan2020first}. Let $\Tilde{\Delta}_t=\gamma_t/(\tsum_{t=0}^{T-1}\gamma_t)$ and $\tau_t=\Tilde{\Delta}_t(\cL'_z(x_t,y_t,z,\lambda_t,\xi'_t)-L'_z(x_t,y_t,z,\lambda_t))(z^*-z)$ for $t=0,...,T-1$, it is easy to see that $\E[\tau_t\mid\cF_{T-1}]=0$ is satisfied. Moreover, according to Lemma~\ref{lem:zsubgaussian}, $\E[\exp\{\tau_t^2/[\Tilde{\Delta}_t^2(z^*-z)^2\Tilde{\sigma}_z^4]\}\mid\cF_{T-1}]\leq\exp\{1\}$ holds. Then
    \begin{align}
        &\P(B\mid\cF_{T-1})=\P\left(\tsum_{t=0}^{T-1}\tau_t\leq\tfrac{\epsilon}{4}\mid\cF_{T-1}\right)=1-\P\left(\tsum_{t=0}^{T-1}\tau_t>\tfrac{\epsilon}{4}\mid\cF_{T-1}\right) \nn\\
        &\geq1-\exp\left\{-\tfrac{\epsilon^2}{48(z^*-z)^2\Tilde{\sigma}_z^4\left(\tsum_{t=0}^{T-1}\Tilde{\Delta}_t^2\right)}\right\}\geq 1-\exp\left\{-\tfrac{\epsilon^2T}{48(z^*-z)^2\Tilde{\sigma}_z^4}\right\}\geq 1-\tfrac{\alpha}{2},
        \label{eq:probBnew}
    \end{align}
    where the first inequality directly follows from~\cite[Lemma 4.1]{lan2020first}, and the second and third inequalities are satisfied due to $\gamma_t=D_V/(\sigma_z\sqrt{T})$ and $T\geq 48\ln(2/\alpha)(z^*-z)^2\Tilde{\sigma}_z^4\epsilon^{-2}$ by~\eqref{Tgamma2}. It follows from~\eqref{eq:probBnew} that
    \begin{align}
        \P(B\mid A)\geq 1-\tfrac{\alpha}{2}.
        \label{eq:probBonAnew}
    \end{align}
    
    We finish the proof by combining~\eqref{eq:probAnew} and~\eqref{eq:probBonAnew}.
\end{proof}
\vgap

Utilizing Lemma~\ref{lem:highprobsubgrad}, we provide the proof of Theorem~\ref{thm:converge1highprob}.

\paragraph{Proof of Theorem~\ref{thm:converge1highprob}} This proof follows the same structure as that of Theorem~\ref{thm:converge1}, so we only highlight the key differences here.

First, under Assumption~\ref{assumption3}(b), taking $m\geq 3\ln(2/\alpha_0)\sigma_2^2$, the event~\eqref{eq:Ycontainy} occurs with probability at least $1-\alpha_0$, since the probability bound in~\eqref{eq:yscope2} with $\delta=1$ now becomes
\begin{align*}
    \P\left(\left|\hat{f}(x_0)-f(x_0)\right|\!>\!1\right)=\P\left(\tsum_{i=1}^m[F(x_0,\xi_i)\!-\!f(x_0)]\!>\!m\right)+\P\left(\tsum_{i=1}^m[F(x_0,\xi_i)\!-\!f(x_0)]\!<\!-m\right)\leq\tfrac{\alpha_0}{2}\!+\!\tfrac{\alpha_0}{2}=\alpha_0,
\end{align*}
where the inequality directly applies~\cite[Lemma 4.1]{lan2020first} to the two sequences of random variables $F(x_0,\xi_i)-f(x_0), i=1,...,m$ and $-[F(x_0,\xi_i)-f(x_0)], i=1,...,m$, respectively.

Second, recall that the inductive statement~\eqref{eq:thm1induction1} used in Theorem~\ref{thm:converge1} follows from the inductive step (see~\eqref{eq:induction2} in Lemma~\ref{lem:induction1}) that for any $k\in\{0,...,k_0-1\}$,
\begin{align*}
    \psi^*\geq\psi(z)+\zeta_T(z^*-z)-\tfrac{\epsilon}{2}
\end{align*}
holds with probability at least $1-\alpha/K$, which, under Assumption~\ref{assumption3} and in view of the parameter settings in~\eqref{Tgammakthmhighprob}, can be shown using Lemma~\ref{lem:highprobsubgrad}. This completes the proof.

\section*{Appendix B: High Probability Guarantees under Multiple Trajectories and Robust Distance Approximation}

Using multiple sample trajectories together with the robust distance approximation (RDA) procedure, we provide high probability guarantees for the events $A$ and $B$ defined in~\eqref{def:eventA} and~\eqref{def:eventB} in Lemma~\ref{lem:subdiff}, respectively. Consequently, in each outer iteration $k$ of Algorithm~\ref{algo:PB}, the relation $\psi^*\geq\psi(z_k)+\zeta_T(z^*-z_k)-\epsilon/2$ holds with high probability.

\begin{lemma}
    Under Assumptions~\ref{assumption1} and~\ref{assumption2}, and conditioning on~\eqref{eq:Ycontainy}, for any $\alpha\in(0,1)$, if we set
    \begin{align}
        T\geq\max\left\{\left[216(c_{2,*}c_{*,2}\!+\!1)D_{V,x_0}M_z\right]^2,\left(36|z^*-z|\Tilde{M}_z^2\right)^2\right\}\epsilon^{-2} \quad \text{and} \quad \gamma_t=\tfrac{D_{V,x_0}}{M_z\sqrt{T}}
        \label{Tgamma_highprobRDA}
    \end{align}
    in Algorithm~\ref{algo:SMD}, where $D_{V,x_0},M_z$ and $\Tilde{M}_z$ are defined in~\eqref{def:D},~\eqref{def:M^2} and~\eqref{def:tildeMz}, respectively, and the parameters $m_1,m_2$ and $m_3$ described in Section~\ref{sec:highprobRDA} are set to
    \begin{align}
        m_1\geq\log_2\left(\tfrac{4}{\alpha}\right), \quad m_2\geq 18\ln\left(\tfrac{4m_1}{\alpha}\right), \quad \text{and} \quad m_3\geq 18\ln\left(\tfrac{2}{\alpha}\right),
        \label{m123_lemRDA}
    \end{align}
    then
    \begin{align*}
        \psi^*\geq\psi(z)+\zeta_T(z^*-z)-\tfrac{\epsilon}{2}
    \end{align*}
    holds with probability at least $1-\alpha$.
    \label{lem:highprob_subgradRDA}
\end{lemma}
\begin{proof}
    For notational simplicity, we omit the outer-iteration subscript $k$ for the quantities introduced in Section~\ref{sec:highprobRDA}.
    
    Let us first consider the event A defined in~\eqref{def:eventA}. Specifically, we show that the event $\{S_t^{i^*}\leq\epsilon/4\}$ occurs with probability at least $1-\alpha/2$, where $S^i_t$ and $i^*$ are defined in~\eqref{def:RDAobject1} and~\eqref{def:i*RDA}, respectively. For the $i$-th call of Algorithm~\ref{algo:SMD}, define the events
    \begin{align*}
        A_{i,1}:=\left\{S^i_t\leq\tfrac{\epsilon}{12}\right\} \quad \text{and} \quad A_{i,2}:=\left\{\left|\hat{S}^{i,j_i^*}_t-S^i_t\right|\leq\tfrac{\epsilon}{12}\right\},
    \end{align*}
    where $\hat{S}^{i,j}_t$ is defined in~\eqref{def:RDAobject2} and $\hat{S}^{i,j_i^*}_t$ is the output of Algorithm~\ref{algo:RDA}. First, we have
    \begin{align}
        \P(A_{i,1})=\P\left(S^i_t\leq\tfrac{\epsilon}{12}\right)\geq 1-\tfrac{12\E\left[S^i_t\right]}{\epsilon}\geq \tfrac{3}{4},
        \label{prob1_lemRDA}
    \end{align}
    where the first inequality uses Markov's inequality, and the second inequality follows from Lemma~\ref{lem:SMDconverge} and the choice of parameters in~\eqref{Tgamma_highprobRDA}. To further provide a probability bound for $A_{i,2}$, for any $j\in\{1,2,...,m_2\}$, let us derive an upper bound for $\E[|\hat{S}^{i,j}_t-S^i_t|\mid\cF_{T-1}]$, where the $\sigma$-algebra $\cF_s=\sigma(u^i_t\mid t=0,...,s)$ (see~\eqref{def:calF}). To this end, we divide $|\hat{S}^{i,j}_t-S^i_t|$ into two parts, i.e.,
    \begin{align}
        \left|\hat{S}^{i,j}_t-S^i_t\right|&\leq\tfrac{1}{T}\max_{u\in U}\left|\tsum_{t=0}^{T-1}\left(l'_z\left(u^i_t,\xi^{i,j}_t\right)-\Bar{l}'_z\left(u^i_t\right)\right)^{\top}\left(u^i_t-u\right)\right| \nn\\
        &\leq\tfrac{1}{T}\left|\tsum_{t=0}^{T-1}\left(l'_z\left(u^i_t,\xi^{i,j}_t\right)-\Bar{l}'_z\left(u^i_t\right)\right)^{\top}\left(u^i_t-u_0\right)\right|+\tfrac{1}{T}\left|\tsum_{t=0}^{T-1}\left(l'_z\left(u^i_t,\xi^{i,j}_t\right)-\Bar{l}'_z\left(u^i_t\right)\right)^{\top}\left(u_0-u^{i,j*}\right)\right|,
        \label{eq1lemRDA}
    \end{align}
    where the first inequality applies $\gamma_t/(\tsum_{i=0}^{T-1}\gamma_t)=1/T$ from~\eqref{Tgamma_highprobRDA}, and $u^{i,j*}:=\argmax_{u\in U}|\tsum_{t=0}^{T-1}(l'_z(u^i_t,\xi^{i,j}_t)-\Bar{l}'_z(u^i_t))^{\top}(u^i_t-u)|$ is defined and used in the second inequality. Denote the two parts on the right hand side of~\eqref{eq1lemRDA} as $S_1$ and $S_2$, respectively. Then, we have
    \begin{align}
        \E\left[S_1\mid\cF_{T-1}\right]&\leq\E^{\tfrac{1}{2}}\left[S_1^2\mid\cF_{T-1}\right] = \left\{\tfrac{1}{T^2}\tsum_{t=0}^{T-1}\E\left[\left[\left(l'_z\left(u^i_t,\xi^{i,j}_t\right)-\Bar{l}'_z\left(u^i_t\right)\right)^{\top}\left(u^i_t-u_0\right)\right]^2\mid\cF_{T-1}\right]\right\}^{\tfrac{1}{2}} \nn\\
        &\leq\tfrac{1}{T}\left\{\tsum_{t=0}^{T-1}\E\left[\left\|u^i_t-u_0\right\|^2\left\|l'_z\left(u^i_t,\xi^{i,j}_t\right)-\Bar{l}'_z\left(u^i_t\right)\right\|_*^2\mid\cF_{T-1}\right]\right\}^{\tfrac{1}{2}} \nn\\
        &\leq\tfrac{\sqrt{2}D_{V,x_0}}{T}\left\{\tsum_{t=0}^{T-1}\E\left[\left\|l'_z\left(u^i_t,\xi^{i,j}_t\right)\right\|_*^2\mid\cF_{T-1}\right]\right\}^{\tfrac{1}{2}} \leq \tfrac{\sqrt{2}D_{V,x_0}M_z}{\sqrt{T}},
        \label{eq2lemRDA}
    \end{align}
    where the equality holds since $(l'_z(u^i_t,\xi^{i,j}_t)-\Bar{l}'_z(u^i_t))^{\top}(u^i_t-u_0), t=0,1,...,T-1$ are zero-mean and independent of each other when conditioning on $\cF_{T-1}$, the second inequality follows from the Cauchy--Schwarz inequality, the third inequality utilizes $D_{V,x_0}^2\geq\max_{u\in U}\|u-u_0\|^2/2$ in which $u_0=(x_0,y_0,\lambda_0)$, and the fourth inequality follows from~\eqref{eq:Mz}. Similarly,
    \begin{align}
        \E\left[S_2\mid\cF_{T-1}\right]&\leq\tfrac{1}{T}\E\left[\left\|u_0-u^{i,j*}\right\|\left\|\tsum_{t=0}^{T-1}\left(l'_z\left(u^i_t,\xi^{i,j}_t\right)-\Bar{l}'_z\left(u^i_t\right)\right)\right\|_*\mid\cF_{T-1}\right] \nn\\
        &\leq\tfrac{\sqrt{2}c_{2,*}D_{V,x_0}}{T}\E^{\tfrac{1}{2}}\left[\left\|\tsum_{t=0}^{T-1}\left(l'_z\left(u^i_t,\xi^{i,j}_t\right)-\Bar{l}'_z\left(u^i_t\right)\right)\right\|_2^2\mid\cF_{T-1}\right] \nn\\
        &=\tfrac{\sqrt{2}c_{2,*}D_{V,x_0}}{T}\left\{\tsum_{t=0}^{T-1}\E\left[\left\|l'_z\left(u^i_t,\xi^{i,j}_t\right)-\Bar{l}'_z\left(u^i_t\right)\right\|_2^2\mid\cF_{T-1}\right]\right\}^{\tfrac{1}{2}} \leq \tfrac{\sqrt{2}c_{2,*}c_{*,2}D_{V,x_0}M_z}{\sqrt{T}},
        \label{eq3lemRDA}
    \end{align}
    where the first inequality applies the Cauchy--Schwarz inequality, the second inequality uses $D_{V,x_0}^2\geq\|u_0-u^{i,j*}\|^2/2$, the equality holds since $l'_z(u^i_t,\xi^{i,j}_t)-\Bar{l}'_z(u^i_t), t=0,1,...,T-1$ are zero-mean and independent of each other when conditioning on $\cF_{T-1}$, and the third inequality follows from~\eqref{eq:Mz}. Combining~\eqref{eq2lemRDA} and~\eqref{eq3lemRDA}, we have
    \begin{align}
        \E\left[\left|\hat{S}^{i,j}_t-S^i_t\right|\mid\cF_{T-1}\right]=\E[S_1\mid\cF_{T-1}]+\E[S_2\mid\cF_{T-1}]\leq\tfrac{\sqrt{2}(c_{2,*}c_{*,2}+1)D_{V,x_0}M_z}{\sqrt{T}}\leq\tfrac{\epsilon}{108}.
        \label{eq4lemRDA}
    \end{align}
    Hence, by Markov's inequality and the parameter setting~\eqref{Tgamma_highprobRDA}, it follows from~\eqref{eq4lemRDA} that
    \begin{align}
        \P\left(\left|\hat{S}^{i,j}_t-S^i_t\right|\leq\tfrac{\epsilon}{36}\mid\cF_{T-1}\right) \geq 1-\tfrac{36\E[|\hat{S}^{i,j}_t-S^i_t|\mid\cF_{T-1}]}{\epsilon}\geq\tfrac{2}{3}.
        \label{eq5lemRDA}
    \end{align}
    Since $|\hat{S}^{i,j}_t-S^i_t|,j=1,2,...,m_2$ are independent of each other and~\eqref{eq5lemRDA} holds, then we directly apply the result in~\cite[Proposition 9]{hsu2016loss} and obtain
    \begin{align}
        \P(A_{i,2}\mid\cF_{T-1})=\P\left(\left|\hat{S}^{i,j_i^*}_t-S^i_t\right|\leq\tfrac{\epsilon}{12}\mid\cF_{T-1}\right) \geq 1-\mathrm{e}^{-\tfrac{m_2}{18}}\geq 1-\tfrac{\alpha}{4m_1},
        \label{prob2_lemRDA}
    \end{align}
    where the second inequality follows from $m_2\geq 18\ln(4m_1/\alpha)$ in~\eqref{m123_lemRDA}. Combining~\eqref{prob1_lemRDA} and~\eqref{prob2_lemRDA}, we have
    \begin{align}
        \P\left(\hat{S}^{i,j_i^*}_t\leq\tfrac{\epsilon}{6}\right)\geq \tfrac{3}{4}\left(1-\tfrac{\alpha}{4m_1}\right)\geq\tfrac{1}{2}.
        \label{prob3_lemRDA}
    \end{align}
    Since $\hat{S}^{i,j_i^*}_t,i=1,2,...,m_1$ are independent of each other, then in view of~\eqref{prob3_lemRDA} and the definition of $i^*$ in~\eqref{def:i*RDA}, we obtain
    \begin{align}
        \P\left(\hat{S}^{i^*,j_{i^*}^*}_t\leq\tfrac{\epsilon}{6}\right)=\P\left(\left\{\min_{i\in\{1,2,...,m_1\}}\hat{S}^{i,j_i^*}_t\right\}\leq\tfrac{\epsilon}{6}\right)\geq 1-\left(\tfrac{1}{2}\right)^{m_1}\geq 1-\tfrac{\alpha}{4},
        \label{prob4lemRDA}
    \end{align}
    where the second inequality follows from $m_1\geq\log_2(4/\alpha)$ in~\eqref{m123_lemRDA}. By taking the union bound over the event $\{\hat{S}^{i^*,j_{i^*}^*}_t\leq\epsilon/6\}$ in~\eqref{prob4lemRDA} and the events $A_{i,2},i=1,2,...,m_1$ in~\eqref{prob2_lemRDA}, we have
    \begin{align}
        \P\left(S^{i^*}_t\leq\tfrac{\epsilon}{4}\right)\geq 1-\tfrac{\alpha}{4}-\tfrac{\alpha}{4m_1}\times m_1=1-\tfrac{\alpha}{2}.
        \label{prob5lemRDA}
    \end{align}

    Now we discuss the event $B$ defined in~\eqref{def:eventB}. Let $\zeta^j$ be defined in~\eqref{def:zetakRDA} and $\zeta^{j^*}$ be the output of Algorithm~\ref{algo:RDA} (again omitting the subscript $k$). Below we show that $\{\zeta^{j^*}(z^*-z)\leq\E[\zeta^{j^*}|\cF_{T-1}](z^*-z)+\tfrac{\epsilon}{4}\mid\cF_{T-1}\}$ holds with probability at least $1-\alpha/2$. For any $j\in\{1,2,...,m_3\}$, set $\cF_s=\sigma(u^{i^*}_t\mid t=0,...,s)$ and $\Bar{\zeta}=\E[\zeta^j\mid\cF_{T-1}]$. Then
    \begin{align}
        \E\left[\left|\zeta^j-\Bar{\zeta}\right|\mid\cF_{T-1}\right] &\leq \mathrm{Var}^{\tfrac{1}{2}}\left(\zeta^j\mid\cF_{T-1}\right) = \left[\tfrac{1}{T^2}\tsum_{t=0}^{T-1}\mathrm{Var}\left(\cL'_z\left(x^{i^*}_t,y^{i^*}_t,z,\lambda^{i^*}_t,\xi'^{i^*,j}_t\right)\mid\cF_{T-1}\right)\right]^{\tfrac{1}{2}}\leq\tfrac{\Tilde{M}_z^2}{\sqrt{T}},
        \label{eq6lemRDA}
    \end{align}
    where the equality holds since $\cL'_z(x^{i^*}_t,y^{i^*}_t,z,\lambda^{i^*}_t,\xi'^{i^*,j}_t),t=0,1,...,T-1$ are independent of each other when conditioning on $\cF_{T-1}$, and the second inequality follows from~\eqref{eq2.5lem2algo}. By Markov's inequality and the parameter setting~\eqref{Tgamma_highprobRDA}, we obtain from~\eqref{eq6lemRDA} that
    \begin{align}
        \P\left(\left|\zeta^j-\Bar{\zeta}\right|\leq\tfrac{\epsilon}{12|z^*-z|}\mid\cF_{T-1}\right)\geq 1-\tfrac{12|z^*-z|\times\E[|\zeta^j-\Bar{\zeta}|\mid\cF_{T-1}]}{\epsilon}\geq\tfrac{2}{3}.
        \label{prob6lemRDA}
    \end{align}
    Since $|\zeta^j-\Bar{\zeta}|,j=1,2,...,m_3$ are independent of each other and~\eqref{prob6lemRDA} holds, then again, we directly apply the result in~\cite[Proposition 9]{hsu2016loss} and have
    \begin{align}
        \P\left(\left|\zeta^{j^*}-\Bar{\zeta}\right|\leq\tfrac{\epsilon}{4|z^*-z|}\mid\cF_{T-1}\right) \geq 1-\mathrm{e}^{-\tfrac{m_3}{18}}\geq 1-\tfrac{\alpha}{2},
        \label{prob7_lemRDA}
    \end{align}
    where the second inequality utilizes $m_3\geq18\ln(2/\alpha)$ in~\eqref{m123_lemRDA}. It then follows from~\eqref{prob7_lemRDA} that
    \begin{align}
        \P\left(\zeta^{j^*}(z^*-z)\leq\Bar{\zeta}(z^*-z)+\tfrac{\epsilon}{4}\mid\cF_{T-1}\right)\geq 1-\tfrac{\alpha}{2}.
        \label{prob8_lemRDA}
    \end{align}

    Combining~\eqref{prob5lemRDA} and~\eqref{prob8_lemRDA} completes the proof.
\end{proof}
\vgap

The proof of Theorem~\ref{thm:highprobRDA} directly follows from Lemma~\ref{lem:highprob_subgradRDA}.

\paragraph{Proof of Theorem~\ref{thm:highprobRDA}} This proof follows the same structure as that of Theorem~\ref{thm:converge1} and is similar to Theorem~\ref{thm:converge1highprob}; we therefore highlight only two differences here.

First, we compute $y_0$ in Algorithm~\ref{algo:PB} via Algorithm~\ref{algo:RDA} using the sample trajectories $\{\xi^j_i\}_{i=1}^m,j=1,2,...,m_0$. Let $\hat{f}^j(x_0)=\tsum_{i=1}^m F(x_0,\xi^j_i)/m$ and suppose that Algorithm~\ref{algo:RDA} outputs $\hat{f}^{j_0}(x_0)$. When $m=\lceil27\beta^2\rceil$, for any $j\in\{1,2,...,m_0\}$, Chebychev's inequality (or~\eqref{eq:yscope2}) yields $\P(|\hat{f}^j(x_0)-f(x_0)|\leq 3)\geq 2/3$. With $m_0=\lceil 18\ln(1/\alpha_0)\rceil$,~\cite[Proposition 9]{hsu2016loss} directly provides $\P(|\hat{f}^{j_0}(x_0)-f(x_0)|\leq 1)\geq 1-\mathrm{e}^{-m_0/18}=1-\alpha_0$. Hence,~\eqref{eq:Ycontainy} holds with probability at least $1-\alpha_0$.

Second, notice that the total number of oracle calls is $mm_0+[m_1(m_2+1)+m_3]K\Tilde{T}$. Indeed, computing $y_0$ needs $m$ samples in each of the $m_0$ sample trajectories. In each outer iteration of Algorithm~\ref{algo:PB}, we call Algorithm~\ref{algo:SMD} independently $m_1$ times, generating $(m_2+1)$ trajectories of $l'_z$ in each call and $m_3$ trajectories of $\cL'_z$ in the $i^*$-th call (see the definition of $i^*$ in~\eqref{def:i*RDA}).

\bibliographystyle{plain}
\bibliography{bibliography}

\end{document}